\documentclass[11pt,a4paper,twoside]{amsart}

\usepackage{amsmath,amssymb,amsfonts,amsthm,mathrsfs}
\usepackage{times,hyperref,color}
\usepackage{graphicx}

\newcommand{\specialcell}[2][c]{%
\begin{tabular}[#1]{@{}c@{}}#2\end{tabular}}

\newcommand{\omu}{\overline{\mu}}
\newcommand{\ow}{\overline{w}}
\newcommand{\oW}{\overline{W}}
\newcommand{\oxi}{\overline{\xi}}
\newcommand{\oz}{\overline{z}}
\newcommand{\oZ}{\overline{Z}}
\newcommand{\pp}[2]{\frac{\partial #1}{\partial #2}}
\newcommand{\vv}{\mathbf{v}}

\let\mathcal\mathscr

\newtheorem{The}{Theorem}[section]
\newtheorem{Theorem}{Theorem}[section]

\newtheorem{Lemma}[The]{Lemma}
\newtheorem{Corollary}[The]{Corollary}

\theoremstyle{definition}
\newtheorem{Definition}[The]{Definition}

\newtheorem{Remark}[The]{Remark}


\begin{document}


\title{
Convergent normal forms for
\\
five dimensional totally nondegenerate
\\
 CR manifolds in $\mathbb C^4$
}

\author{Masoud Sabzevari}
\address{Department of Mathematics,
 Shahrekord University, 88186-34141 Shahrekord, IRAN and School of
Mathematics, Institute for Research in Fundamental Sciences (IPM), P.
O. Box: 19395-5746, Tehran, IRAN}
\email{sabzevari@ipm.ir}

\date{\number\year-\number\month-\number\day}

\maketitle

\begin{abstract}
Applying the equivariant moving frames method, we construct convergent normal forms for real-analytic $5$-dimensional totally nondegenerate CR submanifolds of $\mathbb C^4$. These CR manifolds are divided into several biholomorphically inequivalent subclasses, each of which has its own complete normal form. Moreover it is shown that, biholomorphically, Beloshapka's cubic model is the unique member of this class with the maximum possible dimension seven of the corresponding algebra of infinitesimal CR automorphisms. Our results are also useful in the study of biholomorphic equivalence problem between CR manifolds, in question.
\end{abstract}

\pagestyle{headings} \markright{Convergent normal forms for five dimensional CR manifolds}

\section{Introduction}

One of the most elementary examples of moving frames is the well-known Frenet frame defined on curves. In the late nineteen century, Darboux generalized the construction of moving frames on surfaces in Euclidean spaces. Later on, in the early twentieth century, Cartan \cite{Cartan-1935} developed extensively the theory of moving frames by extending it to more general submanifolds of homogeneous spaces. In Cartan's thought, moving frames were powerful tools to study geometric features of submanifolds under the action of transformation (pseudo-)groups. In 1990's, Peter Olver and his collaborators endeavored to develop a modern and far-reaching reformulation of Cartan's classical approach \cite{Olver-Fels-99, Olver-Pohjanpelto-08}. This modernization, known by the {\it equivariant moving frames theory}, considerably expands on Cartan's construction and provides new algorithmic tools for computing sought differential invariants.

Let $M$ be an arbitrary manifold acted on by a Lie (pseudo-)group $\mathcal G$. Let ${\rm J}^n(M, p)$ be the $n$-th order jet space of $p$-dimensional submanifolds of $M$. An $n$-th order moving frame is defined as a $\mathcal G^{(n)}$-equivariant section of the trivial bundle ${\rm J}^n(M, p)\times\mathcal G^{(n)}\rightarrow {\rm J}^n(M, p)$, where $\mathcal G^{(n)}$ is the induced action of $\mathcal G$ on ${\rm J}^n(M, p)$. Existence of such a map heavily relies upon providing a so-called {\it cross-section} which roughly is a fixed jet point $z_0^{(n)}\in {\rm J}^n(M, p)$ where every $z^{(n)}$ moves to it by a {\it unique} transformation of $\mathcal G^{(n)}$.

The (equivariant) moving frames theory has exhibited its large potential of applications in many other fields ({\it see e.g.} the bibliographies of \cite{Olver-2005, Olver-2018} for a list of relevant works in the literature). Of particular interest, it provides a modern reformulation of Cartan's classical approach \cite{Olver-1995} to equivalence problems \cite{Valiquette-SIGMA}. In this formalism, the cross-section associated with a moving frame plays the role of {\it normalization} in Cartan's method. In contrast to the classical approach, existence of the essential tool {\it recurrence formula} ({\it see} $\S$\ref{sec-moving-frame} for pertinent definition) enables one to normalize the corresponding Maurer-Cartan forms, if possible, without requiring explicit expressions of the appearing torsion coefficients and just by applying some linear algebra techniques. This method is applied so far to several standard equivalence problems such as those of polynomials, differential equations, differential operators and  variational problems ({\it cf.} \cite{Arnnaldsson-17, Berchenko-Olver-2000, Valiquette-11, Valiquette-SIGMA}).

Another significant application of the moving frames theory is in construction of normal forms for real-analytic manifolds $M$ acted on by certain transformation (pseudo-)groups $\mathcal G$ \cite{Olver-2018}. Roughly speaking, a normal form is made by employing the group transformations to simplify, as much as possible, the coefficients of the Taylor series expansion of $M$. This process runs quite parallel to find practical cross-sections underlying the constructing  of the corresponding moving frames. Surprisingly, normal forms also encode  $\mathcal G$-equivalence problems. Indeed, two manifolds $M$ and $M'$ with the same underlying group action $\mathcal G$ are equivalent if and only if they have identical normal forms at the matching points.

{\it In CR geometry setting}, "biholomorphic equivalence" and "normal form" are two central issues of profound interest. Studying the former one was initiated by Cartan in \cite{Cartan-1932} ({\it see also} \cite{Olver-1995}) while the latter was developed significantly by Chern and Moser in \cite{Chern-Moser} ({\it see also} \cite{Jacobowitz}). Despite their considerable overlapping, each of these theories has also its own merits. For example, Cartan's approach not only enables one to solve certain equivalence problems but also provides the opportunity of reaching the structure equations of the corresponding coframes, what encodes the structure of the associated symmetry algebras. Such features are almost inaccessible in the Chern-Moser normal form theory. On the other hand, although the equivalence invariants find themselves as some coefficients in the corresponding normal forms but Cartan's approach is unable to unveil explicit structure of such normal forms.

Much surprisingly, {\it Olver's approach of equivariant moving frames provides one with the opportunity of reaching simultaneously almost all expected advantageous of both theories of equivalence problems and normal forms}. Indeed, it builds a concrete and striking bridge between these two theories. {\it Even more} and slightly in contrast to the Chern-Moser method in \cite{Chern-Moser}\,\,---\,\,which is the original impulse of most works in CR geometry relevant to normal forms\,\,---\,\,the method of moving frames is much {\it systematic} in the sense that one can apply it order by order in an algorithmic way consisting of {\it symbolic} computations. But, in spite of its applications and as a kind of weird, this modern approach has not gained yet its deserving attention in CR geometry, neither for solving arisen biholomorphic equivalence problems nor for constructing desired normal forms. That is while historically, biholomorphic equivalence between $3$-dimensional nondegenerate real hypersurfaces in $\mathbb C^2$ is one of the primary problems studied by Cartan in \cite{Cartan-1932}. Anyway, the wide breadth of applications of the equivariant moving frames theory is encouraging and motivating enough to apply it in the rich geometry of CR manifolds.

A crucial drawback to the Chern-Moser construction of normal forms concerns the {\it convergence}. Indeed, the heuristic and intelligent method of Chern-Moser in constructing normal forms for real-analytic nondegenerate hypersurfaces entailed them to consider first the effect of {\it formal} transformations, {\it instead of holomorphic ones}, on the corresponding Taylor series. Although they proved finally that the constructed formal normal forms are convergent \cite[Theorem 3.5]{Chern-Moser} but the key issue of convergence remained much challenging and occasionally unsolved in the subsequent works inspired by Chern-Moser ({\it see e.g.} \cite{BES, Ebenfelt-98, Kolar-05, Kolar-12, KKZ-17, KZ-15, KZ-19, Wong-82}). For instance, it turned out in \cite{Kolar-12} that the normal form introduced in \cite{Kolar-05} is divergence although it does not damage its important application in solving biholomorphic equivalence problem between finite type degenerate hypersurfaces in $\mathbb C^2$.

The equivariant moving frames approach for constructing normal forms involves infinitesimal counterpart of the associated (pseudo-)groups. This, in contrast to the Chern-Moser approach, does not stick one to work with formal transformations. Thus impressively, {\it the equivariant moving frame method results in constructing convergent normal forms, in essence}.

Totally nondegenerate CR manifolds are investigated widely by Valerii Beloshapka following his researches among the last decades of twentieth century. Let $M$ be a (abstract) CR manifold ({\it see} \cite{BER} for pertinent definitions) equipped with the CR distribution $T^cM\subset TM$. For each $k=1, \ldots, \mu$, let $\frak g^{-k}$ be the vector spaces spanned by iterated Lie brackets between exactly $k$ vector fields of $T^cM$. Then, $M$ is called {\it totally nondegenerate} of {\it length} $\mu$ \cite{Beloshapka2004, MZ} if: (a) $T^cM$ is a regular distribution; (b) $TM$ can be generated by $\frak g^{-1}+\cdots+\frak g^{-\mu}$ and (c) the truncation:
\[
\big(\frak g^{-1}+\cdots+\frak g^{-(\mu-1)}\big)/\frak g^{-\mu}
\]
forms a depth $\mu-1$ free Lie algebra generated by $\frak g^{-1}=T^cM$\,\,---\,\,the latter is the {\it totall nondegeneracy condition}. Beloshapka in \cite{Beloshapka2004} also devised a certain machinery to construct in part defining equations of these CR manifolds. He also extended  Poincar\'{e}'s and Chern-Moser's strategy of introducing {\it model hypersurfaces} to higher codimensions. These models enjoy several interesting features. For example, in each fixed CR dimension $n$ and codimension $k$, dimension of their associated Lie algebras of infinitesimal CR automorphisms is maximum. Moreover, as is proved in \cite{MZ}, if $k>n^2$ then all origin preserving CR automorphisms of these models are linear ({\it see} \cite[Theorem 14]{Beloshapka2004} for further {\it nice} properties).
Let us confine ourselves to discuss in CR dimension $n=1$, a case of which has received more attention in the literature. Totally nondegenerate CR manifolds of codimension $k=1$ are actually real hypersurfaces in $\mathbb C^2$ investigated specifically by Chern and Moser in \cite{Chern-Moser}. In codimension $k=2$, namely for $4$-dimensional totally nondegenerate CR submanifolds of $\mathbb C^3$,  Beloshapka, Ezhov and Schmalz constructed in \cite{BES} the desired normal form. The main approach applied in this work is that of Chern-Moser and the issue of convergence is not studied there.

As the subsequent class in CR dimension one, our goal in this paper is to construct convergent normal forms for real-analytic $5$-dimensional totally nondegenerate submanifolds $M^5\subset\mathbb C^4$ in codimension $k=3$.  For this purpose, we apply the equivariant moving frames techniques as proposed in \cite{Olver-2018}. Biholomorphic equivalences between the elements of this class is investigated in \cite{5-cubic} where the desired invariants of this problem are found. It is discovered \cite[Proposition 2.1]{5-cubic} that, after certain elementary normalizations, each CR manifold $M^5$ can be represented in local coordinates $z, w^1, w^2, w^3$ of $\mathbb C^4$ with $w^j:=u_j+ i v^j$ by:
\begin{equation}
\label{normalization-5cubic}
\aligned
v^1&:=\Phi^1(z,\oz,u)=z\oz+{\rm O}(3),
\\
v^2&:=\Phi^1(z,\oz,u)=z^2\oz+z\oz^2+{\rm O}(4),
\\
v^3&:=\Phi^1(z,\oz,u)=i(z^2\oz-z\oz^2)+{\rm O}(4),
\endaligned
\end{equation}
where, after assigning the weights $[z]=[\oz]=1$, $[u_1]=2$ and $[u_2]=[u_3]=3$ to the variables $z, \oz, u_j, j=1, 2, 3$, then ${\rm O}(t)$ denotes the sum taken over the monomials of weights $\geq t$ in terms of these variables.

Our results gain in more interest if we realize that there are only two types of $5$-dimensional Levi nondegenerate generic CR manifolds, namely those of CR dimension $n=1$ and codimension $k=3$ in $\mathbb C^4$ and those of CR dimension $n=2$ and codimension $k=1$ in $\mathbb C^3$. The construction of normal forms for the latter class is almost concluded specifically in the works \cite{Ebenfelt-98, Loboda-01, Loboda-03} of Ebenfelt and Loboda. But still, there is no considerable work concerning the normal forms of the former class. In \cite{MSP}, this class is divided into two distinct subclasses ${\sf III}_1$ and ${\sf III}_2$. In this terminology, what we aim to consider here is the first class ${\sf III}_1$. Thus, our work can be regarded as the first investigation of normal forms for $5$-dimensional nondegenerate CR manifolds in $\mathbb C^4$.

The outline of this paper is as follows. In the next Section \ref{sec-moving-frame}, we present a brief description of equivariant moving frames for Lie pseudo-group actions. We also present some fundamental formulas and results arisen in this theory. In Section \ref{Sec-function}, we provide requisite materials for constructing the desired equivariant moving frame and normal form associated with the action of biholomorphic pseudo-group $\mathcal G$ on totally nondegenerate CR manifolds $M^5$, in question. In particular, we produce infinitesimal version of this pseudo-group which is necessary for launching the construction in Section \ref{section-NF}. Our main tool in this construction is the powerful recurrence formula. In Section \ref{section-NF}, we make a fourth order cross-section which provides us a {\it partial} normal form for totally nondegenerate CR manifolds $M^5$. We realize that, completing this achieved normal form shall be addressed in higher orders through various branches relying on vanishing and non-vanishing of some of its specific coefficients. This situation is analogues to that of the normal form (3.18) constructed by Chern-Moser in part (d), page 246 of \cite{Chern-Moser} for real hypersurfaces in $\mathbb C^2$, where its complete form relies upon vanishing and non-vanishing of the coefficient $c_{42}$. We encounter at the end of Section \ref{section-NF} three distinct Branches $1$, $2$ and $3$. These are the first but not the last branches which emerge along the way and it turns out soon that several further subbranches are await for us in the next orders. We display the appearing branches via the diagram below.

Size of the computations in orders $\geq 5$ groves explosively and we perform them with the aid of {\sc Maple} program. In the next two sections \ref{Sec-Branch-1} and \ref{Sec-Branch-2} we construct complete convergent normal forms for CR manifolds belonging to Branches $1$ and $2$. These normal forms are achieved in order five ({\it cf.} Theorems \ref{th-branch-1} and \ref{th-branch-2}). In Section \ref{Sec-Branch-3}, where we aim to construct the sought normal form in Branch $3$ we encounter, unpleasantly, several new branches and subbranches as are visible in the diagram. We defer the study of Branch 3-2 to another investigation, but a complete list of convergent normal forms in all other branches 3-1, 3-3-1-a, 3-3-1-b, 3-3-2 and 3-3-3 is provided. The normal form of Branch 3-1 is achieved in order five ({\it cf.} Theorem \ref{th-branch-3-order-5}) while for the next branches, it is found in orders six and seven ({\it cf.} Theorems \ref{th-branch-3-3-1}, \ref{th-branch-3-3-2-a} and \ref{th-branch-3-3-2-b}).

\begin{figure}[h]
\includegraphics[height=3.85cm, width=7 cm]{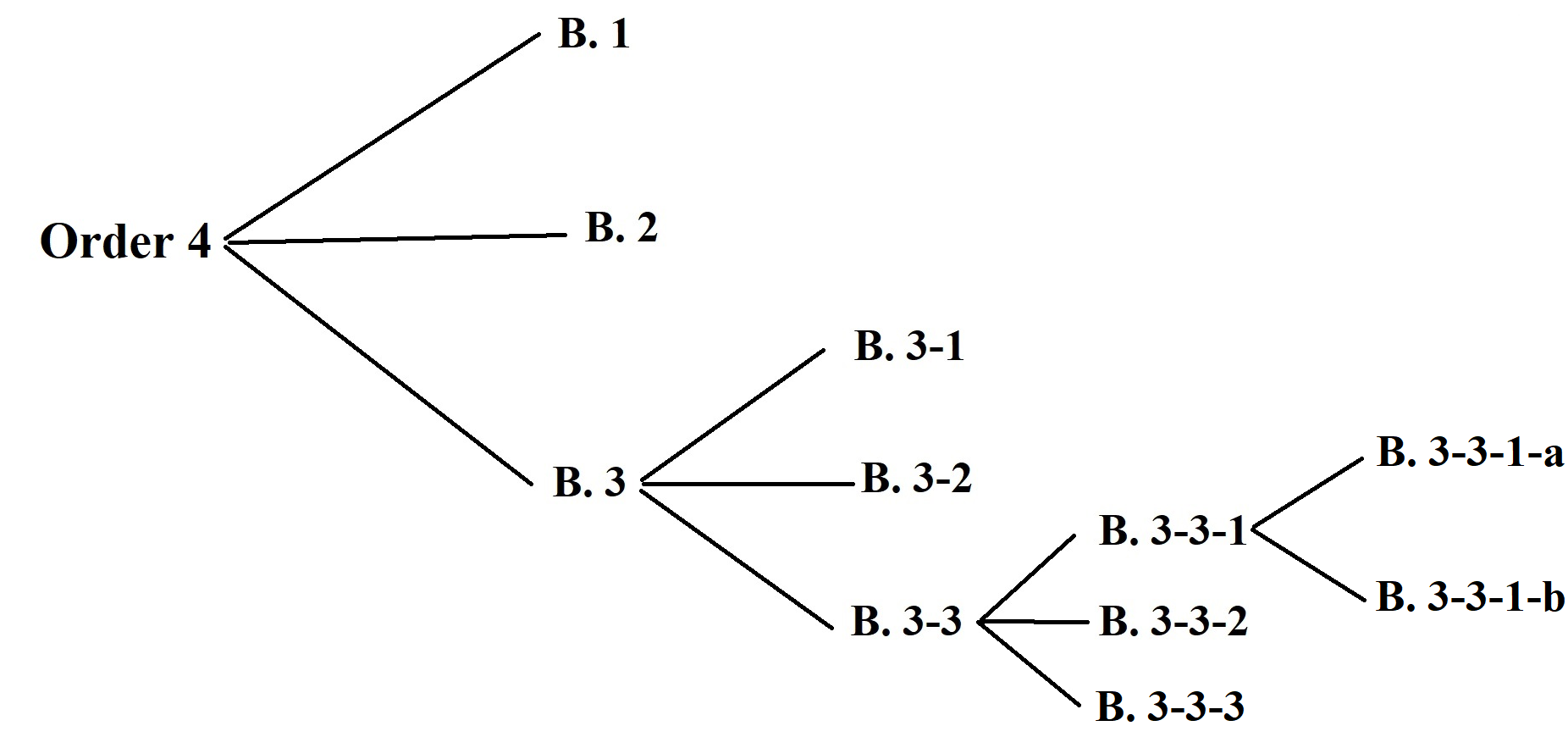}
\caption{The appearing branches}
\end{figure}

Let us collect together and exhibit concisely the main results of this paper.

\begin{Theorem}
Each $5$-dimensional real-analytic totally nondegenerate CR manifold $M^5\subset\mathbb C^4$ can be transformed holomorphically into the {\sl partial} normal form:
\begin{equation*}
\aligned
v^1&=z\oz+ \sum_{j+k+\sharp\ell\geq 5} \frac{1}{j!\,k!\,\ell!}\,V^1_{Z^j\oZ^k U^\ell} z^j \oz^k u^\ell,
\\
v^2&=\frac{1}{2}\,(z^2\oz+z\oz^2)+\sum_{j+k+\sharp\ell\geq 5} \frac{1}{j!\, k!\, \ell!}\,V^2_{Z^j\oZ^k U^\ell} z^j \oz^k u^\ell,
\\
v^3&=-\frac{i}{2}\,(z^2\oz-z\oz^2)+\frac{1}{6}\,V^3_{Z^3\oZ}\,z^3\oz+\frac{1}{6}\,V^3_{Z\oZ^3}\,z\oz^3+\frac{1}{2}\,V^3_{Z^2\oZ U_1} z^2\oz u_1+\frac{1}{2}\,V^3_{Z\oZ^2 U_1} z\oz^2 u_1
\\
& \ \ \ \ \ \ \ +\sum_{j+k+\sharp\ell\geq 5} \frac{1}{j! \, k! \, \ell!}\, V^3_{Z^j\oZ^k U^\ell} z^j \oz^k u^\ell,
\endaligned
\end{equation*}
where for $\ell=(l_1, l_2, l_3)\in\mathbb N^3$ we define $\ell!:=l_1!\, l_2!\, l_3!$ and $u^\ell:=u_1^{l_1} u_2^{l_2} u_3^{l_3}$.  Taking in account the conjugation relations $\overline{V^\bullet_{Z^j\oZ^k U^\ell}}=V^\bullet_{Z^k\oZ^j U^\ell}$, these coefficients enjoy the cross-section normalizations:
\begin{equation*}
\aligned
0&=V^1_{Z^j U^\ell}=V^1_{Z\oZ U^\ell}=V^1_{Z^2\oZ U^\ell}=V^1_{Z^2\oZ^2 U_2^j}=V^1_{Z^3\oZ U_2^j U_3^l},
\\
0&=V^2_{Z^j U^\ell}=V^2_{Z^j\oZ U^\ell}=V^2_{Z^2\oZ^2 U_2^j U_3^l},
\\
0&=V^3_{Z^j U^\ell}=V^3_{Z\oZ U^\ell}=V^3_{Z^2\oZ U_2^j U_3^l}=V^3_{Z^2\oZ^2 U_3^l},
\endaligned
\end{equation*}
for $j, l\in\mathbb N$. Completing this partial normal form relies upon vanishing and non-vanishing of some of its specific coefficient functions which causes splitting the process into several branches. The results are displayed concisely in the following table (cf. \cite[Theorem 1.3]{5-cubic}):
\begin{center}
\begin{tabular}{|c|c|c|c|c|}
  \hline
  {\rm Branch} & {\rm Assumptions} & {\rm CNF} & {\rm dim.} & {\rm Equivalence Problem} \\
  \hline
  $1$ & $V^3_{Z^3\oZ}\neq 0$ & \eqref{normal-form-Branch-1} & $5$ & $\{e\}$-{\rm structure on the base manifold} $M^5$ \\
  \hline
  $2$ & $V^3_{Z^3\oZ}= 0$ {\rm and} $V^3_{Z^2\oZ U_1}\neq 0$ & \eqref{normal-form-branch-2} & $5$ & $\{e\}$-{\rm structure on the base manifold} $M^5$ \\
  \hline
  {\rm 3-1} & \specialcell{$V^3_{Z^3\oZ}= V^3_{Z^2\oZ U_1}= 0$, \ \  \ ${\bf A}\neq 0$  } & \eqref{normal-form-branch-3} & $5$ & $\{e\}$-{\rm structure on the base manifold} $M^5$ \\
  \hline
  {\rm 3-3-1} & \specialcell{ $\bf B$} & \eqref{normal-form-Branch-3-3-1} & $5$, $6$ & \specialcell{
$\{e\}$-{\rm structure either on $M^5$ or on a} \\ {\rm
$6$-dimensional prolonged space of it} } \\
  \hline
  {\rm 3-3-2} & \specialcell{${\bf C}$} & \eqref{normal-form-Branch-3-3-2-a} & $5$ & $\{e\}$-{\rm structure on the base manifold} $M^5$ \\
  \hline
  {\rm 3-3-3} & {\bf D} & \eqref{model-cubic-2} & $7$ & {\rm equivalence to the model} $M^5_{\sf c}$ \\
  \hline
\end{tabular}
\end{center}
\smallskip\noindent
where {\rm CNF} stands for the term "Complete Normal Form", where {\rm dim} denotes the dimension of the infinitesimal CR automorphism algebra $\frak{aut}_{CR}(M^5)$, where $\bf A$ is the collection of four lifted differential invariants visible in \eqref{A-3-1} and where ${\bf B}, {\bf C}, {\bf D}$ are the assumptions made respectively in \eqref{A-3-3-1}, \eqref{A-3-3-2-a} and \eqref{umbilical}. The origin-preserving holomorphic transformation which brings $M^5$ into its associated normal form is essentially unique except when $M^5$ belongs to one of the two branches {\rm 3-3-1-b} or {\rm 3-3-3}. In Branch {\rm 3-3-1-b}, it is unique up to the action of the $1$-dimensional isotropy group of the corresponding normal form and in Branch {\rm 3-3-3} it is unique up to the action of the $2$-dimensional isotropy group of the cubic model $M^5_{\sf c}$ at the origin.
\end{Theorem}

It also turns out that, biholomorphically, Beloshapka's cubic model $M^5_{\sf c}$ represented by the defining equations \eqref{model-cubic} is the unique totally nondegenerate CR manifold with the maximum possible dimension $7$ of the corresponding Lie algebra $\frak{aut}_{CR}(M^5_{\sf c})$ ({\it cf.} Corollary \ref{cor}).

\section{Equivariant moving frames for Lie pseudo-group actions}
\label{sec-moving-frame}

The modern reformulation of the theory of moving frames is established first in the case of finite dimensional Lie group actions by Fels and Olver \cite{Olver-Fels-99}. Later on, Olver and Pohjanpelto \cite{Olver-Pohjanpelto-05, Olver-Pohjanpelto-08} extended it to the case of infinite dimensional Lie pseudo-group actions. In  this section, we present a brief description of equivariant moving frames theory in the latter case, as we aim to apply it in the next sections.

We fix $M$ throughout this section as a $m$-dimensional smooth manifold with local coordinates $x=(x^1, \cdots, x^m)$ and let $\mathcal D(M)$ be the pseudo-group of local diffeomorphisms $\varphi: M\longrightarrow M$. We denote by $\mathcal D^{(n)}(M)$ the bundle of $n$-th order jets of $\mathcal D(M)$. We also denote by the capital letters $X:=(X^1, \ldots, X^m)$ the {\it target coordinates} of $M$ under diffeomorphisms $\varphi\in\mathcal D(M)$, {\it i.e.} $X:=\varphi(x)$. Assume that $\mathcal G\subset\mathcal D(M)$ is a Lie psudo-subgroup acting on $M$. Naturally, $\mathcal G^{(n)}\subset\mathcal D^{(n)}(M)$ denotes the collection of all $n$-th order jets of diffeomorphisms belonging to $\mathcal G$.  As is known in local coordinates ({\it cf. \cite{Olver-2005}}) and for each integer $n$, $\mathcal G^{(n)}$ is identified as the solution of some system of partial differential {\it determining equations}:
\begin{equation}
\label{det-eq-pseu-group}
F^{(n)}(x, X^{(n)})=0,
\end{equation}
where the components $X_J$ of the $n$-th order jet coordinates $X^{(n)}$, for $\sharp J\leq n$, represent the partial derivatives $\partial^J\varphi/\partial x^J$ for $\varphi\in\mathcal G$.

Every Lie psedo-group $\mathcal G\subset\mathcal D(M)$ can be characterized by its corresponding algebra $\frak g$ of locally defined vector fields on $M$ whose flows belong to $\mathcal G$. Such vector fields are called the {\it infinitesimal generators} of $\mathcal G$. Set up the {\it lift} of an arbitrary vector field:
\[
{\bf v}:=\sum_{i=1}^m\,\xi^i(x)\,\frac{\partial}{\partial x^i}
\]
of $M$ as the vector field:
\[
{\bf V}:=\sum_{i=1}^m\sum_{\sharp J\geq 0}\,\mathbb D^J_x\xi^i(X)\,\frac{\partial}{\partial x^i_J},
\]
defined on the diffeomorphism jet bundle $\mathcal D^{(\infty)}(M)$ where:
\[
\mathbb D_{x^i}:= \frac{\partial}{\partial x^i}+\sum_{j=1}^m\sum_{\sharp J\geq 0}\,X^j_{J,i} \frac{\partial}{\partial X^j_J} \qquad i=1, \cdots, m
\]
are total derivative operators and for each multi-index $J$ with $\sharp J=k$, we set $\mathbb D^J_x:=\mathbb D_{x^{j_1}}\ldots \mathbb D_{x^{j_k}}$. Then, the above vector field $\bf v$ belongs to the infinitesimal Lie algebra $\frak g$ if its vector components $\xi:=(\xi^1, \ldots, \xi^m)$ satisfy the {\it infinitesimal determining equations}:
\begin{equation}
\label{infinitesimal-det-eq}
L^{(n)}(x, \xi^{(n)}):={\bf V}[F^{(n)}(x, X^{(n)})]|_{{\bf 1}^{(n)}}
\end{equation}
where $F^{(n)}(x, X^{(n)})$ is the $n$-th order determining equation \eqref{det-eq-pseu-group} of $\mathcal G$ and ${\bf 1}^{(n)}$ is the $n$-th order jet of the identity transformation ({\it see} \cite{Olver-1995, Olver-2005} for more details).

In agreement with the notations introduced above and for each $i=1, \ldots, m$, set:
\begin{equation}
\label{sigma-mu}
\aligned
\sigma^i&:=\sum_{j=1}^m\,X^i_j\,dx^j \qquad {\rm and} \qquad \mu^i:=dX^i-\sum_{j=1}^m\,X^i_j\,dx^j.
\endaligned
\end{equation}
In this terminology, the 1-forms $\sigma^i$ and $\mu^i$ are called the {\it invariant horizontal} and {\it zeroth order Maurer-Cartan} forms, respectively. Dual to the forms $\sigma^1, \ldots, \sigma^m$, we have differential operators:
\[
\mathbb D_{X^i}=\sum_{j=1}^m\,\mathbb W^j_i\, \mathbb D_{x^j} \qquad i=1, \ldots, m,
\]
where $\mathbb W^j_i:=\mathbb W^j_i(x, X^{(1)})$ is the $(j,i)$-th entry of the inverse matrix $(X^i_j)^{-1}$. For $i=1, \ldots, m$ and associated to each multi-index $J$ with $k=\sharp J$, we define the $k$-th order Maurer-Cartan form:
\[
\mu^i_J:=\mathbb D^j_X\mu^i.
\]
These right invariant 1-forms $\mu^{(\infty)}:=(\mu^1, \ldots, \mu^m, \ldots, \mu^i_J, \ldots)$ generate all Maurer-Cartan forms of the diffeomorphism pseudo-group $\mathcal D^{(\infty)}(M)$. This collection together with the horizontal forms $\sigma^1, \ldots, \sigma^m$, makes a right-invariant coframe on $\mathcal D^{(\infty)}(M)$ with the associated structure equations ({\it see} \cite{Olver-2005, Valiquette-11}):
\begin{gather*}
\label{st-eq-sigma-mu}
d\sigma^i = \sum_{j=1}^m \mu^i_j \wedge \sigma^j, \\
d\mu^i_J = \sum_{j = 1}^m \sigma^j \wedge \mu^i_{J,j} + \sum_{j = 1}^m \sum_{\stackrel{I+K=J}{\#K\geq 1}}\binom{J}{I} \mu^i_{I,j}\wedge \mu^j_K,
\end{gather*}
where $I+K$ is the component-wise addition of multi-indices in $\mathbb{N}^m$ and, when it is equal to $J$, we define $\binom{J}{I} = \frac{J!}{I!K!}$.
Up to order two, the first few structure equations are:

\begin{equation}
\label{struc-eq-MC}
\aligned
d\mu^i  =  \sum_{j=1}^m\,\sigma^j \wedge \mu^i_j, \qquad {\rm and} \qquad
d\mu^i_k = \sum_{j=1}^m\,\sigma^j \wedge \mu^i_{kj} + \sum_{j=1}^m\,\mu^i_j \wedge \mu^j_k.
\endaligned
\end{equation}

When we restrict the Maurer-Cartan forms $\mu^{(\infty)}$ to a pseudo-subgroup $\mathcal G$ of $\mathcal D^{\infty}(M)$, it expectedly appears among them some linear dependencies. By \cite[Theorem 6.1]{Olver-2005}  and in each order $n$, these linear dependencies can be detected through the linear system:
\[
L^{(n)}(X, \mu^{(n)})=0,
\]
where $L^{(n)}$ is the infinitesimal determining system \eqref{infinitesimal-det-eq} of $\mathcal G$. This provides one with the opportunity of finding a complete right-invariant coframe on the Lie pseudo-subgroup $\mathcal G^{(\infty)}$, as well.

Now, let us view this fragment of the theory from a broader perspective by considering $n$-th order jet bundle ${\rm J}^n(M, p), 0\leq n\leq\infty$, consisting of equivalence $p$-dimensional submanifolds $S$ of $M$ under the contact form equivalence. Splitting local coordinates $x^1, \ldots, x^m$ of $M$ into independent and dependent variables $x^1, \ldots, x^p$ and $u^1, \ldots, u^q$ with $p+q=m$, every submanifold $S$ can be represented locally as the graph of some $q$ defining equations $u^\alpha:=u^\alpha(x^1,\ldots,x^p), \alpha=1,\ldots, q$. Accordingly, the local coordinates of ${\rm J}^n(M, p)$ are represented by independent variables $x^1, \ldots, x^p$ and dependent jet variables $u^{(n)}:=u^\alpha_J$ for $\alpha=1, \ldots, q$ and for multi-indices $J$ with $\sharp J\leq n$.
Consider again $\mathcal G\subset\mathcal D(M)$ as a Lie pseudo-group acting on $M$ with the induced  prolonged action $\mathcal G^{(n)}$ on ${\rm J}^n(M, p)$ \cite{Olver-1995}. Let $\mathcal H^{(n)}$ be the groupoid obtained by pulling back the pseudo-group $\mathcal G^{(n)}$ via the canonical projection $\pi:{\rm J}^n(M, p)\rightarrow M$. Also let  $d_{\rm J}$ denotes the differentiation with respect to only jet coordinates. Substituting each jet coordinate $x^i, u^\alpha, u^\alpha_{x^i}, u^\alpha_{x^ix^j}, \ldots$ and each standard form $dx^i, du^\alpha, du^\alpha_{x^i}, du^\alpha_{x^ix^j}, \ldots$ with its corresponding $X^i, U^\alpha, U^\alpha_{X^i}, U^\alpha_{X^iX^j}, \ldots$ and $d_{\rm J}X^i, d_{\rm J}U^\alpha, d_{\rm J}U^\alpha_{X^i}, d_{\rm J}U^\alpha_{X^iX^j}, \ldots$, every arbitrary jet form $\Omega$ defined on ${\rm J}^{(n)}(M, p)$ {\it lifts} to a certain form $\lambda(\Omega)$ on $\mathcal H^{(n)}$. We call $\lambda$ the {\it lifting operator}.

Aiming for a concrete formula to compute the expression of standard lifted variables $U^\alpha_J$ in terms of the source jet variables $(x, u^{(\infty)})$,  consider first the total derivations on all group and jet variables:
\[
D_{x^j}:=\mathbb D_{x^j}+\sum_{\alpha=1}^q\,\bigg(u^\alpha_j \mathbb D_{u^\alpha}+\sum_{\sharp J\geq 1}\,u^{\alpha}_{J, j} \frac{\partial}{\partial u^\alpha_J}\bigg), \qquad j=1,\ldots, p.
\]
Corresponding to each $D_{x^i}$, let us set up the total differential operator:
\begin{equation}
\label{D-X-i}
D_{X^i}=\sum_{j=1}^p\,W^j_i D_{x^j}
\end{equation}
where $(W^j_i)=(D_{x^j} X^i)^{-1}$ is the inverse of the total Jacobian matrix. These differential operators serve to compute desired expression of the lifted coordinates as:
\[
U^\alpha_J=D^J_X \, U^\alpha:=D_{X^{j_1}}\cdots D_{X^{j_k}} U^\alpha.
\]
Moreover, the above total differential operators $D_{x^j}$ make the {\it lifted horizontal coframe} consisting of the forms:
\begin{equation}
\label{omega-i}
\omega^i:=\sum_{j=1}^p\, D_{x^j} X^i\, dx^j \qquad i=1, \ldots, p.
\end{equation}

Roughly, after providing an equivariant moving frame, this coframe plays the role of an invariant coframe on the submanifolds, in question. In the language of Cartan's classical theory of equivalence \cite{Olver-1995}, it has the role of {\it lifted coframe} associated with the $\mathcal G$-equivalence problem between the submanifolds $S$.

\subsection{Equivariant moving frame}

Now, we are ready to define the equivariant moving frame associated with the action of a Lie pseudo-group $\mathcal G$ on the manifold $M$.

\begin{Definition}
A {\sl moving frame} of order $n$ is a map $\rho^{(n)}:{\rm J}^n(M, p)\rightarrow\mathcal H^{(n)}$, where for each jet point $z^{(n)}=(x, u^{(n)})$, the source point of $\rho^{(n)}(z^{(n)})$ is $z^{(n)}$ and moreover it enjoys the {\it equivariant condition}:
\[
\rho^{(n)}(g^{(n)}\cdot z^{(n)})=\rho^{(n)}(z^{(n)})\cdot (g^{n})^{-1}
\]
for each $g^{(n)}\in\mathcal G^{(n)}|_z$ with the inverse $(g^{(n)})^{-1}$ where $z=\pi(z^{(n)})$.
\end{Definition}

Local existence of such a moving frame map is guaranteed by {\it regularity} and {\it freeness} of the action. By definition \cite{Olver-Pohjanpelto-05}, an action is regular when its corresponding orbits have the same dimension and moreover there are arbitrarily small neighborhoods whose intersection with each orbit is a connected subset thereof. The ($n$-th order) freeness means that at every point $z^{(n)}\in {\rm J}^n(M, p)$, the {\it isotropy subgroup}:
\[
\mathcal G^{(n)}_{z^{(n)}}:=\{g^{(n)}\in\mathcal G^{(n)}|_{z}: \ g^{(n)}\cdot z^{(n)}=z^{(n)}\}
\]
is trivial. However, even in the case that the underlying action is not free\,\,---\,\,as is customary for example in the case of actions underly many standard equivalence problems\,\,---\,\,it is still possible to construct a weaker version of moving frames, called by the {\it partial moving frame} ({\it see} \cite{Valiquette-SIGMA} for more details). Here, we do not aim to touch on this case.

Constructing a $n$-th order moving frame completely relies upon choosing a {\it cross-section} to the pseudo-group orbits which is a transverse submanifold of the complementary dimension. By the regularity and freeness of the action and for every point $p\in M$, there exists a {\it unique} transformation $g\in\mathcal G$ which maps $p$ to the cross-section. Let us describe the most simple, but practical type of cross-sections named by {\it coordinate} cross-section. Let $\mathcal G^{(n)}$ acts freely and regularly on the open subset $\mathcal V^n$ of ${\rm J}^n(M, p)$. The regularity of this action permits us to assume that the dimension of its orbits at each jet point $z^{(n)}=(x, u^{(n)})$ is constant, say $r_n$. This means that $\mathcal H^{(n)}$ can be coordinatized locally by the jet coordinates $x, u^{(n)}$ and $r_n$ {\it group parameters} $g^{(n)}=g_1, \ldots, g_{r_n}$. Now, one may pick some $r_n$ lifted jet coordinates $X, U^{(n)}$, renamed by $F_1, \ldots, F_{r_n}$, to define the {\it normalization equations}:
\begin{equation}
\label{cross-section}
F_1(x, u^{n}, g^{(n)})=c_1, \qquad \ldots \qquad F_{r_n}(x, u^{(n)}, g^{(n)})=c_{r_n},
\end{equation}
for some appropriate integers $c_1, \ldots, c_{r_{n}}$. Solving this system for the group parameters as unknowns, enables one to express each of them in terms of the jet coordinates, say $g_i=\gamma_i(x, u^{(n)})$ for $i=1, \ldots, r_n$. Accordingly, one defines the desired $n$-order moving frame by:
\[
\rho^{(n)}(x, u^{(n)}):=(x, u^{(n)}, \gamma(x, u^{(n)})).
\]
If $f_1, \ldots, f_{r_n}$ are jet coordinates with the lifts $F_1, \ldots, F_{r_n}$, then $f_1=c_1, \ldots, f_{r_n}=c_{r_n}$ is the cross-section associated to this moving frame.

\begin{Definition}
A differential function $I:{\rm J}^n(M, p)\rightarrow\mathbb R$ is a {\sl differential invariant} of order $n$ if it is unaffected by the action of $\mathcal G^{(n)}$, {\it i.e.}
\[
I(g^{(n)}\cdot (x, u^{(n)}))=I(x, u^{(n)})
\]
for each $g^{(n)}\in\mathcal G^{(n)}$  and $(x, u^{(n)})\in {\rm J}^n(M, p)$.
\end{Definition}

Once one succeeds to construct a complete moving frame $\rho^{(\infty)}:{\rm J}^\infty(M, p)\rightarrow\mathcal H^{(\infty)}$, it appears the crucial {\it invariantization operator} $\iota$ which converts every differential function or form to its invariantized counterpart. More precisely, for each arbitrary form $\Omega$ of ${\rm J}^\infty(M,p)$, the {\it invariantization} $\iota(\Omega)$ is defined as:
\[
\iota(\Omega):=(\rho^{(\infty)})^\ast(\lambda(\Omega))
\]
where $\lambda$ is the lifting operator, defined before. In particular, if we invariantize those jet coordinates $x^i, u^\alpha_J$ which do not appear in the corresponding cross-section then, one receives a {\it complete} system of differential invariants of order $\leq n$. Such differential invariants are termed {\it non-phantom (or basic) invariants}. Roughly speaking and in agreement with the above notations, they can be obtained by inserting $g_i=\gamma_i(x, u^{(n)})$, resulted by the cross-section, into the expressions of the lifted jet coordinates. {\it Even more}, if the action of $\mathcal G$ on $M$ is projectable meaning that $X^i_{u^\alpha}=0$ for each $i=1, \ldots, p$ and $\alpha=1, \ldots, q$ then, the differential operators $\mathcal D_i, i=1, \ldots, p$ dual to the invariantized 1-forms $\iota(\omega^i)$ ({\it see} \eqref{omega-i}) transform each order $n$ differential invariant to a differential invariant of higher order $n+1$.

\subsection{Recurrence formula}

The recurrence formula is one of the fundamental tools in the equivariant moving frames theory which relates differential invariants and invariant forms to their normalized counterparts defined by the corresponding cross-section. Moreover, it completely determines structure of the associated algebra of differential invariants. The point that brings this formula to much applications is that it is established using purely infinitesimal information, requiring only linear algebra and differentiation. In particular, it does not need explicit expressions of neither Maurer-Cartan forms nor lifted differential invariants.

Consider the infinitesimal generators:
\[
{\bf v}:=\sum_{i=1}^p\,\xi^i(x, u)\, \frac{\partial}{\partial x^i}+\sum_{\alpha=1}^q\,\phi^\alpha(x, u)\, \frac{\partial}{\partial u^\alpha}
\]
of the action of Lie pseudo-group $\mathcal G$ on $M$. As is known \cite{Olver-1995}, the prolonged action of $\mathcal G^{(\infty)}$ on ${\rm J}^\infty(M, p)$ is specified by the {\it prolonged infinitesimal generators}:
\[
{\bf v}^{(\infty)}=\sum_{i=1}^p\,\xi^i(x, u)\, \frac{\partial}{\partial x^i}+\sum_{\alpha=1}^q\sum_{k:=\sharp J\geq 0}\,\phi^{\alpha; J}(x, u^{(k)})\, \frac{\partial}{\partial u^\alpha_J},
\]
where the vector components $\phi^{\alpha; J}$ are defined recursively by the {\it prolongation formula}:
\[
\phi^{\alpha; J, j}=D_{x^j} \phi^{\alpha; J}-\sum_{i=1}^p\, (D_{x^j} \xi^i)\, u^\alpha_{J, i}.
\]

Artificially \cite{Olver-Pohjanpelto-05}, the lifts of the jets $\xi^i_J$ and $\phi^\alpha_J$  are defined to be the corresponding Maurer-Cartan forms ({\it cf. \eqref{sigma-mu}}):
\begin{equation}
\label{lift-comp}
\lambda(\xi^i_J):=\mu^i_J \qquad {\rm and} \qquad \lambda(\phi^\alpha_J):=\mu^\alpha_J.
\end{equation}

\medskip
\noindent
{\bf Recurrence Formula.} ({\it cf.} \cite[Theorem 25]{Olver-Pohjanpelto-08}) If $\Omega$ is a differential form defined on ${\rm J}^\infty(M, p)$ then:
\begin{equation*}
d\iota(\Omega)=\iota\big[d\Omega+{\bf v}^{(\infty)}(\Omega)\big]
\end{equation*}
 where ${\bf v}^{(\infty)}(\Omega)$ denotes the Lie derivative of $\Omega$ along ${\bf v}^{(\infty)}$.
\smallskip

This formula measures the distance between the invariantization of a jet form differentiation with the differential of its invariantization.  In particular, by taking $\Omega$ to be one of the standard jet coordinates $x^i$ or $u^\alpha_J$, then one receives modulo contact forms:
 \begin{equation}
 \label{Rec-Rel-General}
 \aligned
 dX^i&\equiv\omega^i+\mu^i,
 \\
 dU^\alpha_J&\equiv \sum_{j=1}^p\,U^\alpha_{J, j} \,\omega^j+\lambda(\phi^{\alpha; J}),
 \endaligned
 \end{equation}
 where, {\it by the customary abuse of notations}, the same symbols $U^\alpha_J$ and $\omega^j$ are used for their invariantizations $\iota(U^\alpha_J)$ and $\iota(\omega^i)$.

 \subsection{Normal form construction}

Let $S\subset M$, acted on by the pseudo-group $\mathcal G$, be a real-analytic $p$-dimensional submanifold represented in local coordinates as the graph of some defining vector equations $u:=u(x)$ for $x=x^1, \ldots, x^p$ and $u=u^1, \ldots, u^q$. At a reference point $z_0:=(x_0, u(x_0))=(x_0, u_0)$, we identify the Taylor series expansion:
\[
u(x)=\sum_{k=\sharp J\geq 0} \, \frac{1}{j_1!\,\cdots j_k!}\, u_{J}(x_0) \,(x-x_0)^J
\]
 of the defining equations with the restriction of jet coordinates to the point $z_0$, {\it i.e.}:
 \[
 u^{(\infty)}|_{z_0}=(x_0, u_0, u_J(x_0), \cdots).
 \]
Let $\rho$ be a complete moving frame associated with the prolonged action $\mathcal G^{(\infty)}$ on ${\rm J}^{(\infty)}(M, p)$ equipped with a cross-section $z^{(\infty)}_0$, where $\pi(z^{(\infty)}_0)=z_0$. By convenient abuse of notations, let us denote by the same $X, U_J$ the invariantization of the lifted coordinates. Then, through the already mentioned identification, the Taylor jet coordinate $u^{(\infty)}|_{z_0}$ is transformed to its corresponding invariantized jet $U^\alpha_J$ at the reference point $z^{(\infty)}_0$. It follows from the way of defining moving frame maps that such transformation is done by a {\it unique} element $g\in\mathcal G$. Furthermore, by the assumption $\pi(z_0^{(\infty)})=z_0$, made on the cross-section $z_0^{(\infty)}$, this transformation maps the reference point $z_0$ to itself. These new coefficients build the Taylor series expansion of the desired normal form of $S$ as:
\[
\sum_{k=\sharp J\geq 0} \, \frac{1}{j_1!\,\cdots j_k!}\, U_{J}(z^{(\infty)}_0)\, (x-x_0)^J.
\]

We refer the reader to  \cite{Olver-2018} for further relevant information and details.

\section{Preliminary materials}
\label{Sec-function}

Before applying the moving frame method, we need to provide some requisite materials. In particular, we need to find the infinitesimal counterpart of the Lie  pseudo-group $\mathcal G$ consisting of biholomorphic maps $\mathbb C^4\rightarrow\mathbb C^4$. This will help us to launch the crucial recurrence formula in the next sections which results in normalizing Maurer-Cartan forms, as much as possible.

By definition \cite{Krantz}, an invertible map $(z,w^1, w^2, w^3)\mapsto (Z, W^1, W^2, W^3)$ with $Z:=Z(z,w)$ and $W^j:=W^j(z,w), j=1,2,3$
  belongs to $\mathcal G$ whenever:
\begin{equation}
\label{eq: zw determining equations}
Z_{\oz} = Z_{\ow^k} = W^j_{\oz} = W^j_{\ow^k} = 0
\end{equation}
for each $j,k=1,2,3$. With the expansion\footnote{There is some technical reason for indicating a bit unnaturally the indices $j$ as $u_j$ and $v^j$.} $w^j:=u_j + iv^j$, let us expand the complex functions $W^j$ to their real and imaginary parts as:
\begin{equation*}
W^j=W^j(z,u,v) := U_j(z,\oz,u,v) + iV^j(z,\oz,u,v).
\end{equation*}
Assuming $\oZ = \overline{Z(z,u,v)}, \oW^j=\overline{W^j(z,u,v)}$ and thanks to the well-known relation $\pp{}{\ow} = \frac{1}{2} (\pp{}{u} + i \pp{}{v})$, equation \eqref{eq: zw determining equations} also implies that:
\[
\oZ_z=\oW^j_z=0,\qquad Z_{v^j}=iZ_{u_j},\qquad \oZ_{v^j}=-i\oZ_{u_j},\qquad W^k_{v^j}=iW^k_{u_j},\qquad \oW^k_{v^j}=-i\oW^k_{u_j},
\]
for each $j, k= 1, 2, 3$. Keeping the two equations $U_j=\frac{1}{2} (W^j+\overline W^j)$ and $V^j=\frac{i}{2} (\overline W^j- W^j)$ in mind, it is easy to verify that:
\[
\frac{\partial U_j}{\partial z} = iV^j_z, \qquad \frac{\partial U_j}{\partial \oz} = -i V^j_{\oz}, \qquad \frac{\partial U_k}{\partial v^j}=-V^k_{u_j}, \qquad V^k_{v^j}=\frac{\partial U_k}{\partial u_j}.
\]
Summing up, then an invertible map $(z, w)\mapsto (Z, W)$ belongs to the pseudo-group $\mathcal G$ whenever it satisfies the system of {\it determining equations} ({\it cf.} \eqref{det-eq-pseu-group}):
\begin{equation}
\label{eq: determining equations}
\aligned
Z_{\oz} = \oZ_z = 0,\qquad Z_{v^j} &= iZ_{u_j},\qquad \oZ_{v^j} = -i\oZ_{u_j},
\\
\frac{\partial U_k}{\partial z} = iV^k_z,\qquad \frac{\partial U_k}{\partial \oz}&=-iV^k_{\oz},\qquad \frac{\partial U_k}{\partial v^j}=-V^k_{u_j},\qquad V^k_{v^j} = \frac{\partial U_k}{\partial u_j}, \qquad j, k= 1, 2, 3.
\endaligned
\end{equation}

Let $\mathfrak{g}$ denotes the local Lie algebra of infinitesimal generators of $\mathcal G$, that is the collection of locally defined vector fields whose flows belong to $\mathcal G$ ({\it cf.} \cite{Olver-1995, Olver-Pohjanpelto-05}). We introduce the real vector field:
\begin{equation}\label{eq: v}
\vv = \xi(z,u,v)\pp{}{z} + \overline{\xi}(\oz,u,v)\pp{}{\oz} + \sum_{j=1}^3 \eta^j(z,\oz,u,v)\pp{}{u_j} + \sum_{j=1}^3\phi^j(z,\oz,u,v)\pp{}{v^j},
\end{equation}
where $\overline{\xi}(\oz,u,v):=\overline{\xi(z,u,v)}$, $\eta^j:=\frac{1}{2}(\zeta^j+\overline\zeta^j)$ and $\phi^j:=\frac{i}{2}(\overline\zeta^j-\zeta^j)$ for some holomorphic functions $\zeta^j(z,w)$, $j=1,2,3$. By linearizing the determining equations \eqref{eq: determining equations} at the identity transformation, one verifies that $\vv$ belongs to the Lie algebra\footnote{Since our interest is specifically to biholomorphisms which leave invariant our totally nondegenerate CR manifolds, then we may identify the Lie algebra $\frak g$ with its real part ({\it see} \cite[Propositions 12.4.25 and  12.4.26]{BER}).} $\mathfrak{g}$ if and only if its vector components satisfy the following \emph{infinitesimal determining equations} for $j, k=1, 2, 3$ ({\it cf.} \eqref{infinitesimal-det-eq}):
\begin{equation}\label{eq: infinitesimal determining equations}
\begin{gathered}
\xi_{\oz}=\oxi_z=0,\qquad \xi_{v^j} = i\,\xi_{u_j},\qquad \oxi_{v^j}=-i\,\oxi_{u_j},
 \\
\phi^k_z = -i\,\eta^k_z,\qquad \phi^k_{\oz} = i\,\eta^k_{\oz},\qquad \phi^k_{u_j} = - \eta^k_{v^j},\qquad \phi^k_{v^j}=\eta^k_{u_j}.
\end{gathered}
\end{equation}

By differentiating these equations, we also obtain the following relations among the second order jets of the vector field coefficients:
\begin{equation}
\label{eq: 2nd order jet relations}
\begin{gathered}
\xi_{\oz,a} = 0,\qquad \xi_{zv^j} = i\,\xi_{zu_j},\qquad \xi_{u_jv^k} = i\,\xi_{u_ju_k},\qquad \xi_{v^jv^k} = -\xi_{u_ju_k},
\\
\oxi_{z,a} = 0, \qquad  \oxi_{\oz v^j} = -i\,\oxi_{\oz u_j},\qquad \oxi_{u_jv^k} = -i\,\oxi_{u_ju_k},\qquad \oxi_{v^jv^k} = -\oxi_{u_ju_k},
\\
\eta^k_{z\oz}=0,  \qquad \eta^k_{v^jv^r}=-\eta^k_{u_ju_r}, \qquad \eta^k_{zv^j}=i\, \eta^k_{zu_j},\qquad
\eta^k_{\oz v^j} = -i\, \eta^k_{\oz u_j},
\\
\phi^k_{z,a} = -i\,\eta^k_{z,a},\qquad \phi^k_{\oz,a} = i\, \eta^k_{\oz,a},\qquad \phi^k_{u_j,a} = -\eta^k_{v^j,a},\qquad
\phi^k_{v^j,a} = \eta^k_{u_j,a}
\end{gathered}
\end{equation}
for $j, k, r=1, 2, 3$.

Now, let us turn our attention to the jet space ${\rm J}:={\rm J}^{(\infty)}(\mathbb R^8=\mathbb C^4, 5)$ of $5$-dimensional real submanifolds. This entails us to split the coordinates into $z, \oz, u_1, u_2, u_3$ as independent and $v^1, v^2, v^3$ as dependent variables. The appearing local jet coordinates are of the form $v^\alpha_J:= v^\alpha_{z^j \oz^k u_1^{r_1} u_2^{r_2} u_3^{r_3}}$ for $j, k, r_1, r_2, r_3\in\mathbb N$, for $\alpha=1, 2, 3$ and for multi-indices $J$ with $\sharp J\geq 0$. It is convenient to verify that the conjugate relation $\overline{v^\alpha_{z^j \oz^k v^r}}=v^\alpha_{z^k \oz^j v^r}$ holds also in higher order jets.

\subsection{Maurer-Cartan forms}

As explained in Section \ref{sec-moving-frame} and corresponding to the variables $z, \oz, u_j, v^j, j=1, 2, 3$, we have zeroth order Maurer-Cartan forms $\mu^z, \mu^{\oz}, \mu^{u_j}, \mu^{v^j}$. As a consequence of the equation \eqref{sigma-mu}, it is easy to check that $\mu^{\oz}=\overline{\mu^z}$ and also to check that $\mu^{u_j}$ and $\mu^{v^j}$ are real forms. For simplicity, let us make some changes in the notations as:
\[
\mu:=\mu^z, \qquad \overline\mu:=\mu^{\oz}, \qquad \alpha^j:=\mu^{u_j}, \qquad \gamma^j:=\mu^{v^j}, \qquad j=1, 2, 3.
\]

Since here we deal only with the holomorphic Lie pseudo-subgroup $\mathcal G$ of $\mathcal D(\mathbb C^4)$ and as stated in Section \ref{sec-moving-frame}, it is anticipated the appearing some linear dependencies among these Maurer-Cartan forms. Such dependencies can be detected by lifting ({\it cf.} \eqref{lift-comp}):
\[
\lambda(\xi_J)=\mu_J ,\qquad \lambda(\oxi_J)=\omu_J ,\qquad \lambda(\eta^j_J)=\alpha^j_J ,\qquad \lambda(\phi^j_J)=\gamma^j_J, \qquad \sharp J\geq 0
\]
of the vector components of $\bf v$ into the infinitesimal determining equations \eqref{eq: infinitesimal determining equations} and \eqref{eq: 2nd order jet relations}. This, provides the relations:
\begin{equation}
\label{MC-relations}
\aligned
\mu_{\oZ} &= \omu_Z = 0,\qquad \mu_{V^j}=i\,\mu_{U_j},\qquad \omu_{V^j} = -i\,\omu_{U_j},
\\
\gamma^k_Z &= -i\, \alpha^k_Z,\qquad \gamma^k_{\oZ}=i\,\alpha^k_{\oZ},\qquad
\gamma^k_{U_j} = - \alpha^k_{V^j},\qquad \gamma^k_{V^j} = \alpha^k_{U_j}.
\\
\mu_{\oZ,a}&= 0,\qquad \mu_{ZV^j} = i\,\mu_{ZU_j},\qquad \mu_{U_jV^r} = i\,\mu_{U_jU_r},\qquad \mu_{V^jV^r}=-\mu_{U_jU_r},
\\
\omu_{Z,a} &= 0,\qquad \omu_{\oZ V^j} = -i\,\omu_{\oZ U_j},\qquad \omu_{U_jV^r}=-i\,\omu_{U_jU_r},\qquad \omu_{V^jV^r} = -\omu_{U_jU_r},
\\
\alpha^k_{Z\oZ}&=0,\qquad \alpha^k_{V^jV^r} = -\alpha^k_{U_jU_r},\qquad \alpha^k_{ZV^j} = i\, \alpha^k_{ZU_j},\qquad \alpha^k_{\oZ V^j} = -i\, \alpha^k_{\oZ U_j},
\\
\gamma^k_{Z,a} &= -i\,\alpha^k_{Z,a},\qquad \gamma^k_{\oZ,a} = i\, \alpha^k_{\oZ,a},\qquad
\gamma^k_{U_j,a} = - \alpha^k_{V^j,a},\qquad \gamma^k_{V^j,a} = \alpha^k_{U_j,a}.
\endaligned
\end{equation}

Thus, a basis of Maurer--Cartan forms is provided by the invariant group forms:
\begin{equation}
\label{Basis-MC-original}
\mu_{Z^kU^\ell},\qquad \omu_{\oZ^k U^\ell},\qquad \alpha^j_{U^\ell V^r}, \qquad \alpha^j_{Z^k U^\ell},\qquad \alpha^j_{\oZ^k U^\ell},\qquad \gamma^k,
\end{equation}
for $j, r=1, 2, 3$, for $k\in\mathbb N$ and for $\ell=(\ell_1, \ell_2, \ell_3)\in\mathbb N^3$ where $U^\ell$ denotes $U_1^{\ell_1} U_2^{\ell_2} U_3^{\ell_3}$.

\subsection{Structure equations}

Henceforth and in order to present our expressions in a simpler fashion, we present them {\it occasionally} by Einstein's summation convention. Associated with the independent variables $z, \oz, u_1, u_2, u_3$ of each $5$-dimensional submanifold $M\subset\mathbb C^4$ represented as the graph of three defining equations $v^j=f^j(z, \oz, u_1, u_2, u_3), j=1, 2, 3$, we have the {\it lifted horizontal coframe} consisting of the five 1-forms $\omega^z, \omega^{\oz}, \omega^k:=\omega^{u_k}, k=1, 2, 3$ with the expressions ({\it cf. } \eqref{omega-i}):
\begin{equation*}
\aligned
\omega^z &= D_z Z\, dz + D_{\oz} Z\, d\oz + D_{u_j}Z\, du_j
\\
&=(Z_z + Z_{v^j} f^j_z)\,dz + Z_{v^j} f^j_{\oz}\, d\oz + (Z_{u_j} + Z_{v^k} f^k_{u_j})\, du_j
 \\
&=(Z_z + i\,Z_{u_j} f^j_z)\, dz + i\,Z_{u_j} f^j_{\oz}\, d\oz + (Z_{u_j} + i\,Z_{u_k} f^k_{u_j})\, du_j,
\\
\omega^{\oz} &= \overline{\omega^z},
\\
\omega^k &= D_z U_k\, dz + D_{\oz}U_k\, d\oz + D_{u_j} U_k\, du_j
\\
&= ({U_k}_z + {U_k}_{v^j} f^j_z)\, dz + ({U_k}_{\oz}+{U_k}_{v^j} f^j_{\oz})\, d\oz + ({U_k}_{u_j} + {U_k}_{v^r} f^r_{u_j})\, du_j.
\endaligned
\end{equation*}
 In matrix form, if we put the coefficients of $dz, d\oz, du_j, j=1, 2, 3$ of the above equations into a $5\times 5$ matrix $A$ so that:
\[
[\omega^z, \omega^{\oz}, \omega^1, \omega^2, \omega^3]^t=A\cdot [dz, d\oz, du_1, du_2, du_3]^t
\]
then, the {\it lifted total derivative operators} are given by ({\it cf.} \eqref{D-X-i}):
\begin{equation*}\label{eq: lifted derivative operators}
\begin{bmatrix}
D_Z \\ D_{\oZ} \\ D_{U_j}
\end{bmatrix}
= A^{-T}
\begin{bmatrix}
D_z \\ D_{\oz} \\ D_{u_j}
\end{bmatrix}.
\end{equation*}
These operators serve to compute the expressions of the lifted jet coordinates $V^\alpha_{Z^j \oZ^k U^\ell}=D_Z^j D_{\oZ}^k D_{U}^\ell V^\alpha$ for $\alpha=1, 2, 3$, though unfortunately, the computations are complicated in practice, specifically in higher orders. We also notice that our lifted jet coordinates also respect the conjugation, {\it i.e.} $\overline{V^\alpha_{Z^j \oZ^k U^\ell}}=V^\alpha_{Z^k \oZ^j U^\ell}$.

Modulo contact forms\footnote{Contact forms play no essential role in this study.}, the lifted horizontal forms $\omega^z, \omega^{\oz}, \omega^j, j=1, 2, 3$, match up to their corresponding invariant horizontal forms $\sigma^z, \sigma^{\oz}, \sigma^j$ introduced in \eqref{sigma-mu}. Therefore, according to  \eqref{st-eq-sigma-mu}, they admit the structure equations:
\begin{equation}
\label{eq: order zero structure equations}
\begin{aligned}
d\omega^z &= \mu_Z\wedge \omega^z + \mu_{\oZ}\wedge \omega^{\oz} + \mu_{U_j}\wedge \omega^j+ \mu_{V^j}\wedge \big(V^j_Z\,\omega^z+V^j_{\oZ}\,\omega^{\oz}+V^j_{U_k}\,\omega^k\big)
\\
&= \mu_Z\wedge \omega^z + \mu_{U_j}\wedge \omega^j+ i\,\mu_{U_j}\wedge (V^j_Z\,\omega^z+V^j_{\oZ}\,\omega^{\oz}+V^j_{U_k}\,\omega^k),
\\
d\omega^{\oz} &= \overline{d\omega^z},
\\
d\omega^k &= \alpha^k_Z \wedge \omega^z + \alpha^k_{\oZ}\wedge \omega^{\oz} + \alpha^k_{U_j}\wedge \omega^j + \alpha^k_{V^j} \wedge (V^j_Z\,\omega^z+V^j_{\oZ}\,\omega^{\oz}+V^j_{U_r}\,\omega^r), \qquad k=1,2,3.
\end{aligned}
\end{equation}

\subsection{Recurrence relations}

The prolongation of the infinitesimal generator \eqref{eq: v} of $\mathcal G$ is given by:
\begin{equation}
\label{v-infty}
\vv^{(\infty)} = \xi \frac{\partial}{\partial z} + \oxi \frac{\partial}{\partial \oz} + \sum_{j=1}^3 \eta^j \frac{\partial}{\partial u_j} + \sum_{\alpha=1}^3\sum_{k=\sharp J\geq 0} \phi^{\alpha;J}\frac{\partial}{\partial v^\alpha_J},
\end{equation}
where the vector components $\phi^{\alpha;J}$ are defined recursively by the prolongation formula:
\begin{equation}\label{prolong-formula}
\phi^{\alpha; J,x^j}=D_{x^j}\phi^{\alpha;J}-(D_{x^j}\xi) \,v^\alpha_{z, J}-(D_{x^j}\overline\xi) \,v^\alpha_{J, \oz}-\sum_{k=1}^3\,(D_{x^j}\eta^k)\, v^\alpha_{J, u_k},
\end{equation}
for $x^j=z, \oz, u_1, u_2, u_3$.
Accordingly, the recurrence relations in our case are represented as ({\it cf.} \eqref{Rec-Rel-General}):
\begin{equation}
\label{rec-formula}
\aligned
dZ&=\omega^z+\mu, \qquad dZ=\omega^{\oz}+\overline\mu,
\\
dU_j&=\omega^j+\alpha^j,
\\
dV^\alpha_J&=V^\alpha_{J, z}\, \omega^z+ V^\alpha_{J, \oz}\, \omega^{\oz}+\sum_{j=1}^3 \, V^\alpha_{J, u_j}\, \omega^j+\lambda(\phi^{\alpha; J}),
\endaligned
\end{equation}
for $j, \alpha=1, 2, 3$ and $\sharp J\geq 0$, where $\lambda(\phi^{\alpha; J})$ is the lift of the vector coefficient $\phi^{\alpha; J}$ ({\it cf.} \eqref{lift-comp}).

\section{Normal form construction}
\label{section-NF}

After providing preliminary requirements in the last section, now we are ready to launch the construction of convergent normal forms for 5-dimensional totally nondegenerate CR manifolds $M^5\subset\mathbb C^4$. As stated before, our strategy is to build an appropriate moving frame for the action of the biholomorphic Lie pseudo-group $\mathcal G$ on theses CR manifolds which runs quite parallel to construct desired normal forms  \cite{Olver-2018}. As is customary in moving frame theory \cite{Valiquette-11, Valiquette-SIGMA}, we do this task by normalizing {\it order by order} the appearing differential invariants. This strategy is based fundamentally on the above mentioned recurrence formula.

\subsection{Order Zero}

According to \eqref{rec-formula}, the recurrence formula in this order gives:
\[
dZ = \omega^z + \mu, \qquad
d\oZ = \omega^{\oz} + \omu,\qquad
dU_j= \omega^j + \alpha^j, \qquad
dV^j = V^j_Z\omega^z+V^j_{\oZ}\omega^{\oz}+V^j_{U_j}\omega^j + \gamma^k
\]
for $j=1, 2, 3$. By the discussion presented at the end of Section \ref{sec-moving-frame} and since we aim to construct our sough normal form at the origin, as the reference point, with some {\it origin-preserving} holomorphic transformation then, we choose the normalizations (order zero cross-section):
\[
Z = \oZ = U_j = V^j = 0
\]
which, according to the above recurrence relations, brings the order zero Maurer-Cartan forms as:
\begin{equation*}
\label{eq: order zero mc form normalizations}
\mu = -\omega^z,\qquad
\omu = -\omega^{\oz},\qquad
\alpha^k = -\omega^k,\qquad
\gamma^k =-V^k_Z\,\omega^z-V^k_{\oZ}\,\omega^{\oz}-V^k_{U_j}\,\omega^j .
\end{equation*}

\subsection{Order One}

The elementary normalization \eqref{normalization-5cubic} indicates that the coefficients of $z, \oz, u_j$ in the power series defining equations of $M^5\subset\mathbb C^4$ remain vanished under holomorphic transformations. Then, we may have the cross-section normalisation:
\[
V^k_Z=V^k_{\oZ}=V^k_{U_j}=0,
\]
for $j, k=1, 2, 3$. Consequently, the order one recurrence relations take the form (henceforth we do not present the recurrence relations of the conjugated lifted invariants due to the fact that $dV_{\overline J}=\overline{dV_J}$):
\begin{equation*}
\label{eq: order one recurrence relations}
\begin{aligned}
0=dV^k_Z &=V^k_{ZZ}\,\omega^z+V^k_{Z\oZ}\,\omega^{\oz}+ V^k_{ZU_j}\,\omega^j-i\,\alpha^k_Z,
\\
0=dV^k_{U_j} &= V^k_{ZU_j}\,\omega^z+V^k_{\oZ U_j}\,\omega^{\oz}+ V^k_{U_j U_r}\,\omega^r-\alpha^k_{V^j}.
\end{aligned}
\end{equation*}
This implies that:
\[
\aligned
\alpha^k_Z=-i\,V^k_{ZZ}\,\omega^z-i\,V^k_{Z\oZ}\,\omega^{\oz}-i\, V^k_{Z,j}\,\omega^j, &\qquad \alpha^k_{\oZ}=i\,V^k_{Z\oZ}\,\omega^{z}+i\,V^k_{\oZ \oZ}\,\omega^{\oz}+i\, V^k_{\oZ,j}\,\omega^j,
\\
\alpha^k_{V^j}=V^k_{ZU_j}\,\omega^z&+V^k_{\oZ U_j}\,\omega^{\oz}+ V^k_{U_j U_r}\,\omega^r,
\endaligned
\]
for $j, k=1,2,3$.

Among performing the above computations and for each $k=1, 2, 3$, we observed that the first order prolonged vector coefficient $\phi^{k;z}$ ({\it see} \eqref{v-infty}) includes the term $1\eta^k_z$. This, according to the prolongation formula \eqref{prolong-formula}, implies that for each integer $j\geq 0$ and each triple $\ell\in\mathbb N^3$, we can see the term $1\eta^k_{z^{j+1} u^\ell}$ in the higher order prolonged vector coefficient $\phi^{k;z^{j+1} u^\ell}$. Since the lift of this term is $\lambda(\eta^k_{z^{j+1} u^\ell})=\alpha^k_{Z^{j+1} U^\ell}$, then we may find through the recurrence relation of $dV^k_{Z^{j+1} U^\ell}$ a nonzero constant coefficient of the Maurer-Cartan form $\alpha^k_{Z^{j+1} U^\ell}$. Thus, by the normalization $V^k_{Z^{j+1} U^\ell}=0$, one may specify this Maurer-Cartan form. Likewise, the term $-1\eta^k_{v^j}$ is visible in the prolonged vector coefficient $\phi^{k,u_j}$ for $j=1, 2, 3$. Then, similar argument shows that one can specify each Maurer-Cartan form $\alpha^k_{U^\ell V^j}, \ell\in\mathbb N^3$ by applying the normalization $V^k_{U^\ell U_j}=0$. Thus, we have

\begin{Lemma}
\label{lem-observ-1}
For $k=1, 2, 3$ and for each integer $j\geq 0$ and triple $\ell\in\mathbb N^3$, it is always possible
\begin{itemize}
  \item[$1.$] to specify the Maurer-Cartan forms $\alpha^k_{Z^{j+1}U^\ell}$ by normalizing $V^k_{Z^{j+1}U^\ell}=0$.
  \item[$2.$] to specify the Maurer-Cartan forms $\alpha^k_{U^\ell V^j}$ by normalizing $V^k_{U^\ell U_j}=0$.
\end{itemize}
\end{Lemma}

\begin{Remark}
Normalizations mentioned in the above lemma are well-known in CR geometry. Indeed, they correspond to the fact that the pluriharmonic terms are always removable from the defining equations of generic CR manifolds ({\it see e.g} \cite[Lemma 3.2]{Chern-Moser}).
\end{Remark}

\subsection{Order Two}

Again, a glance at the elementary normalization \eqref{normalization-5cubic} indicates that in this order we are permitted to normalize:
\[
V^2_{Z\oZ}=V^3_{Z\oZ}=0, \qquad V^1_{Z\oZ}=1.
\]
By Lemma \ref{lem-observ-1}, these differential invariants are the only ones, we shall consider in this order. Taking this lemma into account,
we have the recurrence relations:
\[
\aligned
0=dV^1_{Z\oZ}&=V^1_{Z^2\oZ}\,\omega^z+V^1_{Z\oZ^2}\,\omega^{\oz}+V^1_{Z\oZ U_j}\,\omega^j+\alpha^1_{U_1}-(\mu_Z+\overline\mu_{\oZ}),
\\
0=dV^r_{Z\oZ}&=V^r_{Z^2\oZ}\,\omega^z+V^r_{Z\oZ^2}\,\omega^{\oz}+V^r_{Z\oZ U_j}\,\omega^j+\alpha^r_{U_1},
\endaligned
\]
for $r=1,2$. This implies that:
\begin{gather*}
\aligned
\alpha^1_{U_1}&=-V^1_{Z^2\oZ}\omega^z-V^1_{Z\oZ^2}\omega^{\oz}-V^1_{Z\oZ U_j}\omega^j+(\mu_Z+\overline\mu_{\oZ}),
\\
\alpha^r_{U_1}&=-V^r_{Z^2\oZ}\omega^z-V^r_{Z\oZ^2}\omega^{\oz}-V^r_{Z\oZ U_j}\omega^j.
\endaligned
\end{gather*}

Before concluding this order, we notice that in the expression of the prolonged vector coefficient $\phi^{k,z\oz}, k=1,2,3$, one observes the term $v^1_{z\oz}\eta^k_{u_1}$. Thus, by an argument similar to that of the proof of Lemma \ref{lem-observ-1}, one finds

\begin{Lemma}
\label{lem-observ-2}
For $k=1, 2, 3$ and for each triple $\ell\in\mathbb N^3$, it always is possible to specify the Maurer-Cartan forms $\alpha^k_{U_1U^\ell}$ by normalizing to zero lifted differential invariants $V^k_{Z\oZ U^\ell}$.
\end{Lemma}

\subsection{Order Three}

Keeping still in mind the elementary defining equations \eqref{normalization-5cubic}, here we normalize to zero all lifted invariants $V_J$ with $\sharp J=3$ except:
\[
V^2_{Z^2\oZ}=V^2_{Z\oZ^2}=1, \qquad V^3_{Z^2\oZ}=-i, \qquad V^3_{Z\oZ^2}=i.
\]
These normalizations together with those chosen in the previous order, are actually imposed by the total nondegeneracy of the CR manifolds $M^5$, in question.
Taking into accounts the general observations, found in the preceding orders, here we have the simplified recurrence relations:
\[
\aligned
0=dV^1_{Z^2\oZ}&=V^1_{Z^3\oZ}\,\omega^z+V^1_{Z^2\oZ^2}\,\omega^{\oz}+V^1_{Z^2\oZ j}\,\omega^j+4i\,\overline\mu_{U_1}+\alpha^1_{U_2}-i\,\alpha^1_{U_3}-\mu_{ZZ},
\\
0=dV^2_{Z^2\oZ}&=V^2_{Z^3\oZ}\,\omega^z+V^2_{Z^2\oZ^2}\,\omega^{\oz}+V^2_{Z^2\oZ j}\,\omega^j-\overline\mu_{\oZ}-2\,\mu_{Z}+\alpha^2_{U_2}-i\,\alpha^2_{U_3},
\\
0=dV^3_{Z^2\oZ}&=V^3_{Z^3\oZ}\,\omega^z+V^3_{Z^2\oZ^2}\,\omega^{\oz}+V^3_{Z^2\oZ j}\,\omega^j+i\,\overline\mu_{\oZ}+2i\,\mu_{Z}+\alpha^3_{U_2}-i\,\alpha^3_{U_3}.
\endaligned
\]
Then, our cross-section gives rise to the following expressions of the Maurer-Cartan forms:
\[
\aligned
\mu_{U_1}&=-\frac{i}{4} \big(V^1_{Z^2\oZ^2}\,\omega^z+V^1_{Z\oZ^3}\,\omega^{\oz}+V^1_{Z\oZ^2 j}\,\omega^j+\alpha^1_{U_2}+i\,\alpha^1_{U_3}-\overline\mu_{\oZ \,\oZ}\big),
\\
\mu_Z&=\frac{1}{6} \big(3\, V^2_{Z^3\oZ}-3\, V^2_{Z^2\oZ^2}+i\, V^3_{Z^3\oZ}-i\, V^3_{Z^2\oZ^2}\big)\, \omega^z + \frac{1}{6} \big(3\, V^2_{Z^2\oZ^2}-3\, V^2_{Z\oZ^3}+i\, V^3_{Z^2\oZ^2}-i\, V^3_{Z\oZ^3}\big)\, \omega^{\oz}
\\
&+\frac{1}{6} \big(3\, V^2_{Z^2\oZ j}-3\, V^2_{Z\oZ^2j}-i\, V^3_{Z\oZ^2j}+i\, V^3_{Z^2\oZ j}\big) \,\omega^j+\frac{1}{3} \alpha^3_{U_3}-i\, \alpha^2_{U_3},
\\
\alpha^2_{U_2}&=\frac{1}{2}\big(i\,V^3_{Z^3\oZ}-V^2_{Z^2\oZ^2}-V^2_{Z^3\oZ}-i\,V^3_{Z^2\oZ^2}\big)\, \omega^z + \frac{1}{2} \big(i\,V^3_{Z^2\oZ^2}-V^2_{Z^2\oZ^2}-V^2_{Z\oZ^3}-i\,V^3_{Z\oZ^3}\big)\, \omega^{\oz}
\\
&+\frac{1}{2} \big(i\,V^3_{Z^2\oZ j}-V^2_{Z^2\oZ j}-V^2_{Z\oZ^2j}-i\,V^3_{Z\oZ^2j}\big)\, \omega^j+\alpha^3_{U_3}
\\
\alpha^3_{U_2}&=\frac{1}{2}\big(i\, V^2_{Z^2\oZ^2}-i\, V^2_{Z^3\oZ}-V^3_{Z^3\oZ}-V^3_{Z^2\oZ^2}\big)\,\omega^z+\frac{1}{2}\big(i\, V^2_{Z\oZ^3}-i\, V^2_{Z^2\oZ^2}-V^3_{Z^2\oZ^2}-V^3_{Z\oZ^3}\big) \,\omega^{\oz}
\\
&+\frac{1}{2} \big(i\, V^2_{Z\oZ^2j}-i\, V^2_{Z^2\oZ j}-V^3_{Z^2\oZ j}-V^3_{Z\oZ^2j}\big)\, \omega^j-\alpha^2_{U_3}.
\endaligned
\]

\begin{Lemma}
\label{lem-observ-3}
Let $j, l\geq 0$ and $\ell\in\mathbb N^3$. One can specify
\begin{itemize}
\item[$1.$] the Maurer-Cartan form $\mu_{U_1U^\ell}$ by normalizing to zero the lifted differential invariant $V^1_{Z^{2}\oZ U^\ell}$.
 \item[$2.$] the Maurer-Cartan form $\mu_{Z^{j+1}U^\ell}$ by normalizing to zero the lifted differential invariant $V^2_{Z^{j+2}\oZ U^\ell}$.
  \item[$3.$] the real Maurer-Cartan forms $\alpha^2_{U_2^{j+1} U_3^l}$ and $\alpha^3_{U_2^{j+1} U_3^l}$ by normalizing to zero the lifted differential invariant $V^3_{Z^{2}\oZ U_2^{j} U_3^l}$.
 \end{itemize}
\end{Lemma}

\proof
In the expression of the third order prolonged vector coefficient $\phi^{1;z^2\oz}$, one finds some nonzero constant coefficient of the term $(v^1_{z\oz})^2 \overline\xi_{u_1}$. Thus, according to the prolongation formula \eqref{prolong-formula} and for each $\ell\in\mathbb N^3$, one expects the higher order prolonged vector coefficient $\phi^{1; z^2\oz u^\ell}$ to include the monomial $(v^1_{z\oz})^2 \overline\xi_{u_1 u^\ell}$, as well. Due to the normalization $V^1_{Z\oZ}=1$, the lift of this monomial is nothing but the Maurer-Cartan form $\overline\mu_{U_1U^\ell}$. Hence, one can specify this form\,\,---\,\,and also its conjugation\,\,---\,\,by normalizing $V^1_{Z^2\oZ U^\ell}=0$ in the recurrence relation of this lifted differential invariant. This completes the proof of the first assertion. The proof of the next two assertions is quite similar. Indeed, they hold since the monomials $-v^2_{z^2\oz} (2\xi+\overline\xi)$ and $i v^2_{z^2\oz}\eta^2_{u_2}+i v^3_{z^2\oz}\eta^2_{u_3}$ appear respectively in the expressions of $\phi^{2; z^2\oz}$ and $i\phi^{2; z^2\oz}+\phi^{3; z^2\oz}$.
\endproof

It follows from Lemmas \ref{lem-observ-1}, \ref{lem-observ-2} and \ref{lem-observ-3} that the collection of linearly independent Maurre-Cartan forms \eqref{Basis-MC-original} is reducible to:
\begin{equation}\label{unnormalized-MC-prd-3}
\mu_{U_2^k U_3^l}, \qquad \overline\mu_{U_2^k U_3^l} \qquad \alpha^1_{U_2^k U_3^l}, \qquad \alpha^2_{U_3^{l+1}}, \qquad \alpha^3_{U_3^{l+1}},
\end{equation}
for  $(0,0)\neq(k,l)\in\mathbb N^2$ and $\ell\in\mathbb N^3$.

\subsection{Order Four}

We start this order with the normalizations $V^1_{Z^3\oZ}=V^j_{Z^2\oZ^2}=0$ for $j=1, 2, 3$. Thanks to Lemmas \ref{lem-observ-1}, \ref{lem-observ-2} and \ref{lem-observ-3}, the recurrence relations of these lifted differential invariants simplify to:

\begin{equation}
\label{rec-ord-4}
\aligned
0&=dV^1_{Z^2\oZ^2}=\big(4i\, V^2_{Z^2\oZ^2 U_1}+V^1_{Z^3\oZ^2}\big)\, \omega^z-\big(4i\, V^2_{Z^2\oZ^2 U_1}-V^1_{Z^2\oZ^3}\big)\, \omega^{\oz}+V^1_{Z^2\oZ^2 j}\, \omega^j
\\
&\ \ \ \ \ \ \ \ \ \ \ \ \ \ \ \ \ -4i\, \big(V^3_{Z^2\oZ U_1}-V^3_{Z\oZ^2U_1}\big)\, \alpha^2_{U_3} -4i\, \big(\mu_{U_2}-\overline\mu_{U_2}\big)-12 \,\big(\mu_{U_3}+\overline\mu_{U_3}\big),
\\
0&=dV^1_{Z^3\oZ}=\frac{5}{4}\, V^1_{Z^4\oZ}\, \omega^z+\frac{1}{4} \big(4 \,V^1_{Z^3\oZ^2}+V^1_{Z^4\oZ}-V^2_{Z^4\oZ^2}\big)\, \omega^{\oz}+V^1_{Z^3\oZ j}\, \omega^j
\\
&\ \ \ \ \ \ \ \ \ \ \ \ \ \ \ \ +V^3_{Z^3\oZ}\, \alpha^1_{U_3}-\frac{1}{4}\, V^3_{Z^4\oZ}\, \alpha^2_{U_3}+3i\, \overline\mu_{U_2}+3\, \overline\mu_{U_3},
\\
0&=dV^2_{Z^2\oZ^2}=V^2_{Z^3\oZ^2}\, \omega^z+V^2_{Z^2\oZ^3}\, \omega^{\oz}+V^2_{Z^2\oZ^2 j}\, \omega^j-2\,\alpha^1_{U_2},
\\
0&=dV^3_{Z^2\oZ^2}=\big(V^3_{Z^3\oZ^2}-2i\,V^3_{Z^2\oZ U_1}\big)\, \omega^z+\big(V^3_{Z^2\oZ^3}+2i\, V^3_{Z\oZ^2 U_1}\big)\, \omega^{\oz}+V^3_{Z^2\oZ^2 j} \, \omega^j-2\,\alpha^1_{U_3}.
\endaligned
\end{equation}

These equations result in the following expressions for the Maurer-Cartan forms:
\begin{equation*}
\label{normalization-ord-4}
\aligned
{\rm Im}\,\mu_{U_2}&=-\frac{1}{8} \bigg(\big(4i \,V^2_{Z^2\oZ^2 U_1}+V^1_{Z^3\oZ^2}\big)\, \omega^z-\big(4i\, V^2_{Z^2\oZ^2 U_1}-V^1_{Z^2\oZ^3}\big)\, \omega^{\oz}+V^1_{Z^2\oZ^2 U_j}\, \omega^j
\\
&-4i\, \big(V^3_{Z^2\oZ U_1}-V^3_{Z\oZ^2U_1}\big) \alpha^2_{U_3}-12\, \big(\mu_{U_3}+\overline\mu_{U_3}\big)\bigg),
\\
\mu_{U_3}&=\frac{1}{96} \bigg(15\, V^1_{Z^4 \oZ}-12i\, V^3_{Z^3\oZ} V^3_{Z^2\oZ U_1}+20i\, V^3_{Z\oZ^3} V^3_{Z^2\oZ U_1}+12i\, V^2_{Z^2\oZ^2 U_1}+3\, V^1_{Z^3\oZ^2}
\\
&-20\, V^1_{Z^2\oZ^3}-5\, V^1_{Z\oZ^4}+5\, V^2_{Z^2\oZ^4}-10\, V^3_{Z\oZ^3} V^3_{Z^3\oZ^2}+6\, V^3_{Z^3\oZ} V^3_{Z^3\oZ^2}\bigg)\, \omega^z
\\
&+\frac{1}{96}\bigg(12\, V^1_{Z^3\oZ^2}-25\, V^1_{Z\oZ^4}+12i\, V^3_{Z^3\oZ} V^3_{Z\oZ^2 U_1}-20i\, V^3_{Z\oZ^3} V^3_{Z\oZ^2 U_1}-12i\, V^2_{Z^2\oZ^2 U_1}
\\
&+3\, V^1_{Z^2\oZ^3}+3 \,V^1_{Z^4\oZ}-3\, V^2_{Z^4\oZ^2}-10\, V^3_{Z\oZ^3} V^3_{Z^2\oZ^3}+6 \,V^3_{Z^3\oZ} V^3_{Z^2\oZ^3}\bigg)\, \omega^{\oz}
\\
&+\frac{1}{96} \bigg(3\, V^1_{Z^2\oZ^2U_j}+12\, V^1_{Z^3\oZ U_j}-20\, V^1_{Z\oZ^3U_j}+6 \,V^3_{Z^3\oZ} V^3_{Z^2\oZ^2U_j}-10\, V^3_{Z\oZ^3} V^3_{Z^2\oZ^2U_j}\bigg) \,\omega^j
\\
&+\frac{1}{96} \bigg(5\, V^3_{Z\oZ^4}-3\, V^3_{Z^4\oZ}-12i\, V^3_{Z2\oZ U_1}+12i\, V^3_{Z\oZ^2U1}\bigg)\, \alpha^2_{U_3}+i \,{\rm Re}\, \mu_{U_2},
\\
\alpha^1_{U_2}&=\frac{1}{2} \,V^2_{Z^3\oZ^2}\, \omega^z+\frac{1}{2}\,V^2_{Z^2\oZ^3} \,\omega^{\oz}+\frac{1}{2}\,V^2_{Z^2\oZ^2 U_j}\, \omega^j,
\\
\alpha^1_{U_3}&=\big(\frac{1}{2} \,V^3_{Z^3\oZ^2}-i\, V^3_{Z^2\oZ U_1}\big)\, \omega^z+ \big(\frac{1}{2}\, V^3_{Z^2\oZ^3}+i\, V^3_{Z\oZ^2 U_1}\big)\, \omega^{\oz}+\frac{1}{2}\, V^3_{Z^2\oZ^2 U_j}\, \omega^j.
\endaligned
\end{equation*}

Following the same lines of the proofs of Lemmas \ref{lem-observ-1}, \ref{lem-observ-2} and \ref{lem-observ-3}, a symbolic view to the computations of the recurrence relations \eqref{rec-ord-4} helps us to find some more general observations concerning possible normalizations in higher orders. For brevity, we omit the proof of these observations as they essentially resemble an slight repetition of the arguments presented in the proofs of the already mentioned lemmas.

\begin{Lemma}
\label{lem-observ-4}
For each $j, l\geq 0$ it is always possible to
\begin{itemize}
  \item[1.] specify the Maurer-Cartan form ${\rm Im}\,\mu_{U_2^{j+1}}$ by normalizing to zero the lifted invariant $V^1_{Z^2\oZ^2 U_2^j}$.
  \item[2.] specify the Maurer-Cartan form $\mu_{U_2^j U_3^{l+1}}$ by normalizing to zero the lifted invariant $V^1_{Z^3\oZ U_2^j U_3^l}$.
  \item[3.] specify the Maurer-Cartan form $\alpha^1_{U_2^{j+1} U_3^l}$ by normalizing to zero the lifted invariant $V^2_{Z^2\oZ^2 U_2^j U_3^l}$.
  \item[4.] specify the Maurer-Cartan form $\alpha^1_{U_3^{l+1}}$ by normalizing to zero the lifted invariant $V^3_{Z^2\oZ^2 U_3^l}$.
\end{itemize}
\end{Lemma}

Following this lemma, the collection of yet unnormalized Maurer-Cartan forms \eqref{unnormalized-MC-prd-3} reduces to:
\begin{equation}
\label{unnormalized-MC-prd-4}
{\rm Re}\,\mu_{U_2^{j+1}}, \qquad \alpha^2_{U_3^{l+1}}, \qquad \alpha^3_{U_3^{l+1}}, \qquad\qquad j, l\geq 0.
\end{equation}

By Lemmas \ref{lem-observ-1}, \ref{lem-observ-2} and \ref{lem-observ-3}, it remains in this order to only consider the recurrence relations of the two lifted differential invariants $V^3_{Z^3\oZ}$ and $V^3_{Z^2\oZ U_1}$. Our computations show:
\begin{equation}
\label{VZ3Zbar}
\aligned
dV^3_{Z^3\oZ}&=\big(V^3_{Z^4\oZ}-\frac{2i}{3}\, V^3_{Z^3\oZ}V^3_{Z^3\oZ}\big)\,\omega^z+\big(V^3_{Z^3\oZ^2}+i\, V^2_{Z^3\oZ^2}+\frac{2i}{3}\,V^3_{Z^3\oZ}V^3_{Z\oZ^3}\big)\,\omega^{\oz}
\\
& +\big(\frac{2i}{3}\, V^3_{Z^3\oZ} V^3_{Z\oZ^2 U_j}-\frac{2i}{3}\, V^3_{Z^3\oZ} V^3_{Z^2\oZ U_j}+V^3_{Z^3\oZ U_j}\big)\, \omega^j+V^3_{Z^3\oZ} \,\big(3i\, \alpha^2_{U_3}-\frac{1}{3}\, \alpha^3_{U_3}\big),
\\
dV^3_{Z^2\oZ U_1}&=\frac{1}{12} \bigg(4i\, V^3_{Z^3\oZ} V^3_{Z^2\oZ U_1}-10 \,V^1_{Z^4\oZ}+3i\, V^3_{Z^3\oZ} V^2_{Z^3\oZ^2}-2i\, V^3_{Z^3\oZ} V^2_{Z^2\oZ^3}+ V^3_{Z^3\oZ U_1}-7\, V^3_{Z^3\oZ} V^3_{Z^3\oZ^2}\bigg)\, \omega^z
\\
&+\frac{1}{12} \bigg(10 i\, V^3_{Z\oZ^3} V^3_{Z^2\oZ U_1}-14i \,V^3_{Z^3\oZ} V^3_{Z\oZ^2 U_1}+3i \,V^3_{Z^3\oZ} V^2_{Z^2\oZ^3}+V^3_{Z^2\oZ^2 U_1}-7\, V^3_{Z^3\oZ} V^3_{Z^2\oZ^3}
\\
&\ \ \ \ \ \ \ \ \ \ \ +i\, V^2_{Z^2\oZ^2 U_1}-8\, V^1_{Z^3\oZ^2}-2 \,V^1_{Z^4\oZ}+2\, V^2_{Z^4\oZ^2}\bigg) \,\omega^{\oz}
\\
&+\frac{1}{12} \bigg(10i\, V^3_{Z^2\oZ U_1} V^3_{Z\oZ^2 U_j}-10i\, V^3_{Z^2\oZ U_1} V^3_{Z^2\oZ U_j}+3i\, V^3_{Z^3\oZ} V^2_{Z^2\oZ^2 U_j}-7\, V^3_{Z^3\oZ} V^3_{Z^2\oZ^2 U_j}
\\
&\ \ \ \ +V^3_{Z^2\oZ U_1 U_j} -8\, V^1_{Z^3\oZ U_j}\bigg)\, \omega^j+\frac{1}{6}\, \big(V^3_{Z^4\oZ}-i\, V^3_{Z^3\oZ} V^3_{Z\oZ^3}+12i\, V^3_{Z^2\oZ U_1}\big)\, \alpha^2_{U_3}-\frac{2}{3}\, V^3_{Z^2\oZ U_1}\, \alpha^3_{U_3}.
\endaligned
\end{equation}

These equations indicate that $V^3_{Z^3\oZ}$ and $V^3_{Z^2\oZ U_1}$ are {\it relative differential invariants} \cite{Olver-1995}. Therefore, vanishing or non-vanishing of the coefficients of the monomials $z^3\oz$ and $z^2\oz u_1$ in the third defining equation of \eqref{normalization-5cubic} is an invariant property under holomorphic transformations. By the above recurrence relations, the possibility of normalizing the two real Maurer-Cartan forms $\alpha^2_{U_3}$ and $\alpha^3_{U_3}$ depends upon the vanishing and non-vanishing of the two differential invariants $V^3_{Z^3\oZ}$ and $V^3_{Z^2\oZ U_1}$. Thus, we have to distinguish three major branches in the future issues of constructing desired moving frame (normal form):
\begin{itemize}
  \item[$\Box$] {\bf Branch 1.} $V^3_{Z^3\oZ}\neq 0$.
  \item[$\Box$] {\bf Branch 2.} $V^3_{Z^3\oZ}= 0$ and $V^3_{Z^2\oZ U_1}\neq 0$.
  \item[$\Box$] {\bf Branch 3.} $V^3_{Z^3\oZ}= V^3_{Z^2\oZ U_1}= 0$.
\end{itemize}

Thus, the current order four was the last order in which our desired normal form appears in a uniform fashion. From now on, we have to proceed our normal form construction along three separate lines. Anyway, according to the discussion introduced at the end of Section \ref{sec-moving-frame} ({\it cf.} \cite{Olver-2018}) and up to this order, we have succeeded in construction a {\it universal} and {\it partial} normal form, presented in the following

\begin{Theorem}
Every $5$-dimensional real-analytic totally nondegenerate CR manifold $M^5\subset\mathbb C^4$ can be transformed, through some origin-preserving holomorphic transformation, into the {\sl partial} normal form:
\begin{equation}
\label{partial-normal-form}
\aligned
v^1&=z\oz+ \sum_{j+k+\sharp\ell\geq 5} \frac{1}{j!\,k!\,\ell!}\,V^1_{Z^j\oZ^k U^\ell} z^j \oz^k u^\ell,
\\
v^2&=\frac{1}{2}\,(z^2\oz+z\oz^2)+\sum_{j+k+\sharp\ell\geq 5} \frac{1}{j!\, k!\, \ell!}\,V^2_{Z^j\oZ^k U^\ell} z^j \oz^k u^\ell,
\\
v^3&=-\frac{i}{2}\,(z^2\oz-z\oz^2)+\frac{1}{6}\,V^3_{Z^3\oZ}\,z^3\oz+\frac{1}{6}\,V^3_{Z\oZ^3}\,z\oz^3+\frac{1}{2}\,V^3_{Z^2\oZ U_1} z^2\oz u_1+\frac{1}{2}\,V^3_{Z\oZ^2 U_1} z\oz^2 u_1
\\
& \ \ \ \ \ \ \ +\sum_{j+k+\sharp\ell\geq 5} \frac{1}{j! \, k! \, \ell!}\, V^3_{Z^j\oZ^k U^\ell} z^j \oz^k u^\ell,
\endaligned
\end{equation}
where, taking into account the conjugation relation $\overline{V^\bullet_{Z^j\oZ^k U^\ell}}=V^\bullet_{Z^k\oZ^j U^\ell}$, we have the cross-section normalizations:
\begin{equation}
\label{normalization-partial}
\aligned
0&=V^1_{Z^j U^\ell}=V^1_{Z\oZ U^\ell}=V^1_{Z^2\oZ U^\ell}=V^1_{Z^2\oZ^2 U_2^j}=V^1_{Z^3\oZ U_2^j U_3^l},
\\
0&=V^2_{Z^j U^\ell}=V^2_{Z^j\oZ U^\ell}=V^2_{Z^2\oZ^2 U_2^j U_3^l},
\\
0&=V^3_{Z^j U^\ell}=V^3_{Z\oZ U^\ell}=V^3_{Z^2\oZ U_2^j U_3^l}=V^3_{Z^2\oZ^2 U_3^l},
\endaligned
\end{equation}
for $j, l\in\mathbb N$ and $\ell\in\mathbb N^3$.
\end{Theorem}

In order to complete this partial normal form, we have to proceed in the next sections along the three appeared branches. This situation may remember that of part (d), page 246 of \cite{Chern-Moser}, where a more explicit expression of the normal form (3.18), constructed for real hypersurfaces in $\mathbb C^2$, depends upon the invariant property of vanishing or non-vanishing of the coefficient of the monomial $z^4\oz^2$.

\section{Branch $1$: $V^3_{Z^3\oZ}\neq 0$}
\label{Sec-Branch-1}

As the first branch, let us assume that the coefficient $V^3_{Z^3\oZ}$ in the partial normal form \eqref{partial-normal-form} is nonzero. Hence, we may plainly normalize it to $1$. This normalization brings the recurrence relation of this differential invariant in \eqref{VZ3Zbar} into the simple form:

 \begin{equation*}
\aligned
0=dV^3_{Z^3\oZ}&=\big(V^3_{Z^4\oZ}-\frac{2i}{3}\big)\,\omega^z+\big(V^3_{Z^3\oZ^2}+i\, V^2_{Z^3\oZ^2}+\frac{2i}{3}\big)\,\omega^{\oz}
\\
& +\big(\frac{2i}{3}\, V^3_{Z\oZ^2 U_j}-\frac{2i}{3}\, V^3_{Z^2\oZ U_j}+V^3_{Z^3\oZ U_j}\big) \,\omega^j+3i\, \alpha^2_{U_3}-\frac{1}{3}\, \alpha^3_{U_3}.
\endaligned
\end{equation*}
Keeping in mind that $\alpha^2_{U_3}$ and $\alpha^3_{U_3}$ are real Maurer-Cartan forms, this equation leads us to specify them as:
\[
\aligned
\alpha^2_{U_3}&=\frac{i}{6}\big(i\, V^2_{Z^2\oZ^3}-V^3_{Z^2 \oZ^3}+ V^3_{Z^4 \oZ}\big)\, \omega^z+\frac{i}{6} \big(i\, V^2_{Z^3 \oZ^2}+ V^3_{Z^3 \oZ^2}-V^3_{Z \oZ^4}\big)\, \omega^{\oz}+ \frac{i}{6}\big(V^3_{Z^3\oZ U_j}-V^3_{Z\oZ^3 U_j}\big)\, \omega^j,
\\
\alpha^3_{U_3}&=\frac{2}{3}\big(V^3_{Z^4 \oZ}+V^3_{Z^2 \oZ^3}-i\, V^2_{Z^2\oZ^3}-\frac{4i}{3}\big)\, \omega^z+\frac{2}{3}\big(V^3_{Z^3 \oZ^2}+i\, V^2_{Z^3 \oZ^2}+V^3_{Z \oZ^4}+\frac{4i}{3}\big)\, \omega^{\oz}
\\
&+\frac{2}{3}\big(\frac{4i}{3}\, V^3_{Z \oZ^2 U_j}-\frac{4i}{3}\, V^3_{Z^2 \oZ U_j}+V^3_{Z^3 \oZ U_j}+V^3_{Z \oZ^3 U_j}\big)\, \omega^j.
\endaligned
\]

A close inspection of the computations concerned the recurrence relation of $dV^3_{Z^3\oZ}$ reveals that

\begin{Lemma}
  \label{normalization-alpha2-alpha3-U3-br-1}
In Branch $1$ and for any $j\geq 0$, it is possible to normalize Maurer-Cartan forms $\alpha^2_{U_3^{j+1}}$ and $\alpha^3_{U_3^{j+1}}$ by normalizing $V^3_{Z^3\oZ}=1$ and $V^3_{Z^3\oZ U_3^j}=0$, when $j\neq 0$.
\end{Lemma}

At this stage, most of the Maurer-Cartan forms are normalized except those of the real form ${\rm Re}\, \mu_{U_2^{j+1}}$ ({\it cf.} \eqref{unnormalized-MC-prd-4}). In order to check the possibility of their normalization, we have to proceed into the next order five.

\subsection{Order 5 -- Branch 1}

After examining several order five recurrence relations, we realized finally that the normalization ${\rm Re} V^1_{Z^4\oZ}=0$ may bring our desired result. By this normalization, the corresponding recurrence formula gives the long expression:
\[
\aligned
dV^1_{Z^4\oZ}&=\bigg(\frac{i}{2} V^1_{Z^2\oZ^3}-\frac{11i}{12} V^3_{Z^4\oZ} V^3_{Z^2\oZ U_1}+\frac{i}{48} V^3_{Z\oZ^4} V^2_{Z^2\oZ^3}-\frac{i}{12} V^3_{Z\oZ^2U_1} V^3_{Z^4\oZ}-\frac{i}{30} V^3_{Z^5\oZ} V^3_{Z^4\oZ}+\frac{i}{8} V^1_{Z^3\oZ^2}
\\
&+\frac{i}{12} V^3_{Z\oZ^2U_1} V^3_{Z^2\oZ^3}-\frac{5i}{144} V^3_{Z^4\oZ} V^2_{Z^2\oZ^3}+\frac{5}{144} V^3_{Z^4\oZ} V^3_{Z^2\oZ^3}+\frac{1}{2} V^3_{Z^4\oZ} V^3_{Z^3\oZ^2}+\frac{1}{48} V^3_{Z\oZ^4} V^3_{Z^4\oZ}
\\
&-\frac{1}{48} V^3_{Z\oZ^4} V^3_{Z^2\oZ^3}+\frac{1}{12} V^3_{Z\oZ^2U_1} V^2_{Z^2\oZ^3}-\frac{1}{12} V^3_{Z^2\oZ U_1} V^2_{Z^2\oZ^3}+\frac{1}{30} V^3_{Z^5\oZ} V^2_{Z^2\oZ^3}-\frac{i}{8} V^2_{Z^2\oZ^4}+\frac{6}{5} V^1_{Z^5\oZ}
\\
&+\frac{i}{30} V^3_{Z^5\oZ} V^3_{Z^2\oZ^3}-\frac{1}{2} V^2_{Z^2\oZ^2U_1}-\frac{1}{3} V^3_{Z^2\oZ U_1}-\frac{5}{144} V^3_{Z^4\oZ} V^3_{Z^4\oZ}-\frac{i}{12}V^3_{Z^2\oZ U_1} V^3_{Z^2\oZ^3}-\frac{i}{6} V^3_{Z^3\oZ^2}
\\
&-\frac{7}{6} V^1_{Z^4\oZ} V^3_{Z^4\oZ}-\frac{1}{6} V^1_{Z^4\oZ} V^3_{Z^2\oZ^3}+\frac{i}{6} V^1_{Z^4\oZ} V^2_{Z^2\oZ^3}-\frac{7i}{9} V^1_{Z^4\oZ}\bigg)\, \omega^z
\\
&+\bigg(\frac{5i}{24} V^2_{Z^4\oZ^2}-\frac{5i}{6} V^1_{Z^3\oZ^2}-\frac{i}{6} V^3_{Z^2\oZ^3}+\frac{i}{8} V^1_{Z^2\oZ^3}-\frac{5i}{144} V^3_{Z^4\oZ} V^2_{Z^3\oZ^2}-\frac{i}{30} V^3_{Z^5\oZ} V^3_{Z^3\oZ^2}+\frac{i}{30} V^3_{Z^5\oZ} V^3_{Z\oZ^4}
\\
&+\frac{i}{12} V^3_{Z^2\oZ U_1} V^3_{Z^3\oZ^2}+\frac{i}{48} V^3_{Z\oZ^4} V^2_{Z^3\oZ^2}-\frac{i}{12} V^3_{Z\oZ^2U_1} V^3_{Z^3\oZ^2}-\frac{i}{12} V^3_{Z^2\oZ U_1} V^3_{Z\oZ^4}+\frac{i}{12} V^3_{Z\oZ^2U_1} V^3_{Z\oZ^4}
\\
&+\frac{1}{2} V^3_{Z^4\oZ} V^3_{Z^2\oZ^3}-\frac{5}{144} V^3_{Z^4\oZ} V^3_{Z^3\oZ^2}+\frac{5}{144} V^3_{Z\oZ^4} V^3_{Z^4\oZ}+\frac{1}{48} V^3_{Z\oZ^4} V^3_{Z^3\oZ^2}+\frac{1}{12} V^3_{Z\oZ^2U_1} V^2_{Z^3\oZ^2}
\\
&-\frac{1}{12} V^3_{Z^2\oZ U_1} V^2_{Z^3\oZ^2}+\frac{1}{30} V^3_{Z^5\oZ} V^2_{Z^3\oZ^2}+\frac{1}{2} V^2_{Z^2\oZ^2U_1}+\frac{1}{3} V^3_{Z\oZ^2U_1}+\frac{1}{5} V^1_{Z^5\oZ}-\frac{1}{48} V^3_{Z\oZ^4} V^3_{Z\oZ^4}
\\
&+V^1_{Z^4\oZ^2}-\frac{1}{5} V^2_{Z5\oZ^2}+i V^3_{Z\oZ^2U_1}V^3_{Z^4\oZ}-\frac{1}{6} V^1_{Z^4\oZ} V^3_{Z\oZ^4}-\frac{7}{6} V^1_{Z^4\oZ} V^3_{Z^3\oZ^2}-\frac{7i}{6} V^1_{Z^4\oZ} V^2_{Z^3\oZ^2}-\frac{11i}{9} V^1_{Z^4\oZ}\bigg) \, \omega^{\oz}
\\
&+\bigg(\frac{i}{12} V^3_{Z^2\oZ U_1} V^3_{Z^3\oZ U_j}-\frac{i}{6} V^3_{Z^2\oZ^2 U_j}+\frac{i}{2} V^1_{Z\oZ^3 U_j}-\frac{5i}{6} V^1_{Z^3\oZ U_j}+\frac{i}{12} V^3_{Z\oZ^2 U_1} V^3_{Z\oZ^3 U_j}+\frac{1}{2} V^3_{Z^4\oZ} V^3_{Z^2\oZ^2 U_j}
\\
&-\frac{5}{144} V^3_{Z^4\oZ} V^3_{Z^3\oZ U_j}+\frac{5}{144} V^3_{Z^4\oZ} V^3_{Z\oZ^3 U_j}+\frac{1}{48} V^3_{Z\oZ^4} V^3_{Z^3\oZ U_j}-\frac{1}{48} V^3_{Z\oZ^4} V^3_{Z\oZ^3 U_j}+\frac{i}{30} V^3_{Z^5\oZ} V^3_{Z\oZ^3 U_j}
\\
&+\frac{i}{8} V^1_{Z^2\oZ^2 U_j}-\frac{i}{12} V^3_{Z^2\oZ U_1} V^3_{Z\oZ^3 U_j}-\frac{i}{30} V^3_{Z^5\oZ} V^3_{Z^3\oZ U_j}-\frac{i}{12} V^3_{Z\oZ^2 U_1} V^3_{Z^3\oZ U_j}+V^1_{Z^4\oZ U_j}
\\
&-\frac{7i}{18} V^1_{Z^4\oZ} V^3_{Z\oZ^2 U_j}+\frac{7i}{18} V^1_{Z^4\oZ} V^3_{Z^2\oZ U_j}-\frac{7}{6} V^1_{Z^4\oZ} V^3_{Z^3\oZ U_j}-\frac{1}{6} V^1_{Z^4\oZ} V^3_{Z\oZ^3 U_j}\bigg)\, \omega^j +4 \, {\rm Re} \mu_{U_2}.
\endaligned
\]
This equation readily specifies the expression of the real Maurer-Cartan form ${\rm Re}\, \mu_{U_2}$. Examining the monomials which appeared among the above performed computations shows that in general we have

\begin{Lemma}
\label{lemma-normalization-Re-mu-U2-br-1}
In Branch $1$ and for any $j\geq 0$, one can specify each real Maurer-Cartan form ${\rm Re}\, \mu_{U_2^{j+1}}$ by normalizing to zero the real lifted differential invariant ${\rm Re} V^1_{Z^4\oZ U_2^j}$.
\end{Lemma}

At this stage, there is no further Maurer-Cartan form remained unnormalized. This amounts to state that in this branch a complete equivariant moving frame and equivalently a complete normal form is available in the current order five. By the discussion presented at the end of Section \ref{sec-moving-frame}, the holomorphic transformation which brings a CR manifold $M^5$  to this normal form is unique. Summing up the obtained results, we have

\begin{Theorem}
\label{th-branch-1}
Let $M^5\subset\mathbb C^4$ be a five dimensional real-analytic totally nondegenerate CR manifold enjoying the assumption $V^3_{Z^3\oZ}\neq 0$. Then, there exists a unique origin-preserving holomorphic transformation which brings it into the complete normal form:
\begin{equation}
\label{normal-form-Branch-1}
\aligned
v^1&=z\oz+ \sum_{j+k+\sharp\ell\geq 5} \frac{1}{j!\,k!\,\ell!}\,V^1_{Z^j\oZ^k U^\ell} z^j \oz^k u^\ell,
\\
v^2&=\frac{1}{2}\,(z^2\oz+z\oz^2)+\sum_{j+k+\sharp\ell\geq 5} \frac{1}{j!\, k!\, \ell!}\,V^2_{Z^j\oZ^k U^\ell} z^j \oz^k u^\ell,
\\
v^3&=-\frac{i}{2}\,(z^2\oz-z\oz^2)+\frac{1}{6}\,(z^3\oz+z\oz^3)+\frac{1}{2}\,V^3_{Z^2\oZ U_1} z^2\oz u_1+\frac{1}{2}\,V^3_{Z\oZ^2 U_1} z\oz^2 u_1
\\
& \ \ \ \ \ \ \ +\sum_{j+k+\sharp\ell\geq 5} \frac{1}{j! \, k! \, \ell!}\, V^3_{Z^j\oZ^k U^\ell} z^j \oz^k u^\ell,
\endaligned
\end{equation}
with  the cross-section normalizations \eqref{normalization-partial} added by:
\begin{equation}
\label{normalization-branch-1}
\aligned
0={\rm Re} V^1_{Z^4\oZ U_2^l}=V^3_{Z^3\oZ U_3^l}
\endaligned
\end{equation}
for $l\geq 0$.
Moreover, in this branch, the infinitesimal CR automorphism algebra $\frak{aut}_{CR}(M^5)$ is five dimensional and the biholomorphic equivalence problem to $M^5$ can be reduced to an absolute parallelism, namely $\{e\}$-structure, on itself with the structure equations of the lifted horizontal coframe $\omega^z, \omega^{\oz}, \omega^j, j=1, 2, 3$ of $M^5$, obtained by \eqref{eq: order zero structure equations} after applying the already mentioned normalizations and inserting the achieved expressions of the normalized Maurer-Cartan forms.
\end{Theorem}

Following the Chern-Moser terminology in \cite{Chern-Moser}, we may represent the normal form \eqref{normal-form-Branch-1} as:
\begin{equation}
\label{NM-CM}
v^k=\Phi^k(z,\oz, u):=\sum_{i, j}\,\Phi^k_{ij}(u_1, u_2, u_3)\, z^i \oz^j, \qquad j=1,2,3,
\end{equation}
for some complex power series $\Phi^k_{ij}:=\Phi^k_{ij}(u_1, u_2, u_3)$ with $\Phi^k_{ij}=\overline{\Phi^k_{ji}}$. Via the following auxiliary schematic diagram we attempt to partly describe these functions in terms of the applied normalizations \eqref{normalization-partial}-\eqref{normalization-branch-1}. In each diagram, a filled black circle standing at a point $(i,j)$ amounts to state that the corresponding function $\Phi_{ij}$ vanishes. But, in contrast, the black-white circles like those standing at the points $(2,1)$ and $(2,2)$ of the diagram associated to $\Phi^3$, respectively mean that the coefficients of the monomials $u_2^j u_3^l$ in $\Phi^3_{2,1}$ and the coefficients of the monomials $u_3^l$ in $\Phi^3_{2,2}$ are identically zero for each arbitrary $j, l\geq 0$.

\begin{figure}[h]
  \centering
\parbox{1.4in}{%
\includegraphics[height=4cm, width=4 cm]{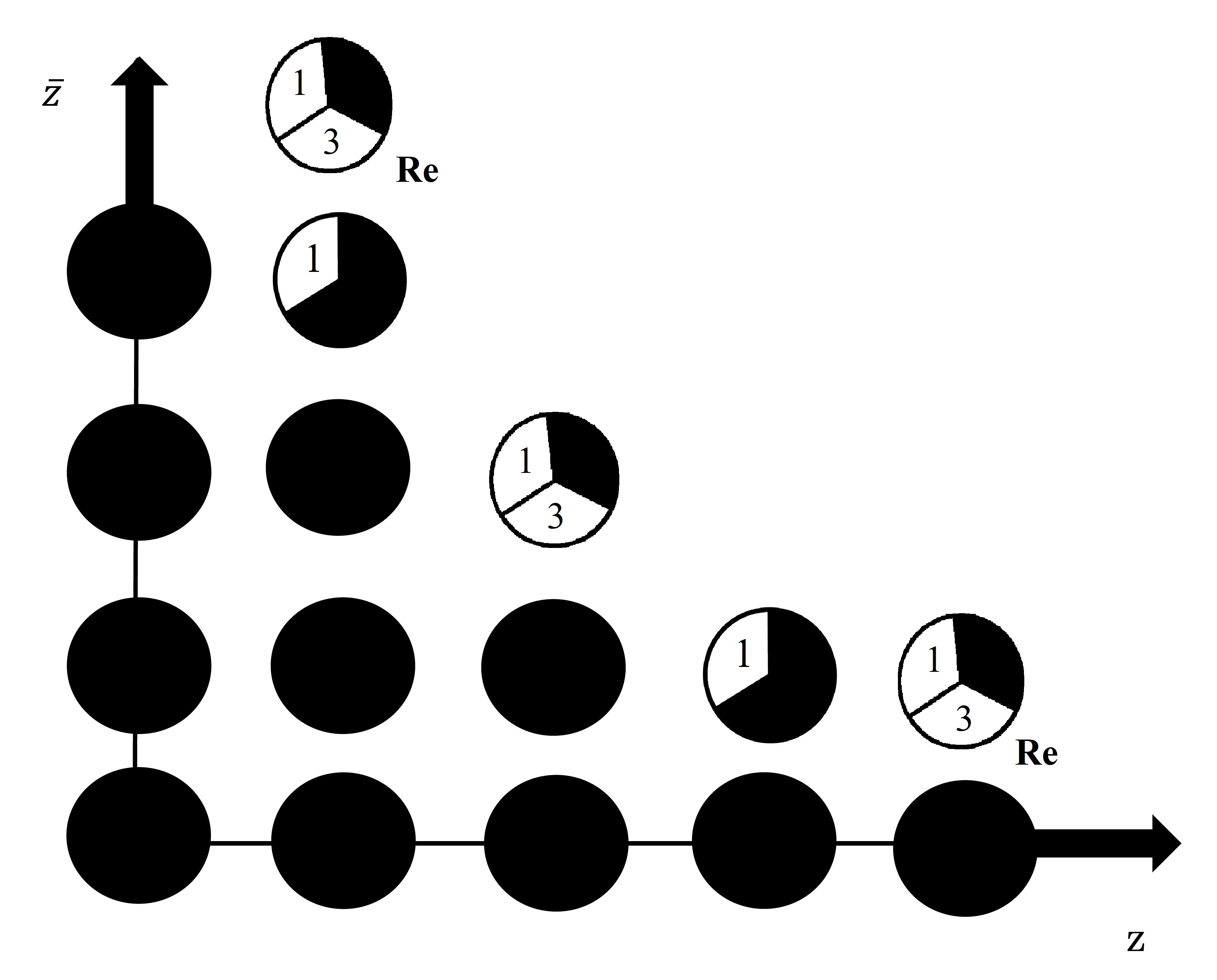}
\caption{$\Phi^1$}}
\qquad\qquad
\parbox{1.4in}{%
\includegraphics[height=4cm, width=4 cm]{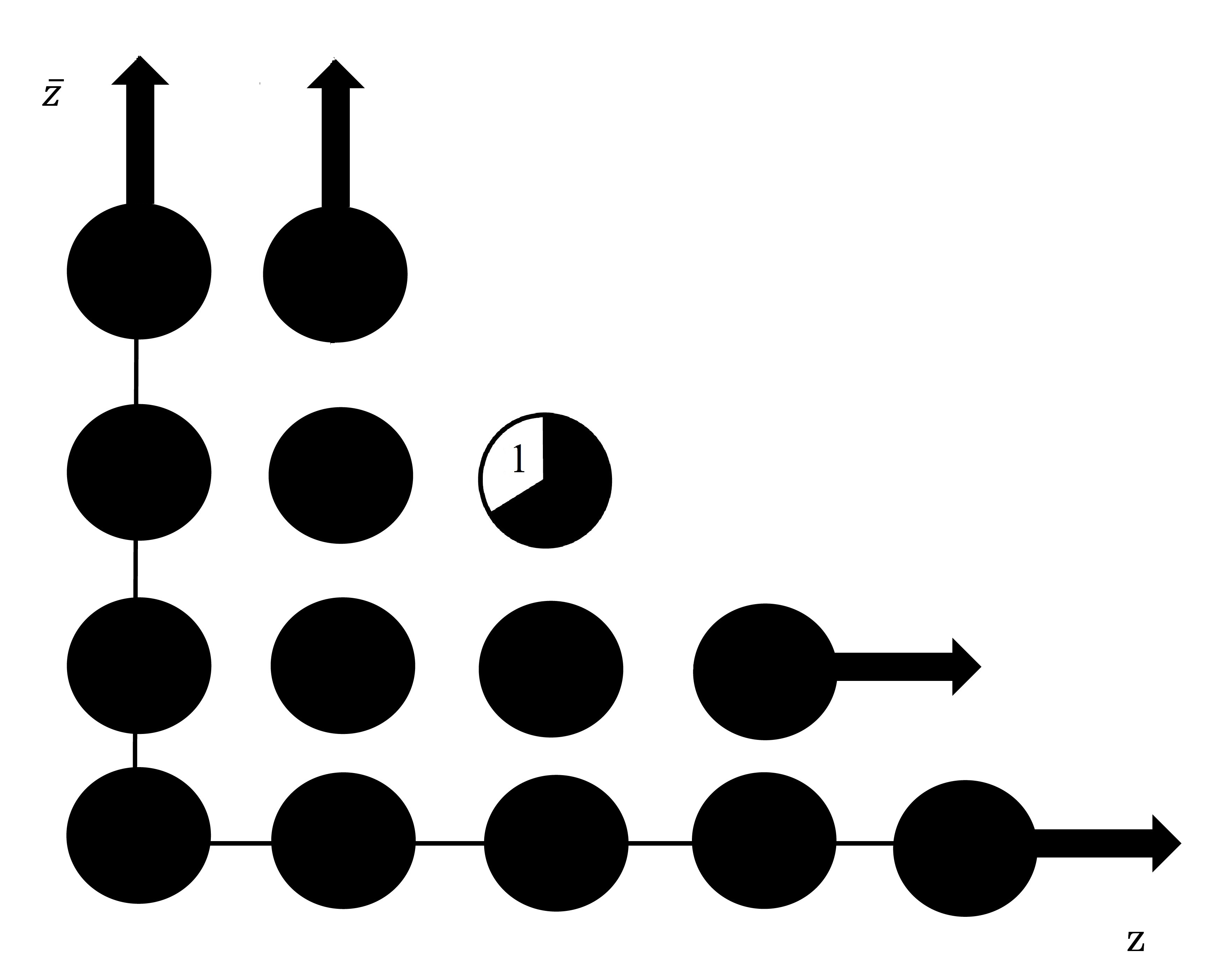}
\caption{ِ$\Phi^2$}}
\qquad\qquad
\begin{minipage}{1.4in}%
\includegraphics[height=4cm, width=4 cm]{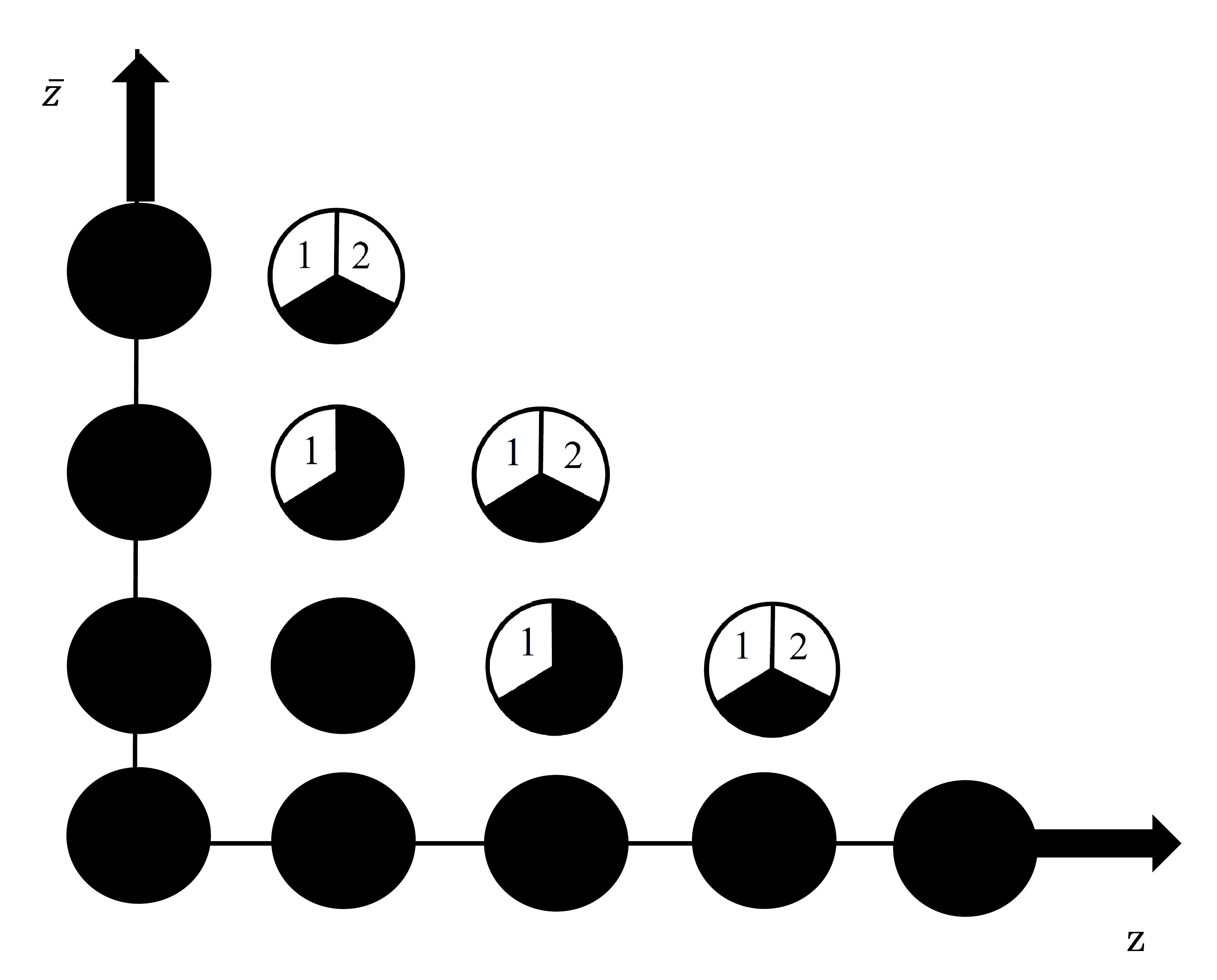}
\caption{$\Phi^3$}
\end{minipage}
\end{figure}
Roughly speaking, a certain cross-section is of {\it minimal order} if it is chosen in such a way that it has as low an order as possible ({\it see} \cite[Definition 3.1]{Olver-2007} for the explicit definition). It is easy to see that our cross-section \eqref{normalization-branch-1} enjoys this property.  By definition (\cite[Definition 5.2]{Olver-2007}), a differential invariant $V^k_{J,j}$ is an {\it edge differential invariant} of our cross-section if $V^k_{J}$ appears in \eqref{normalization-partial}-\eqref{normalization-branch-1}. It is convenient to turn up the edge differential invariants of the cross-section, constructed in this branch.

\begin{Corollary}
The collection of edge differential invariants of the cross-section \eqref{normalization-partial}-\eqref{normalization-branch-1} forms a generating system of differential invariants for the biholomorphic equivalence problem to $5$-dimensional totally nondegenerate CR manifolds $M^5$, belonging to Branch $1$.
\end{Corollary}

\proof
It is a straightforward consequence of \cite[Theorem 7.2]{Olver-2007}.
\endproof

Notice that the above generating set is not necessarily minimal.

\section{Branch $2$: $V^3_{Z^3\oZ}= 0$ and $V^3_{Z^2\oZ U_1}\neq 0$}
\label{Sec-Branch-2}

As the second branch, let us assume that in the partial normal form \eqref{partial-normal-form}, we have $V^3_{Z^3\oZ}=0$ but, in contrary, the coefficient $V^3_{Z^2\oZ U_1}$ does not vanish. By the former assumption, it plainly follows from the first recurrence relation in \eqref{VZ3Zbar} that:
\begin{equation}
\label{v-order-5-br-2}
V^3_{Z^4\oZ}=0, \qquad V^3_{Z^3\oZ^2}=-i V^2_{Z^3\oZ^2}, \qquad V^3_{Z^3\oZ U_j}=0, \qquad \qquad j=1, 2, 3.
\end{equation}
Moreover, by normalizing $V^3_{Z^2\oZ U_1}=1$, the recurrence relation of this differential invariant in \eqref{VZ3Zbar} converts into the form:
\begin{equation*}
\aligned
0=dV^3_{Z^2\oZ U_1}=&-\frac{5}{6}\,  V^1_{Z^4\oZ}\, \omega^z+\frac{1}{12} \big(V^3_{Z^2\oZ^2 U_1}+i V^2_{Z^2\oZ^2 U_1}-8 V^1_{Z^3\oZ^2}-2 V^1_{Z^4\oZ}+2 V^2_{Z^4\oZ^2}\big) \, \omega^{\oz}
\\
&+\frac{1}{12} \big(10i V^3_{Z\oZ^2 U_j}-10i V^3_{Z^2\oZ U_j}+V^3_{Z^2\oZ U_1 U_j}-8 V^1_{Z^3\oZ U_j}\big) \,\omega^j+ 2i\, \alpha^2_{U_3}-\frac{2}{3}\, \alpha^3_{U_3}.
\endaligned
\end{equation*}
Taking into account that $\alpha^2_{U^3}$ and $\alpha^3_{U_3}$ are real Maurer-Cartan forms, this equation results in the expressions:
\[
\aligned
\alpha^2_{U_3}&=\frac{i}{48} \big(i\, V^2_{Z^2\oZ^2 U_1}+8 \,V^1_{Z2\oZ^3}-10\, V^1_{Z^4\oZ}+2\, V^1_{Z\oZ^4}-2\, V^2_{Z^2\oZ^4}-V^3_{Z^2\oZ^2 U_1}\big)\, \omega^{z}
\\
&+\frac{i}{48} \big(i\,V^2_{Z^2\oZ^2 U_1}-8\, V^1_{Z^3\oZ^2}-2 \,V^1_{Z^4\oZ}+10\, V^1_{Z\oZ^4}+2\, V^2_{Z^4\oZ^2}+V^3_{Z^2\oZ^2 U_1}\big)\, \omega^{\oz}
\\
&+\frac{i}{48} \big(8\, V^1_{Z\oZ^3 U_j}-8\,V^1_{Z^3\oZ U_j}+V^3_{Z^2\oZ U_1 U_j}-V^3_{Z\oZ^2 U_1 U_j}\big)\, \omega^j,
\\
\alpha^3_{U_3}&=\big(\frac{1}{16}\, V^3_{Z^2\oZ^2 U_1}-\frac{5}{8} \,V^1_{Z^4\oZ}-\frac{i}{16}\, V^2_{Z^2\oZ^2 U_1}-\frac{1}{2}\, V^1_{Z^2\oZ^3}-\frac{1}{8}\, V^1_{Z\oZ^4}+\frac{1}{8} V^2_{Z^2\oZ^4}\big) \,\omega^z
\\
&+\big(\frac{i}{16} V^2_{Z^2\oZ^2 U_1}-\frac{1}{2}\, V^1_{Z^3\oZ^2}-\frac{1}{8}\, V^1_{Z^4\oZ}+\frac{1}{8} \,V^2_{Z^4\oZ^2}+\frac{1}{16}\, V^3_{Z^2\oZ^2 U_1}-\frac{5}{8} \, V^1_{Z\oZ^4}\big)\, \omega^{\oz}
\\
&+\big(\frac{5i}{4}\, V^3_{Z\oZ^2 U_j}-\frac{5i}{4}\, V^3_{Z^2\oZ U_j}-\frac{1}{2}\, V^1_{Z^3\oZ U_j}+\frac{1}{16}\, V^3_{Z^2\oZ U_1 U_j}-\frac{1}{2}\, V^1_{Z\oZ^3 U_j}+\frac{1}{16}\, V^3_{Z\oZ^2 U_1 U_j}\big) \,\omega^j.
\endaligned
\]

Generally, inspecting monomials appeared among the performed computations shows that

\begin{Lemma}
  \label{normalization-alpha2-alpha3-U3-br-2}
In Branch $2$ and for each $j\geq 0$, it is possible to identically specify the Maurer-Cartan forms $\alpha^2_{U_3^{j+1}}$ and $\alpha^3_{U_3^{j+1}}$ by normalizing $V^3_{Z^2\oZ U_1}=1$ and $V^3_{Z^2\oZ U_1 U_3^j}=0$, when $j\neq 0$.
\end{Lemma}

Similar to the case of Branch $1$ and up to the end of order four, the only yet unnormalized Maurer-Cartan forms are those of the real form ${\rm Re} \,\mu_{U_2^{j+1}}$. Let us proceed into the next order to check the possibility of normalizing these Maurer-Cartan forms.

\subsection{Order 5 -- Branch 2}

In this branch and as suggested by the performed computations, we normalize to zero the imaginary part of the lifted invariant $V^2_{Z^3\oZ^2}$. We have the following long recurrence expression:
\[
\aligned
dV^2_{Z^3\oZ^2}&=\big(\frac{5}{6}\, V^1_{Z^2\oZ^3}-\frac{5}{24}\, V^2_{Z^2\oZ^4}+\frac{5}{24}\, V^1_{Z\oZ^4}-\frac{23}{8}\, V^1_{Z^3\oZ^2}+\frac{i}{3}\, V^2_{Z^3\oZ^2} V^3_{Z2\oZ^2}+\frac{i}{24} \,V^2_{Z^2 \oZ^2 U_1} V^2_{Z^3\oZ^2}
\\
&-\frac{1}{12}\, V^2_{Z^3\oZ^2} V^2_{Z^2\oZ^4}+\frac{1}{3}\, V^2_{Z^3\oZ^2} V^1_{Z^2\oZ^3}-\frac{1}{24}\, V^3_{Z^2\oZ^2 U_1} V^2_{Z^3\oZ^2}+\frac{5}{12}\, V^2_{Z^3\oZ^2} V^1_{Z^4\oZ}+\frac{1}{12}\, V^2_{Z^3\oZ^2} V^1_{Z\oZ^4}
\\
&-\frac{19}{24}\, V^1_{Z^4\oZ}-\frac{1}{24}\, V^3_{Z^2\oZ^2 U_1}-V^2_{Z^3\oZ^2} V^2_{Z^3\oZ}+V^2_{Z^4\oZ^2}+\frac{13i}{24}\, V^2_{Z^2\oZ^2 U_1}\big) \,\omega^{z}
\endaligned
\]
\[
\aligned
&+\big(\frac{1}{24}\, V^1_{Z^4\oZ}+\frac{25}{24}\, V^1_{Z\oZ^4}-\frac{11}{6}\, V^1_{Z^3\oZ^2}-\frac{15}{8} \,V^1_{Z^2\oZ^3}-\frac{i}{3}\, V^2_{Z^3\oZ^2} V^3_{Z^2\oZ^2}-\frac{i}{24}\, V^2_{Z^2\oZ^2 U_1} V^2_{Z^3\oZ^2}
\\
&+\frac{1}{12}\, V^2_{Z^3\oZ^2} V^1_{Z^4\oZ}+\frac{1}{3}\, V^2_{Z^3 \oZ^2} V^1_{Z^3 \oZ^2}+\frac{5}{12}\, V^2_{Z^3 \oZ^2} V^1_{Z\oZ^4}-\frac{1}{12}\, V^2_{Z^3\oZ^2} V^2_{Z^4\oZ^2}+V^2_{Z^3\oZ^3}+\frac{1}{24} \,V^3_{Z^2\oZ^2 U_1}
\\
&-\frac{1}{24} \,V^2_{Z^4\oZ^2}-V^2_{Z^3\oZ^2} V^2_{Z^2\oZ^2}-\frac{1}{24} \,V^3_{Z^2\oZ^2 U_1} V^2_{Z^3\oZ^2}+\frac{13i}{24}\, V^2_{Z^2\oZ^2 U_1}\big) \,\omega^{\oz}
\\
&+\big(\frac{5}{6} \,V^1_{Z\oZ^3 U_j}-\frac{5}{6}\, V^1_{Z^3\oZ U_j}-\frac{15}{8}\, V^1_{Z^2 \oZ^2 U_j}+\frac{i}{2}\, V^2_{Z^3\oZ^2} V^3_{Z^2\oZ U_j}-\frac{i}{2}\, V^2_{Z^3\oZ^2} V^3_{Z\oZ^2 U_j}+\frac{1}{3}\, V^2_{Z^3\oZ^2} V^1_{Z\oZ^3 U_j}
\\
&+\frac{1}{3}\, V^2_{Z^3\oZ^2} V^1_{Z^3\oZ U_j}-\frac{1}{24} \,V^3_{Z\oZ^2 U_1 U_j}+\frac{1}{24}\, V^3_{Z^2 \oZ U_1 U_j}-V^2_{Z^3\oZ^2} V^2_{Z^2\oZ U_j}-\frac{1}{24}\, V^2_{Z^3\oZ^2} V^3_{Z\oZ^2 U_1 U_j}
\\
&+V^2_{Z^3\oZ^2 U_j}-\frac{1}{24}\, V^2_{Z^3\oZ^2} V^3_{Z^2\oZ U_1 U_j}\big)\, \omega^j-4i \, {\rm Re} \mu_{U_2}.
\endaligned
\]

Clearly, this equation brings the expression of the real Maurer-Cartan form ${\rm Re} \,\mu_{U_2}$. More generally, a close inspection of the performed computations reveals that

\begin{Lemma}
In Branch $2$ and for each $j\geq 0$, it is possible to specify the real Maurer-Cartan form ${\rm Re} \,\mu_{U_2^{j+1}}$ by normalizing ${\rm Im}\, V^2_{Z^3\oZ^2 U_2^j}=0$.
\end{Lemma}

As one sees and similar to Branch 1, all Maurer-Cartan forms have gained their normalized expressions at this order. Thus, we have succeeded again in constructing a complete moving frame and equivalently a complete normal form at this stage. It follows from this possibility of constructing this moving frame that the holomorphic transformation which brings each CR manifold to this normal form is unique. Summing up the results, we have

\begin{Theorem}
  \label{th-branch-2}
Let $M^5\subset\mathbb C^4$ be a five dimensional real-analytic totally nondegenerate CR manifold enjoying two assumptions $V^3_{Z^3\oZ}=0$ and $V^3_{Z^2\oZ U_1}\neq 0$. Then, there exists a unique origin-preserving holomorphic transformation that brings $M^5$ into the complete normal form:
\begin{equation}
\label{normal-form-branch-2}
\aligned
v^1&=z\oz+ \sum_{j+k+\sharp\ell\geq 5} \frac{1}{j!\,k!\,\ell!}\,V^1_{Z^j\oZ^k U^\ell} z^j \oz^k u^\ell,
\\
v^2&=\frac{1}{2}\,(z^2\oz+z\oz^2)+\sum_{j+k+\sharp\ell\geq 5} \frac{1}{j!\,k!\,\ell!}\,V^2_{Z^j\oZ^k U^\ell} z^j \oz^k u^\ell,
\\
v^3&=-\frac{i}{2}\,(z^2\oz-z\oz^2)+\frac{1}{2}\,(z^2\oz u_1+ z\oz^2 u_1)+\sum_{j+k+\sharp\ell\geq 5} \frac{1}{j!\,k!\,\ell!}\,V^3_{Z^j\oZ^k U^\ell} z^j \oz^k u^\ell,
\endaligned
\end{equation}
where, in addition to the equations \eqref{v-order-5-br-2}, we have the cross-section normalizations \eqref{normalization-partial} together with:
\begin{equation}
\label{normalization-branch-2}
\aligned
0={\rm Im} V^2_{Z^3\oZ^2 U_2^j}=V^3_{Z^2\oZ U_1 U_3^{j+1}}
\endaligned
\end{equation}
for $j\geq 0$. Moreover, in this branch, the Lie algebra $\frak{aut}_{CR}(M^5)$ is five dimensional and the biholomorphic equivalence problem to $M^5$ can be reduced to an absolute parallelism, namely $\{e\}$-structure, on itself with the structure equations of the horizontal coframe $\omega^z, \omega^{\oz}, \omega^j, j=1, 2, 3$ of $M^5$, obtained by \eqref{eq: order zero structure equations} after applying the already mentioned normalizations and inserting the achieved expressions of the normalized Maurer-Cartan forms.
\end{Theorem}

Representing the achieved normal form \eqref{normal-form-branch-2} as \eqref{NM-CM}, the following schematic diagrams describe partly the associated functions $\Phi^k_{ij}$, $k=1, 2, 3$, in terms of the applied normalizations \eqref{normalization-partial}-\eqref{normalization-branch-2}:
\begin{figure}[h]
    \centering
\parbox{1.4in}{%
\includegraphics[height=4cm, width=4 cm]{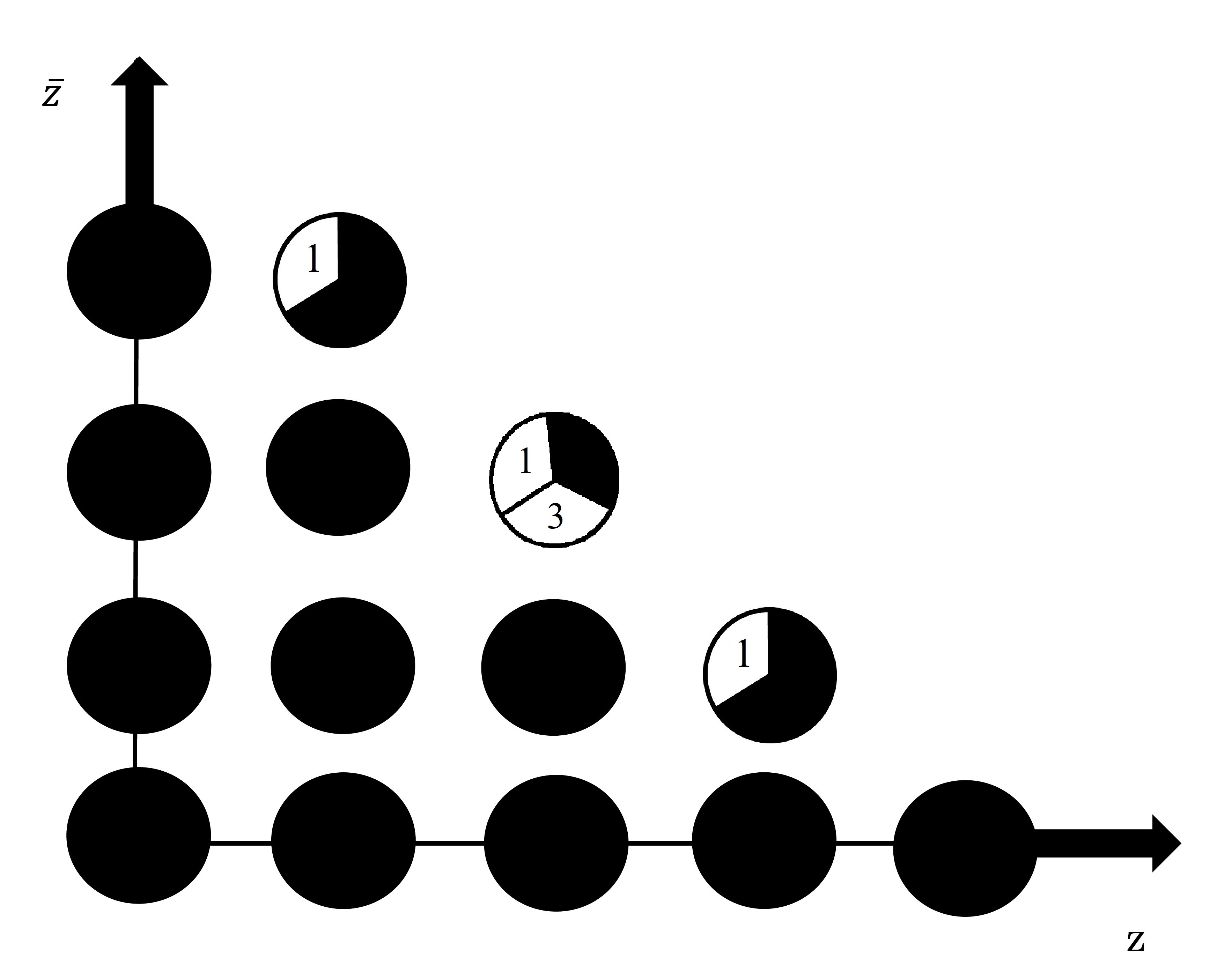}
\caption{$\Phi^1$}}
\qquad\qquad
\parbox{1.4in}{%
\includegraphics[height=4cm, width=4 cm]{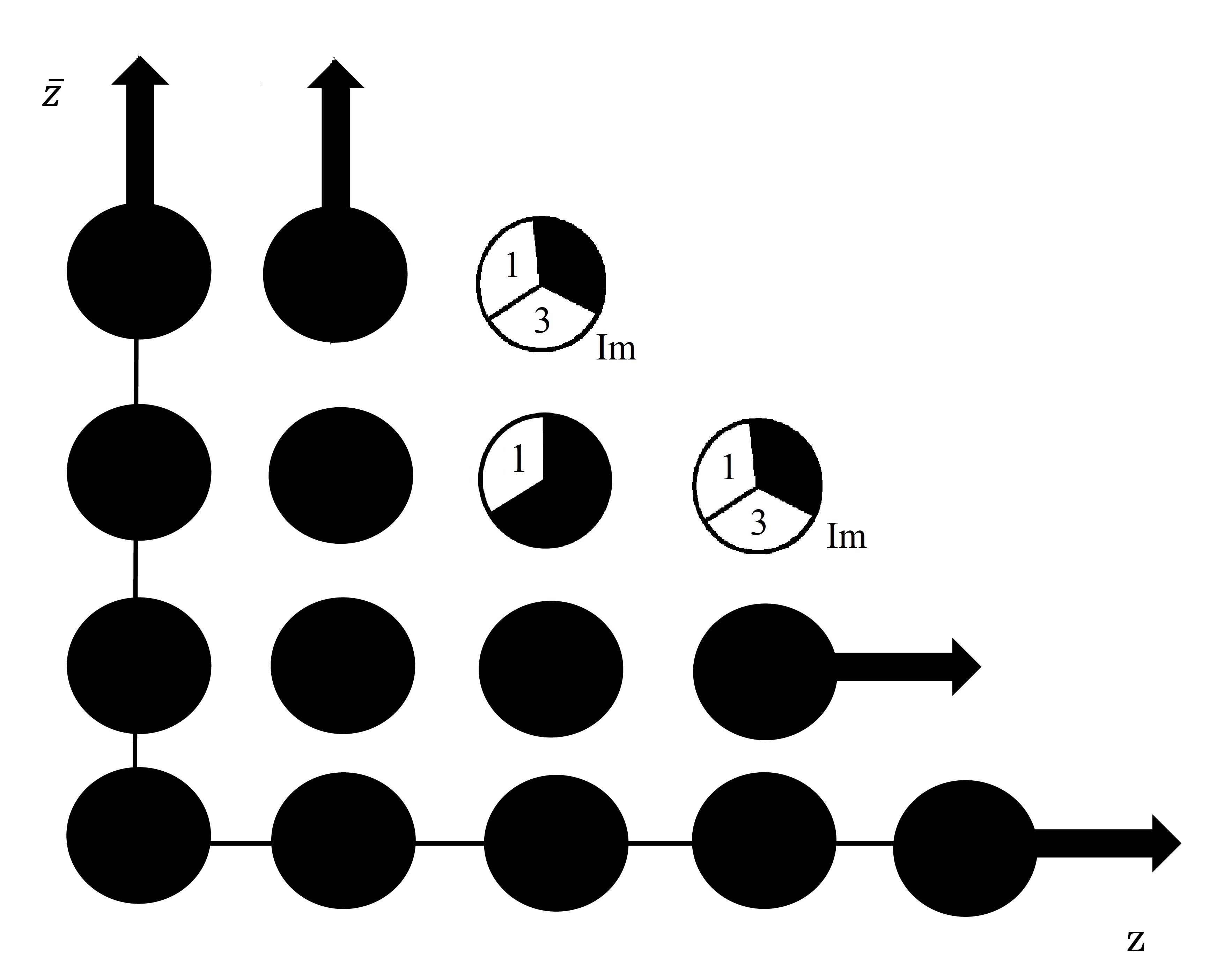}
\caption{ِ$\Phi^2$}}
\qquad\qquad
\begin{minipage}{1.4in}%
\includegraphics[height=4cm, width=4 cm]{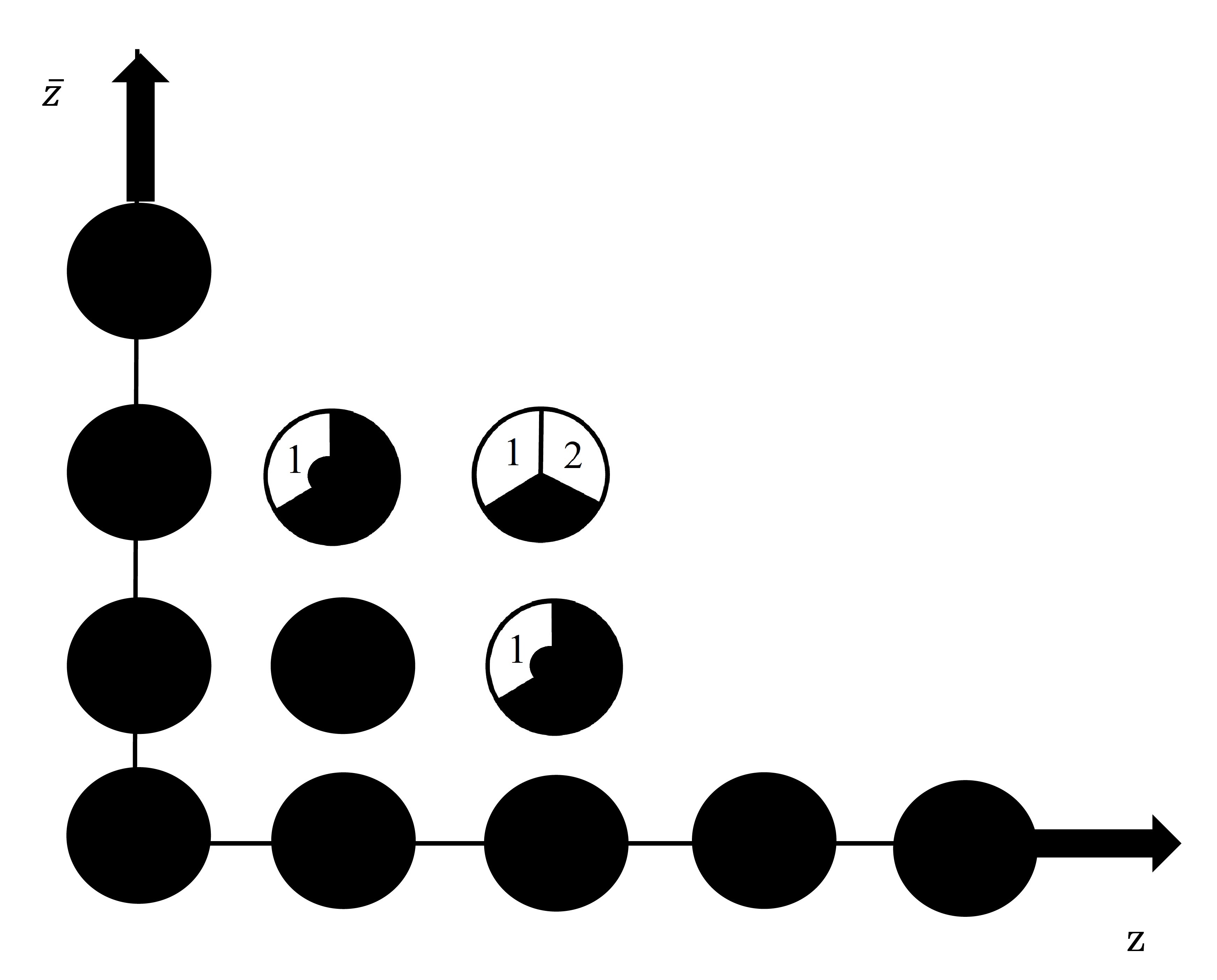}
\caption{$\Phi^3$}
\end{minipage}
\end{figure}

The interpretation of the circles at each point of these diagrams is as before, but it may be necessary to state that the circle standing at the point $(2,1)$ in the diagram of $\Phi^3$ amounts to the fact that $\Phi^3_{21}$ not only does not include the monomials $u_2^j u_3^l$ for $j, l\geq 0$ but also monomials of the form $u_1 u_3^l$s are absent in it.

Our cross-section \eqref{normalization-partial}-\eqref{normalization-branch-2} associated to this branch is minimal, as well. Thus, we have:

\begin{Corollary}
The collection of edge differential invariants of the cross-section \eqref{normalization-partial}-\eqref{normalization-branch-2} forms a generating system of differential invariants for the biholomorphic equivalence problem to $5$-dimensional totally nondegenerate CR manifolds $M^5$, belonging to Branch $2$.
\end{Corollary}

\section{Branch $3$: $V^3_{Z^3\oZ}=V^3_{Z^2\oZ U_1}=0$}
\label{Sec-Branch-3}

Now we arrive at Branch $3$, where the coefficients of $z^3\oz$ and $z^2\oz u_1$ in the partial moving frame \eqref{normalization-partial} are identically zero.
This branch includes Beloshapka's cubic model $M^5_{\tt c}$:
\begin{equation}
\label{model-cubic}
\aligned
v^1&=z\oz,
\\
v^2&=z^2\oz+z\oz^2,
\\
v^3&=i\,z^2\oz-i\,z\oz^2.
\endaligned
\end{equation}
The Lie algebra  $\frak{aut}_{CR}(M^5_{\sf c})$  of infinitesimal CR automorphisms of $M^5_{\sf c}$ is of the maximum possible dimension $7$ \cite[Proposition 3.2]{5-cubic} and since none of the CR manifolds $M^5$ in the former branches exhibits such CR symmetry dimension, then submanifolds equivalent to this model may possibly emerge in this branch.

Losing the benefit of the lifted differential invariants $V^3_{Z^3\oZ}$ and $V^3_{Z^2\oZ U_1}$ for normalizing further Maurer-Cartan forms in order four may convince oneself to expect more complicated computations at this branch. Before proceeding into the next order, we point it out that by equating to zero the coefficients of the linearly independent lifted horizontal forms $\omega^z, \omega^{\oz}, \omega^j, j=1, 2, 3,$ in the recurrence relations of these lifted differential invariants in  \eqref{VZ3Zbar}, we receive the following helpful equations:
\begin{equation}
\label{v-order-5}
\aligned
\ \ \ \ \ \ V^3_{Z^4\oZ}&=0, \qquad V^3_{Z^3\oZ^2}=-i V^2_{Z^3\oZ^2}, \qquad V^3_{Z^3\oZ U_j}=0,
\\
V^1_{Z^4\oZ}&=0, \qquad V^2_{Z^4\oZ^2}=4 V^1_{Z^3\oZ^2}-\frac{1}{2} V^3_{Z^2\oZ^2 U_1}-\frac{i}{2} V^2_{Z^2\oZ^2 U_1},
\\
V^3_{Z^2\oZ U_1 U_1}&=8 V^1_{Z^3\oZ U_1}, \qquad V^3_{Z^2\oZ U_1 U_2}=0, \qquad V^3_{Z^2\oZ U_1 U_3}=0.
\endaligned
\end{equation}

Recall that, up to the end of order four, the remained unnormalized Maurer-Cartan forms are those visible in \eqref{unnormalized-MC-prd-4}.

\subsection{Order 5 -- Branch 3}

In comparison to the former orders, computations in order five grow {\it explosively}. Then, for the sake of brevity, we will present from now on upcoming long expressions modulo the lifted horizontal coframe $\omega^z, \omega^{\oz}, \omega^j, j=1, 2, 3$. Nevertheless, the full expressions and intermediate computations are accessible in the {\sc Maple} worksheet \cite{Maple}.
Let us start with the recurrence relation of $V^1_{Z^3\oZ^2}$. After a large amount of simplification, we get:
\[
\aligned
dV^1_{Z^3\oZ^2}&=V^1_{Z^3\oZ^2} \big(2i\, \alpha^2_{U_3}-\frac{2}{3}\, \alpha^3_{U_3}\big)-8\, \alpha^3_{U_3U_3}+12i\, \alpha^2_{U_3 U_3}, \qquad {\rm mod}\ \ \omega^z, \omega^{\oz}, \omega^j.
\endaligned
\]
Thus, by the normalization $V^1_{Z^3\oZ^2}=0$, one receives that:
\[
\alpha^2_{U_3 U_3}\equiv 0, \qquad {\rm and} \qquad \alpha^3_{U_3 U_3}\equiv 0, \qquad {\rm mod}\ \ \omega^z, \omega^{\oz}, \omega^j.
\]

We continue with the recurrence formula of $dV^2_{Z^2\oZ^3}$. Our computations simplify to:
\[
dV^2_{Z^3\oZ^2}=-\frac{2}{3} \,V^2_{Z^3\oZ^2}\, \alpha^3_{U_3}-4i\, {\rm Re}\, \mu_{U_2}, \qquad {\rm mod}\ \ \omega^z, \omega^{\oz}, \omega^j.
\]
Hence,  by normalizing ${\rm Im} V^2_{Z^3\oZ^2}=0$, one may also specify:
\[
{\rm Re}\, \mu_{U_2}\equiv 0, \qquad {\rm mod}\ \ \omega^z, \omega^{\oz}, \omega^j.
\]

A close inspection of the performed computations for the above two recurrence relations shows that

\begin{Lemma}
\label{normalization-alpha3-U3-l+2}
Let $j\geq 0$. It is always possible
\begin{itemize}
  \item[$1.$] to specify each two Maurer-Cartan forms $\alpha^2_{U_3^{j+2}}$ and $\alpha^3_{U_3^{j+2}}$ by normalizing to zero the lifted differential invariant $V^1_{Z^3\oZ^2 U_3^j}$.
  \item[$2.$] to specify each real Maurer-Cartan form ${\rm Re}\,\mu_{U_2^{j+1}}$, by normalizing to zero the lifted differential invariant ${\rm Im} V^2_{Z^3\oZ^2 U_2^j}$.
\end{itemize}
\end{Lemma}

We notice that following the normalization ${\rm Im} \,V^2_{Z^3\oZ^2}=0$, the recurrence relation of this differential invariant changes into the form:

\begin{equation}
\label{dREV2Z2Zbar3}
d({\rm Re} \,V^2_{Z^3\oZ^2})=-\frac{2}{3}\, ({\rm Re}\, V^2_{Z^3\oZ^2}) \, \alpha^3_{U_3}, \qquad {\rm mod}\ \ \omega^z, \omega^{\oz}, \omega^j.
\end{equation}

Thanks to Lemma \ref{normalization-alpha3-U3-l+2}, the collection of yet unnormalized Maurer-Cartan forms \eqref{unnormalized-MC-prd-4} is now extensively reduced to only two Maurer-Cartan forms:

\begin{equation}
\label{unnormalized-MC-br2-ord5}
\alpha^2_{U_3} \qquad {\rm and} \qquad \alpha^3_{U_3}.
\end{equation}

At this stage, let us probe the recurrence relations corresponding to some of the lifted differential invariants appearing in \eqref{v-order-5}. Although these equations will not result in normalizing the remained Maurer-Cartan forms but they can reveal some helpful relations between lifted differential invariants. After a large amount of simplifications,  we receive that:

\begin{equation*}
\label{recu-rel-vanished-ord-5}
\aligned
0&=dV^3_{Z^4\oZ}=V^3_{Z^5\oZ}\, \omega^z+\big(V^3_{Z^4\oZ^2}+i\, V^2_{Z^4\oZ^2}\big)\, \omega^{\oz} +V^3_{Z^4\oZ U_j} \omega^j,
\\
0&=dV^3_{Z^3\oZ U_1}=V^3_{Z^4\oZ U_1} \omega^z+\big(V^3_{Z^3\oZ^2 U_1}+i\, V^2_{Z^3\oZ^2 U_1}-2i\, V^1_{Z^3\oZ U_1}\big)\, \omega^{\oz} +V^3_{Z^3\oZ U_1 U_j} \omega^j,
\\
0&=dV^3_{Z^3\oZ U_2}=V^3_{Z^4\oZ U_2} \omega^z+\big(V^3_{Z^3\oZ^2 U_2}+i \,V^2_{Z^3\oZ^2 U_2}\big)\, \omega^{\oz} +V^3_{Z^3\oZ U_2 U_j} \omega^j,
\\
0&=dV^3_{Z^3\oZ U_3}=V^3_{Z^4\oZ U_3} \omega^z+\big(V^3_{Z^3\oZ^2 U_3}+i\, V^2_{Z^3\oZ^2 U_3}\big)\, \omega^{\oz} +V^3_{Z^3\oZ U_3 U_j} \omega^j,
\\
0&=dV^1_{Z^4\oZ}= \frac{6}{5}\, V^1_{Z^5\oZ} \,\omega^z+\big(V^1_{Z^4\oZ^2}+\frac{1}{5} \,V^1_{Z^5\oZ}-\frac{1}{5} \,V^2_{Z^5\oZ^2}\big)\, \omega^{\oz}+V^1_{Z^4\oZ U_j} \omega^j-\frac{1}{5}\, V^3_{Z^5\oZ}\, \alpha^2_{U_3}.
\endaligned
\end{equation*}
Equating to zero the coefficients of the horizontal lifted 1-forms implies the following helpful relations:
\begin{equation}
\label{v-order-5-second-part}
\aligned
0= V^3_{Z^5\oZ}&=V^3_{Z^4\oZ U_j}=V^3_{Z^3\oZ U_k U_j}=V^1_{Z^5\oZ}=V^1_{Z^4\oZ U_j}
\\
V^3_{Z^4\oZ^2}&=-i V^2_{Z^4\oZ^2}, \qquad V^3_{Z^3\oZ^2 U_1}=2i V^1_{Z^3\oZ U_1}-i V^2_{Z^3\oZ^2U1},
\\
\ \ \ \ \ \ \ V^3_{Z^3\oZ^2 U_2}&=-i V^2_{Z^3\oZ^2 U_2}, \qquad V^3_{Z^3\oZ^2 U_3}=-i V^2_{Z^3\oZ^2 U_3}, \qquad V^2_{Z^5\oZ^2}=5 V^1_{Z^4\oZ^2}.
\endaligned
\end{equation}
 for $ j, k=1, 2, 3$.
Now, let us continue examining the yet unnormalized lifted invariants in this order. According to the observations introduced in the former orders, it remains to check the recurrence relations of the six lifted differential invariants:
\begin{equation}
\label{rec-rel-branch-3-order-5}
\aligned
dV^1_{Z^3\oZ U_1}&= \big(2i\, V^1_{Z^3\oZ U_1}-\frac{3}{2}\, V^3_{Z^2\oZ^2 U_2}\big)\, \alpha^2_{U_3}-\frac{4}{3}\, V^1_{Z^3\oZ U_1} \,\alpha^3_{U_3},
\\
dV^1_{Z^2\oZ^2 U_1}&=\big(4\, V^3_{Z^2\oZ^2 U_2}-\frac{8i}{3}\, V^1_{Z^3\oZ U_1}+\frac{8i}{3}\, V^1_{Z\oZ^3 U_1}\big)\, \alpha^2_{U_3}-\frac{4}{3}\, V^1_{Z^2\oZ^2U_1}\, \alpha^3_{U_3},
\\
dV^1_{Z^2\oZ^2 U_3}&= \big(4\, V^2_{Z^2 \oZ^3 U_2}+4 \,V^2_{Z^3\oZ^2 U_2}-4i\, V^2_{Z^3\oZ^2 U_3}+4i \,V^2_{Z^2\oZ^3 U_3}-4i\, V^3_{Z^3\oZ^2 U_2}+4i\, V^3_{Z^2\oZ^3 U_2}\big) \,\alpha^2_{U_3}
\\
&-\frac{5}{3}\, V^1_{Z^2\oZ^2 U_3}\, \alpha^3_{U_3},
\\
dV^2_{Z^2\oZ^2 U_1}&=V^3_{Z^2\oZ^2 U_1}\, \alpha^2_{U_3}-V^2_{Z^2\oZ^2U_1}\, \alpha^3_{U_3},
\\
dV^3_{Z^2\oZ^2 U_1}&=-V^2_{Z^2\oZ^2 U_1}\,\alpha^2_{U_3}-V^3_{Z^2\oZ^2 U_1}\,\alpha^3_{U_3},
\\
dV^3_{Z^2\oZ^2 U_2}&=-\frac{4}{3} \,V^3_{Z^2\oZ^2 U_2}\, \alpha^3_{U_3},
\endaligned
\end{equation}
presented modulo the lifted horizontal forms.
Together with \eqref{dREV2Z2Zbar3}, these expressions indicate the appearance of several new subbranches, relying upon the values of the involving lifted differential invariants.

\begin{Theorem}
\label{th-branch-3-order-5}
Let $M^5\subset\mathbb C^4$ be a five dimensional real-analytic totally nondegenerate CR manifold belonging to Branch $3$ which amounts to assume that it enjoys $V^3_{Z^3\oZ}=V^3_{Z^2\oZ U_1}=0$. Then, there exists some origin-preserving holomorphic transformation which brings $M^5$ into the normal form:
\begin{equation}
\label{normal-form-branch-3}
\aligned
v^1&=z\oz+ \frac{1}{6}\,V^1_{Z^3\oZ U_1} z^3\oz u_1+ \frac{1}{6}\,V^1_{Z\oZ^3 U_1} z\oz^3 u_1+\frac{1}{4}\,V^1_{Z^2\oZ^2 U_1} z^2\oz^2 u_1+\frac{1}{4}\,V^1_{Z^2\oZ^2 U_3} z^2\oz^2 u_3
\\
& \ \ \ \ \ \ \ \ \ +\sum_{j+k+\sharp\ell\geq 6} \frac{1}{j!\,k!\,\ell!}\,V^1_{Z^j\oZ^k U^\ell} z^j \oz^k u^\ell,
\\
v^2&=\frac{1}{2}\,(z^2\oz+z\oz^2)+\frac{1}{24}\,V^2_{Z^3\oZ^2} (z^3\oz^2+z^2\oz^3)+\frac{1}{4}\,V^2_{Z^2\oZ^2 U_1} z^2\oz^2 u_1+\sum_{j+k+\sharp\ell\geq 6} \frac{1}{j!\,k!\,\ell!}\,V^2_{Z^j\oZ^k U^\ell} z^j \oz^k u^\ell,
\\
v^3&=-\frac{i}{2}\,(z^2\oz-z\oz^2)+\frac{1}{24}\,V^3_{Z^3\oZ^2} (z^3\oz^2-z^2\oz^3)+\frac{1}{4}\,V^3_{Z^2\oZ^2 U_1} z^2\oz^2 u_1+\frac{1}{4}\,V^3_{Z^2\oZ^2 U_2} z^2\oz^2 u_2
\\
& \ \ \ \ \ \ \ \ +\frac{1}{2}\,V^3_{Z^2\oZ U_1 U_1} z^2\oz u_1 u_1+\frac{1}{2}\,V^3_{Z\oZ^2 U_1 U_1} z\oz^2 u_1 u_1
+\sum_{j+k+\sharp\ell\geq 6} \frac{1}{j!\,k!\,\ell!}\,V^2_{Z^j\oZ^k U^\ell} z^j \oz^k u^\ell,
\endaligned
\end{equation}
where in addition to the equations \eqref{v-order-5} and \eqref{v-order-5-second-part}, we have the cross-section normalizations \eqref{normalization-partial} together with:
\begin{equation}
\label{normalizations-branch-3}
\aligned
0=V^1_{Z^3\oZ^2 U_3^j}={\rm Im} V^2_{Z^3\oZ^2 U_2^j},
\endaligned
\end{equation}
for $j\geq 0$. Furthermore,
\begin{itemize}
\item[{\bf Branch 3-1.}] If at least one of the four lifted differential invariants:
\begin{equation}
\label{A-3-1}
V^2_{Z^2\oZ^2U^1}, \qquad V^3_{Z^2\oZ^2U^1}, \qquad  V^1_{Z^3\oZ U^1}, \qquad V^3_{Z^2\oZ^2 U^2},
\end{equation}
is nonzero then, by normalizing it to $1$, the corresponding recurrence relation brings $0\equiv\alpha^2_{U_3}=\alpha^3_{U_3}$, modulo $\omega^z, \omega^{\oz}, \omega^j, j=1, 2, 3$. In this case, the CR automorphism algebra $\frak{aut}_{CR}(M^5)$ is five dimensional and the biholomorphic equivalence problem to $M^5$ can be reduced to an
    absolute parallelism, namely $\{e\}$-structure, on itself with the structure equations of the lifted horizontal coframe, obtained by \eqref{eq: order zero structure equations} after applying the already mentioned normalizations and substituting the achieved expressions of the normalized Maurer-Cartan forms.
\item[{\bf Branch 3-2.}] Otherwise, if $V^2_{Z^2\oZ^2U_1}=V^3_{Z^2\oZ^2U_1}=V^1_{Z^3\oZ U_1}=V^3_{Z^2\oZ^2 U_2}= 0$ but at least one of the three real lifted differential invariants $V^1_{Z^2\oZ^2 U_1}, V^1_{Z^2\oZ^2 U_3}, {\rm Re} V^2_{Z^3\oZ^2}$  is nonzero then, by normalizing it to $1$, the real Maurer-Cartan form $\alpha^3_{U_3}$ can be specified.
\item[{\bf Branch 3-3.}] Otherwise, if all of the above mentioned seven differential invariants vanish identically, then none of the remained Maurer-Cartan forms $\alpha^2_{U_3}$ and $\alpha^3_{U_3}$ is normalizable in the current order five.
\end{itemize}
\end{Theorem}

\begin{proof}
If at least one of the invariants $V^2_{Z^2\oZ^2U^1}$, $V^3_{Z^2\oZ^2U_1}$, $V^1_{Z^3\oZ U_1}$ is nonzero then, by normalizing it to $1$, the corresponding first, fourth or five equation in \eqref{rec-rel-branch-3-order-5} yields $0\equiv\alpha^2_{U_3}=\alpha^3_{U_3}$, modulo the lifted horizontal coframe (notice that $V^3_{Z^2\oZ^2U_2}$ is real). Otherwise, if these invariants vanish but $V^3_{Z^2\oZ^2 U_1}\neq 0$, then by normalizing it to $1$, the first and last equations in \eqref{rec-rel-branch-3-order-5} imply the same result. In both of these cases, all the appearing Maurer-Cartan forms are normalized and hence the Lie algebra $\frak{aut}_{CR}(M^5)$ has the same dimension as the CR manifold $M^5$.

\noindent
Otherwise, if the above four mentioned differential invariants vanish, identically, but at least one of the real differential invariants ${\rm Re} V^2_{Z^3\oZ^2}, V^1_{Z^2\oZ^2U_1}, V^1_{Z^2\oZ^2U_3}$ is nonzero, then its corresponding recurrence relation in \eqref{dREV2Z2Zbar3} or \eqref{rec-rel-branch-3-order-5} results in specifying the Maurer-Cartan form $\alpha^3_{U_3}$.

\noindent
If all of these seven differential invariants vanish, then the recurrence relations \eqref{dREV2Z2Zbar3} and  \eqref{rec-rel-branch-3-order-5} are of no use to normalize any Maurer-Cartan form.
\end{proof}

Representing the achieved normal form \eqref{normal-form-branch-2} as \eqref{NM-CM} with the corresponding cross-section \eqref{normalization-partial}-\eqref{normalizations-branch-3}, the general schematic diagram of the normal forms in Branch 3 can be displayed as follows\,\,---\,\,disregarding the normalizations which may appear in the mentioned subbranches:
\begin{figure}[h]
   \centering
\parbox{1.4in}{%
\includegraphics[height=4cm, width=4 cm]{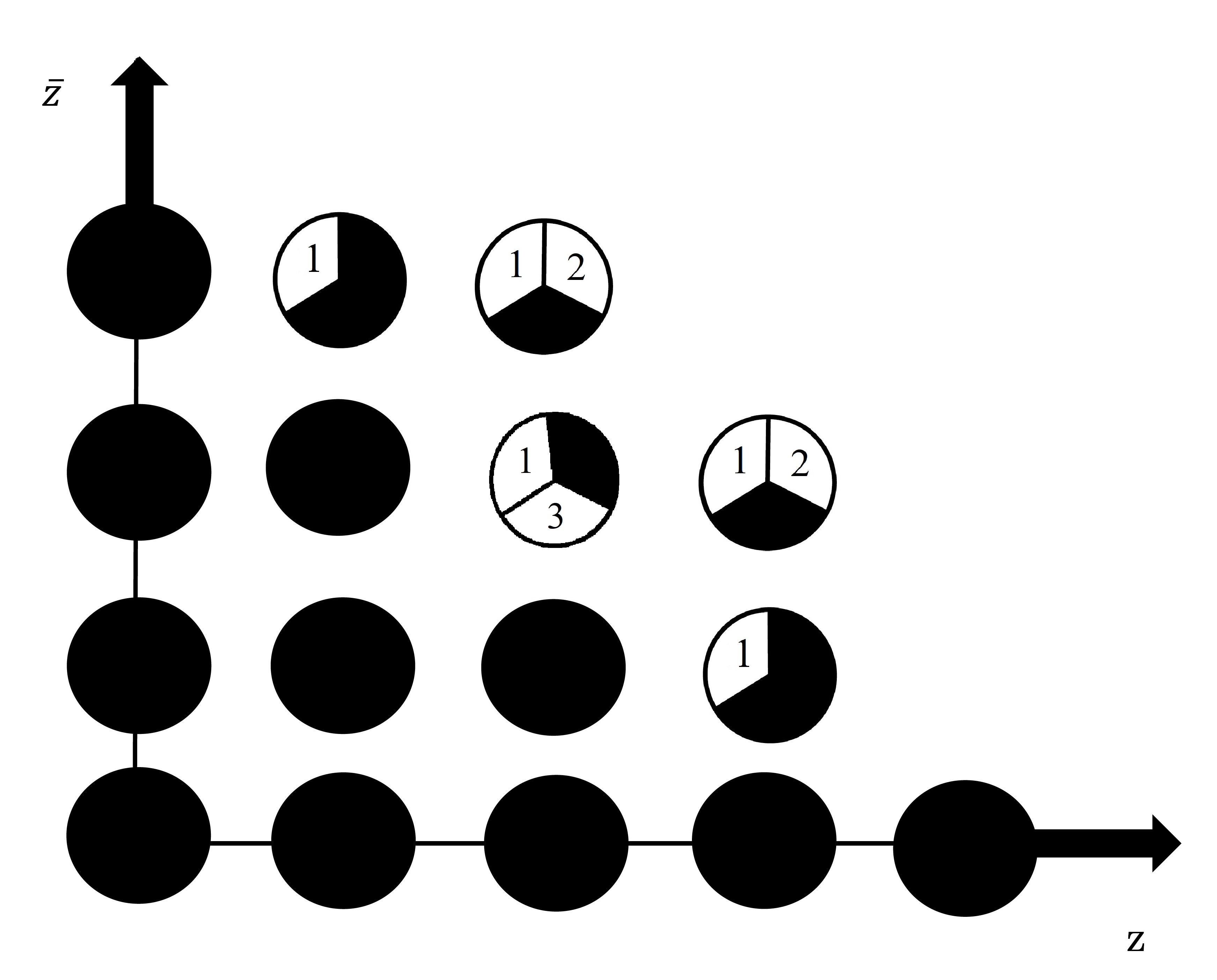}
\caption{$\Phi^1$}}
\qquad\qquad
\parbox{1.4in}{%
\includegraphics[height=4cm, width=4 cm]{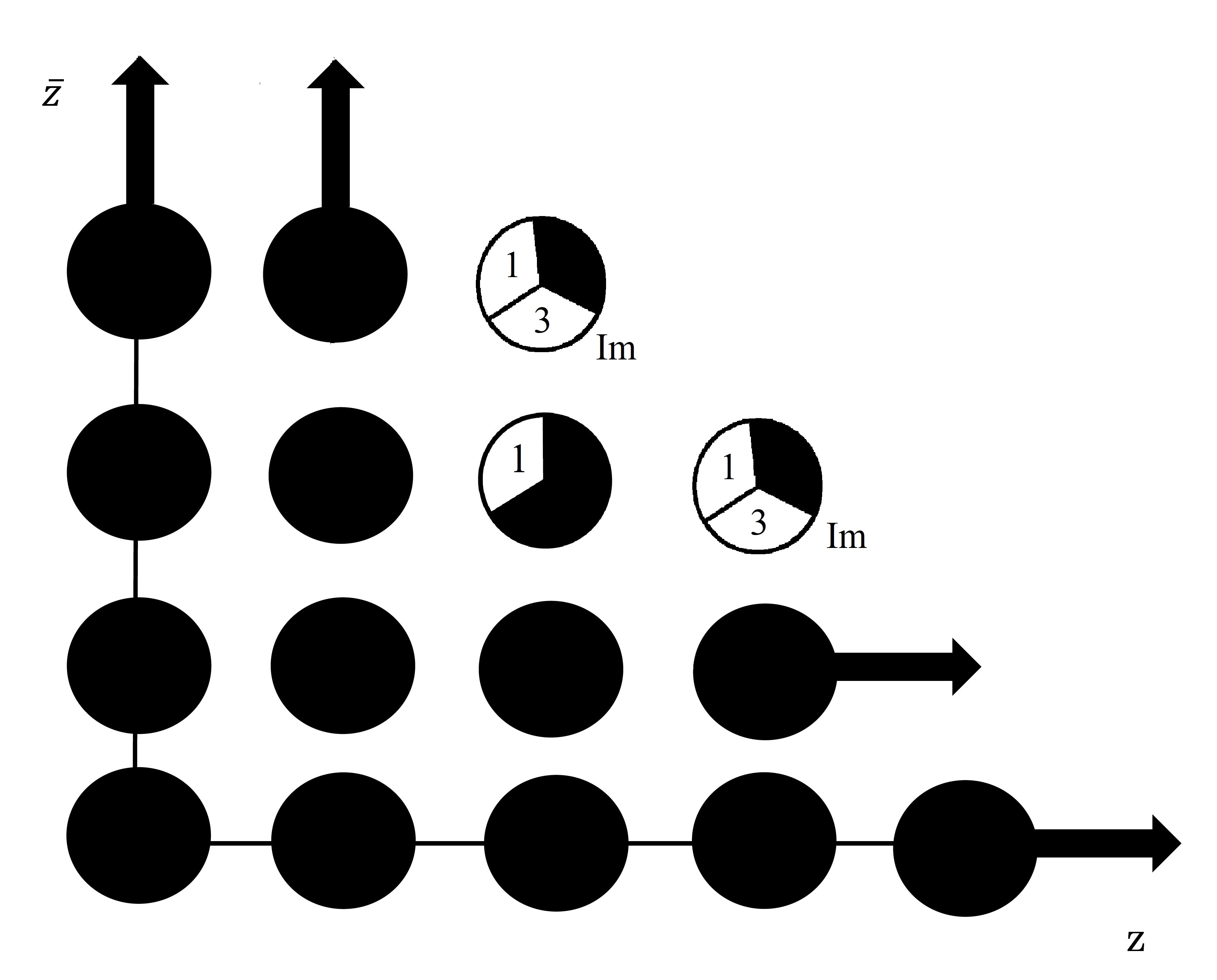}
\caption{ِ$\Phi^2$}}
\qquad\qquad
\begin{minipage}{1.4in}%
\includegraphics[height=4cm, width=4 cm]{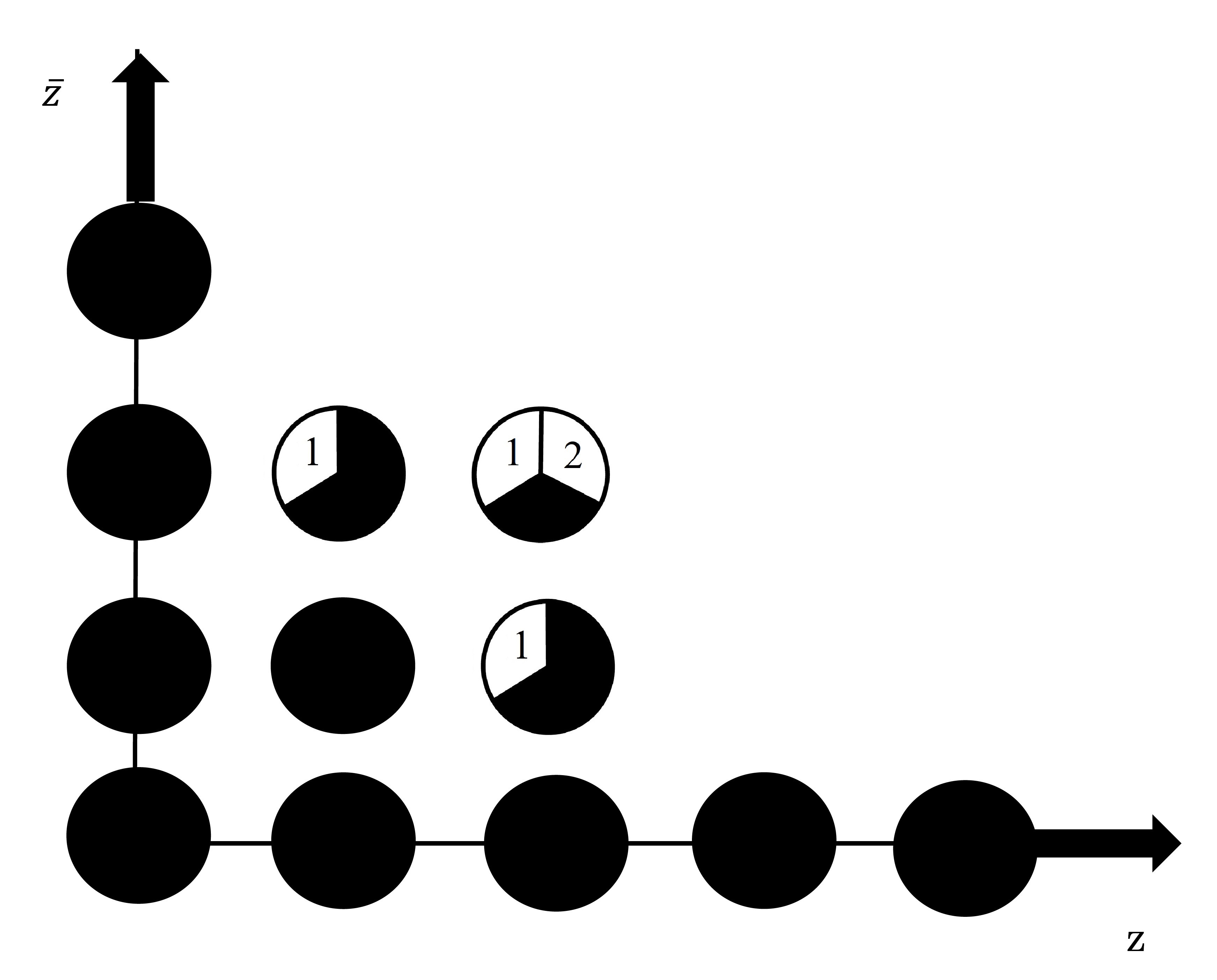}
\caption{$\Phi^3$}
\end{minipage}
\end{figure}

We notice that the normal form of Branch 3-1 is complete since in this subbranch, all the appeared Maurer-Cartan forms are normalized. It also follows, from the possibility of constructing a moving frame in this branch, that the transformation which brings each CR manifold $M^5$ of this branch to its normal form is {\it unique}.
For Branch 3-2, the complete normal form depends on the possibility of normalizing the last remained Maurer-Cartan form $\alpha^2_{U_3}$ in the next orders, however, at least we know that the infinitesimal CR automorphism algebras in this branch are either of dimension $5$ or $6$. Similarly, for a complete normal form in Branch 3-3 we need to proceed to the next orders. Since this last branch includes the model cubic $M^5_{\tt c}$ ({\it cf.} \eqref{model-cubic}), let us skip Branch 3-2 and continue with it.

\subsection{Order 6 - Branch 3-3}

Although the order five recurrence relations \eqref{dREV2Z2Zbar3} and \eqref{rec-rel-branch-3-order-5} are of no help  to normalize Maurer-Cartan forms in this branch but, however, equating to zero the coefficients of the independent lifted horizontal $1$-forms brings us some helpful relations. Because of their length, we did not present these coefficients in \eqref{rec-rel-branch-3-order-5} but they are available in the {\sc Maple} worksheet \cite{Maple}. Taking into account the expressions \eqref{v-order-5} and \eqref{v-order-5-second-part} and also the normalizations \eqref{normalizations-branch-3}, the arisen system results in the solution:
\begin{equation}
\label{eq-1}
\aligned
0&\equiv V^1_{Z^4\oZ^2}= V^1_{Z^2\oZ^2U_1U_1} =V^1_{Z^2\oZ^2 U_1U_2}=V^1_{Z^2\oZ^2U_1U_3}
\\
&=V^1_{Z^3\oZ U_1U_1}=V^1_{Z^3\oZ U_1U_2}=V^1_{Z^3\oZ U_1U_3}=V^1_{Z^3\oZ^2U_1}= V^1_{Z^3\oZ^2 U_2}
\\
&=V^2_{Z^2\oZ^2U_1U_1}=V^2_{Z^2\oZ^2U_1U_2}=V^2_{Z^2\oZ^2U_1U_3}=V^2_{Z^3\oZ^2U_1}=V^2_{Z^3\oZ^2U_2}=V^2_{Z^3\oZ^2U_3}
\\
&=V^3_{Z^2\oZ^2U_1U_1}=V^3_{Z^2\oZ^2U_1U_2}=V^3_{Z^2\oZ^2U_1U_3}=V^3_{Z^2\oZ^2U_2U_2}=V^3_{Z^2\oZ^2U_2U_3}
\\
&=V^3_{Z^2\oZ U_1U_1U_1}=V^3_{Z^2\oZ U_1U_1U_2}= V^3_{Z^2\oZ U_1U_1U_3}=V^3_{Z^2\oZ U_1U_2U_2}=V^3_{Z^2\oZ U_1U_2U_3}=V^3_{Z^2\oZ U_1U_3U_3}.
\endaligned
\end{equation}

Furthermore, the recurrence relation of order five differential invariant $V^3_{Z^3\oZ^2}$, vanished in this subbranch, can be simplified extensively by means of the above equations, \eqref{v-order-5} and \eqref{v-order-5-second-part} as:
\[
\aligned
0=dV^3_{Z^3\oZ^2}&=-\frac{i}{2}V^2_{Z^3\oZ^3} \omega^z+\big(V^3_{Z^3 \oZ^3}+\frac{i}{2} V^2_{Z^3\oZ^3}\big) \omega^{\oz}.
\endaligned
\]

Equating to zero the coefficients of $\omega^z$ and $\omega^{\oz}$ of this equation and also inserting \eqref{eq-1} in \eqref{v-order-5} and \eqref{v-order-5-second-part}, give in addition:
\begin{equation}
\label{eq-2}
\aligned
0&\equiv V^1_{Z^5\oZ}=V^1_{Z^4\oZ U_j}=V^2_{Z^4\oZ^2}= V^2_{Z^3\oZ^3}
\\
&=V^3_{Z^4\oZ^2}=V^3_{Z^3\oZ^3}=V^3_{Z^3\oZ^2 U_j}=V^3_{Z^5\oZ}&=V^3_{Z^4\oZ U_j}=V^3_{Z^3\oZ U_j U_k}, \qquad j, k=1,2,3.
\endaligned
\end{equation}

Accordingly and taking into accounts also the formerly provided normalizations \eqref{normalization-partial}-\eqref{normalizations-branch-3}, we shall consider in this order just the following 3 out of 130 possible recurrence relations:

\begin{equation}
\label{rec-rel-ord-6}
\aligned
dV^1_{Z^3\oZ^3}&=-\frac{4}{3}\,V^1_{Z^3\oZ^3} \,\alpha^3_{U_3},
\\
dV^1_{Z^2\oZ^2 U_2U_3}&=-\frac{1}{3}\,V^1_{Z^2\oZ^2 U_3U_3} \alpha^2_{U_3}-\frac{8}{9}\,V^1_{Z^2\oZ^2 U_2U_3} \alpha^3_{U_3}, \qquad\qquad {\rm mod} \ \omega^z, \omega^{\oz}, \omega^j.
\\
dV^1_{Z^2\oZ^2 U_3U_3}&=-\frac{2}{9}\,V^1_{Z^2\oZ^2 U_2U_3} \alpha^2_{U_3}-\frac{8}{3}\,V^1_{Z^2\oZ^2 U_3U_3} \alpha^3_{U_3}.
\endaligned
\end{equation}

Following these equations, we find ourselves forced to deal again with some new subbranches, depending upon vanishing$/$non-vanishing of the three lifted relative invariants $V^1_{Z^3\oZ^3}, V^1_{Z^2\oZ^2 U_2 U_3}$ and  $V^1_{Z^2\oZ^2 U_3 U_3}$. Indeed, we may imagine the following three subbranches at the heart of Branch 3-3:

\begin{itemize}
  \item[$\Box$] {\bf Branch 3-3-1:} $V^1_{Z^3\oZ^3}\neq 0$ and $V^1_{Z^2\oZ^2 U_2 U_3}=V^1_{Z^2\oZ^2 U_3 U_3}= 0$,
  \item[$\Box$] {\bf Branch 3-3-2:} $V^1_{Z^3\oZ^3}= 0$ and either $V^1_{Z^2\oZ^2 U_2 U_3}\neq 0$ or $V^1_{Z^2\oZ^2 U_3 U_3}\neq 0$.
  \item[$\Box$] {\bf Branch 3-3-3:} $V^1_{Z^3\oZ^3}=V^1_{Z^2\oZ^2 U_2 U_3}=V^1_{Z^2\oZ^2 U_3 U_3} =0$.
\end{itemize}

\subsubsection{{\bf Branch 3-3-1}}

In this case and by normalizing $V^1_{Z^3\oZ^3}=1$, the first recurrence relation in \eqref{rec-rel-ord-6} plainly specifies the Maurer-Cartan form $\alpha^3_{U_3}$. Our computations show that, after simplification, we have:
\[
\alpha^3_{U_3}=\frac{3}{4}\,V^1_{Z^4\oZ^3}\,\omega^z+\frac{3}{4}\,V^1_{Z^3\oZ^4}\,\omega^{\oz}+\frac{3}{4}\,V^1_{Z^3\oZ^3 U_j}\,\omega^j.
\]
Then, only one Maurer-Cartan form, namely $\alpha^2_{U_3}$, remains yet unnormalized here. In order to realize the possibility of its normalization, we have to proceed into order seven. Although similar to \eqref{eq-1}, we can find the value of many order seven lifted differential invariants by inspecting those appearing in \eqref{eq-1} and \eqref{eq-2} but, unfortunately, it needs to perform enormous computations for both finding the mentioned recurrence relations and solving the arisen system. In order to bypass such complication, we benefit Cartan's results in his classical approach to equivalence problems \cite{Olver-1995}. Our computations show that in this branch, we have:

\begin{equation}\label{MC-forms}
\aligned
\mu_Z=-i\,\alpha^2_{U_3}+\frac{1}{3}\,\alpha^3_{U_3}, & \qquad \mu_{U_j}=0, \ \ j=1, 2, 3,
\\
\alpha^1_{Z}=-i\,\omega^{\oz}, \qquad \alpha^2_Z=\alpha^3_Z=0, & \qquad
\alpha^1_{U_1}=\frac{2}{3} \, \alpha^3_{U_3},  \qquad \alpha^1_{U_2}=\alpha^1_{U_3}=0,
\\
\alpha^2_{U_1}=-(\omega^z+\omega^{\oz}), \qquad \alpha^2_{U_2}=\alpha^3_{U_3}, & \qquad \alpha^3_{U_1}=i\,(\omega^z-\omega^{\oz}),  \qquad \alpha^3_{U_2}=-\alpha^2_{U_3},
\endaligned
\end{equation}
with $\alpha^3_{U_3}$ as determined above. Inserting these expressions into the structure equations \eqref{eq: order zero structure equations} of the lifted horizontal coframe gives:

\begin{equation}
\label{struc-eq}
\aligned
d\omega^z&=-i\,\alpha^2_{U_3}\wedge\omega^z-\frac{1}{4}\,V^1_{Z^3\oZ^4}\,\omega^z\wedge\omega^{\oz}-\frac{1}{4}\,V^1_{Z^3\oZ^3 U_j}\,\omega^z\wedge\omega^j,
\\
d\omega^{\oz}&=\overline{d\omega^{z}},
\\
d\omega^1&=2i\,\omega^z\wedge\omega^{\oz}+\frac{1}{2}\,V^1_{Z^4\oZ^3}\,\omega^z\wedge\omega^{\oz}-\frac{1}{2}\,V^1_{Z^3\oZ^3 U_j}\,\omega^{\oz}\wedge\omega^j,
\\
d\omega^2&=\alpha^2_{U_3}\wedge\omega^3-\omega^z\wedge\omega^1-\omega^{\oz}\wedge\omega^1+\frac{3}{4}\,\big(V^1_{Z^4\oZ^3}\,\omega^z+V^1_{Z^3\oZ^4}\,\omega^{\oz}+V^1_{Z^3\oZ^3 U_j}\,\omega^j\big)\wedge\omega^2,
\\
d\omega^3&=-\alpha^2_{U_3}\wedge\omega^2+i\,\omega^z\wedge\omega^1-i\,\omega^{\oz}\wedge\omega^1+\frac{3}{4}\,\big(V^1_{Z^4\oZ^3}\,\omega^z+V^1_{Z^3\oZ^4}\,\omega^{\oz}+V^1_{Z^3\oZ^3 U_j}\,\omega^j\big)\wedge\omega^3.
\endaligned
\end{equation}
Thus, by the principles of Cartan's theory, the possibility of normalizing the last Maurer-Cartan form $\alpha^2_{U_3}$ relies upon the expressions of the four order seven lifted differential invariants $V^1_{Z^4\oZ^3}, V^1_{Z^3\oZ^3 U_j}$, $j=1, 2, 3$. Our much complicated computations show that, after normalizations, two lifted differential invariants $V^1_{Z^4\oZ^3}$ and $V^1_{Z^3\oZ^3U_1}$ are independent of the remained group parameters. Indeed we have:
\[
 dV^1_{Z^4\oZ^3}\equiv 0 \ \ \ {\rm and} \ \ \ dV^1_{Z^3\oZ^3 U_1}\equiv 0, \qquad {\rm mod} \ \omega^z, \omega^{\oz}, \omega^j.
\]
Hence, these two lifted differential invariants are essentially of no use to normalize the Maurer-Cartan form $\alpha^2_{U_3}$. Moreover, checking the recurrence relations of $V^1_{Z^3\oZ^2 U_2}$ and $V^1_{Z^3 \oZ^2 U_3}$ shows that:
\[
V^1_{Z^3\oZ^3 U_2}=V^1_{Z^3\oZ^3 U_3}.
\]
Accordingly, $V^1_{Z^3\oZ^3 U_2}$ is the only lifted differential invariant which may help one to normalized the remained Maurer-Cartan form in \eqref{struc-eq}. After a large amount of simplifications, we found that:

\[
dV^1_{Z^3\oZ^3 U_2}=V^1_{Z^3\oZ^3 U_2}\,\alpha^2_{U_3} \qquad {\rm mod} \ \omega^z, \omega^{\oz} , \omega^j.
\]
Consequently, the normalization of $\alpha^2_{U_3}$ depends upon vanishing$/$nonvanishing of the order seven lifted differential invariant $V^1_{Z^3\oZ^3 U_2}$. Thus we may imagine the following two subbranches of Branch 3-3-1:

\begin{itemize}
  \item[$\Box$] {\bf Branch 3-3-1-a:} $V^1_{Z^3\oZ^3U_2}\neq 0$,
  \item[$\Box$] {\bf Branch 3-3-1-b:} $V^1_{Z^3\oZ^3U_2}= 0$.
\end{itemize}

 In the first branch, the normalization $V^1_{Z^3\oZ^3 U_2}=1$ plainly specifies the Maurer-Cartan form $\alpha^2_{U_3}$. In this case, the  construction of a complete equivariant moving frame and hence a unique normal form is successfully finalized and the corresponding Lie algebras of infinitesimal CR automorphisms are $5$-dimensional. Moreover, according to the structure equations \eqref{struc-eq}, the equivalence problem to CR manifolds belonging to this branch is determined by the two lifted differential invariants $V^1_{Z^4\oZ^3}$ and $V^1_{Z^3\oZ^3 U_1}$.

 But in Branch 3-3-1-b, the above recurrence relation $dV^1_{Z^3\oZ^3 U_2}$ is of no use to specify $\alpha^2_{U_3}$. More precisely, in this case the structure equations \eqref{struc-eq} show that this remained Maurer-Cartan form is {\it never normalizable}. One can verify from these structure equations that the horizontal lifted coframe $\{\omega^z, \omega^{\oz}, \omega^1, \omega^2, \omega^3\}$ is {\it non-involutive}\footnote{Indeed, the reason is that the {\it degree of indeterminancy} of the system \eqref{struc-eq} is zero ({\it cf.} \cite[Definition 11.2]{Olver-1995}).} ({\it see} \cite[$\S$11]{Olver-1995} for definition).  Hence by \cite[Proposition 12.1]{Olver-1995}, the solution of the biholomorphic equivalence problem to $5$-dimensional CR manifolds $M^5$, belonging to this subbranch, completely relies upon the solution of the equivalence problem between $6$-dimensional {\it prolonged spaces} $M^{\sf pr}:=M^5\times \mathcal G^{\sf red}$ equipped with the coframe $\{\omega^z, \omega^{\oz}, \omega^1, \omega^2, \omega^3, \alpha^2_{U_3}\}$. Here, $\mathcal G^{\sf red}$ is the $1$-dimensional subgroup obtained by the normalizations applied to the original pseudo-group $\mathcal G$. According to the principles of Cartan's classical method, we only need to find the structure equations of the single added $1$-forms $\alpha^2_{U_3}$. By the help of the formula \eqref{struc-eq-MC}, we have:
\[
\aligned
d\alpha^2_{U_3}&=\omega^z\wedge\alpha^2_{ZU_3}+\omega^{\oz}\wedge\alpha^2_{\oZ U_3}+\omega^j\wedge\alpha^2_{U_jU_3}
+\alpha^2_{Z}\wedge\mu_{U_3}+\alpha^2_{\oZ}\wedge\overline\mu_{U_3}+\alpha^2_{U_j}\wedge\alpha^j_{U_3}+\alpha^2_{V^j}\wedge\gamma^j_{U_3}.
\endaligned
\]
Our computations show that in this branch:
\[
\aligned
0=\alpha^2_{ZU_3}&=\alpha^2_{U_1U_3}=\alpha^2_{U_3U_3}=\mu_{U_3}=\alpha^1_{U_3}=\alpha^2_{V^j}, \qquad \alpha^2_{U_2U_3}=\frac{1}{16}\,\big(\omega^z+\omega^{\oz}\big),
\\
\alpha^2_{U_2}&=\alpha^3_{U_3}=\frac{3}{4}\,V^1_{Z^4\oZ^3}\,\omega^z+\frac{3}{4}\,V^1_{Z^3\oZ^4}\,\omega^{\oz}+\frac{3}{4}\,V^1_{Z^3\oZ^3 U_1}\,\omega^1.
\endaligned
\]
 Thus, we have:

 \begin{equation}
 \label{dalpha2-U3}
  d\alpha^2_{U_3}=-\frac{1}{16}\,\big(\omega^z+\omega^{\oz}\big)\wedge\omega^2.
 \end{equation}

 This structure equation together with \eqref{struc-eq}\,\,---\,\,modified  by the branch assumptions $0=V^1_{Z^3\oZ^3U_2}=V^1_{Z^3\oZ^3U_3}$\,\,---\,\,determine the biholomorphic equivalence problem between CR manifolds $M^5$ of this subbranch. They also indicate that the Lie algebras of infinitesimal CR automorphisms corresponding to these manifolds are all of dimension six. Summing up the results, we have

\begin{Theorem}
\label{th-branch-3-3-1}
Let $M^5\subset\mathbb C^4$ be a five dimensional real-analytic totally nondegenerate CR manifold of Branch $3$-$3$-$1$, namely with the assumptions:
\begin{equation}
\label{A-3-3-1}
\aligned
0&=V^3_{Z^3\oZ}=V^3_{Z^2\oZ U_1},
\\
0&= V^2_{Z^2\oZ^2U_1}=V^3_{Z^2\oZ^2U_1}=V^1_{Z^3\oZ U_1}=V^3_{Z^2\oZ^2 U_2}=V^1_{Z^2\oZ^2 U_1}= V^1_{Z^2\oZ^2 U_3}={\rm Re}\, V^2_{Z^3\oZ^2},
\\
0&\neq V^1_{Z^3\oZ^3} \ \ \ and \ \ \ V^1_{Z^2\oZ^2 U_2 U_3}=V^1_{Z^2\oZ^2 U_3 U_3}=0.
\endaligned
\end{equation}
 Then, there exists some origin-preserving holomorphic transformation which brings $M^5$ into the normal form:
\begin{equation}
\label{normal-form-Branch-3-3-1}
\aligned
v^1&=z\oz+ \frac{1}{36}\,z^3\oz^3+\sum_{j+k+\sharp\ell\geq 7} \frac{1}{j!\,k!\,\ell!}\,V^1_{Z^j\oZ^k U^\ell} z^j \oz^k u^\ell,
\\
v^2&=\frac{1}{2}\,(z^2\oz+z\oz^2)+\sum_{j+k+\sharp\ell\geq 7} \frac{1}{j!\,k!\,\ell!}\,V^2_{Z^j\oZ^k U^\ell} z^j \oz^k u^\ell,
\\
v^3&=-\frac{i}{2}\,(z^2\oz-z\oz^2)+\sum_{j+k+\sharp\ell\geq 7} \frac{1}{j!\,k!\,\ell!}\,V^3_{Z^j\oZ^k U^\ell} z^j \oz^k u^\ell,
\endaligned
\end{equation}
with the cross-section normalizations \eqref{normalization-partial}-\eqref{normalizations-branch-3}. Moreover,
\begin{itemize}
  \item[{\bf Branch 3-3-1-a.}] If in addition we have $V^1_{Z^3\oZ^3 U_2}\neq 0$ then, by adding the normalization $V^1_{Z^3\oZ^3 U_2}=1$ to the already mentioned cross-section, the lastly remained Maurer-Cartan form $\alpha^2_{U_3}$ will be normalized and one constructs a complete equivariant moving frame. In this case, the already mentioned origin-preserving transformation is unique and the Lie algebra $\frak{aut}_{CR}(M^5)$ is $5$-dimensional. Furthermore, the biholomorphic equivalence problem to $M^5$ is reducible to an absolute parallelism, namely $\{e\}$-structure, on itself with the structure equations \eqref{struc-eq}, modified by $V^1_{Z^3\oZ^3 U_2}=V^1_{Z^3\oZ^3 U_3}=1$.
  \item[{\bf Branch 3-3-1-b.}] Otherwise, if $V^1_{Z^3\oZ^3 U_2}=0$ then, the  remained Maurer-Cartan form $\alpha^2_{U_3}$ is never normalizable. In this case, the Lie algebra $\frak{aut}_{CR}(M^5)$ is $6$-dimensional and the biholomorphic equivalence problem to $M^5$ is reducible to an absolute parallelism, namely $\{e\}$-structure, on the $6$-dimensional prolonged space $M^5\times\mathcal G^{\sf red}$ with the structure equations \eqref{struc-eq}-\eqref{dalpha2-U3}, modified by $V^1_{Z^3\oZ^3 U_2}=V^1_{Z^3\oZ^3 U_3}=0$.
\end{itemize}
\end{Theorem}

Remaining finally unnormalized  the Maurer-Cartan form $\alpha^2_{U_3}$ demonstrates that the construction of  a complete equivariant moving frame is impossible in Branch 3-3-1-b. It is for this reason ({\it cf.} \cite{Valiquette-SIGMA}) that the action of $\mathcal G$ on CR manifolds $M^5$ belonging to this branch is not free and it admits a $1$-dimensional {\it isotropy subgroup}, with the associated Maurer-Cartan coframe $\{\alpha^2_{U_3}\}$. However, as suggested in \cite{Valiquette-SIGMA}, our normalizations lead one to construct a {\it partial} moving frame on $M^5$. Fortunately, at this stage that all possible normalizations are applied in this branch and as stated at the page $18$ of \cite{Valiquette-SIGMA}, none of the lifted differential invariants will depend on the isotropy group parameter. Thus, the normal form $N$ of each CR manifold $M^5$ belonging to Branch 3-3-1-b is unique. If $\varphi, \psi: M^5\rightarrow N$ are two origin-preserving holomorphic transformations which bring $M^5$ into its normal form, then clearly the combination $\varphi\circ\psi^{-1}$ belongs to the isotropy group of $N$ at the origin ({\it see also} \cite[Theorem 4.11]{Valiquette-SIGMA}). Consequently

\begin{Corollary}
The origin-preserving holomorphic transformation which brings a $5$-dimensional totally nondegenerate CR manifold $M^5$ of Branch {\rm 3-3-1-b} into its normal form $N$ is unique up to the action of the $1$-dimensional isotropy group of $N$ at the origin.
\end{Corollary}

\subsubsection{{\bf Branch 3-3-2:}}

Now, let us assume that $V^1_{Z^4\oZ^2}=0$ but at least one of the lifted differential invariants $V^1_{Z^2\oZ^2 U_2 U_2}$ or $V^1_{Z^2\oZ^2 U_2 U_3}$ is nonzero. Disregarding the very specific case that $2V^1_{Z^2\oZ^2U_2U_3}=9 V^1_{Z^2\oZ^2U_3U_3}$, an slightly careful glance on the second and third recurrence relations in \eqref{rec-rel-ord-6} shows that by normalizing $V^1_{Z^2\oZ^2 U_2 U_2}, V^1_{Z^2\oZ^2 U_2 U_3}=1$ or $0$, depending to their vanishing$/$non-vanishing, it is possible in this branch to specify both the remained Maurer-Cartan forms $\alpha^2_{U_2}$ and $\alpha^3_{U_3}$ as:
\[
0\equiv\alpha^2_{U_3} \qquad {\rm and} \qquad 0\equiv\alpha^3_{U_3}, \qquad {\rm mod} \ \omega^z, \omega^{\oz}, \omega^j.
\]

Therefore, all the appearing Maurer-Cartan forms are normalizable here and we can construct successfully a complete moving frame and normal form.

\begin{Theorem}
\label{th-branch-3-3-2-a}
Let $M^5\subset\mathbb C^4$ be a five dimensional real-analytic totally nondegenerate CR manifold enjoying the branch assumptions:
\begin{equation}
\label{A-3-3-2-a}
\aligned
0&=V^3_{Z^3\oZ}=V^3_{Z^2\oZ U_1},
\\
0&= V^2_{Z^2\oZ^2U_1}=V^3_{Z^2\oZ^2U_1}=V^1_{Z^3\oZ U_1}=V^3_{Z^2\oZ^2 U_2}=V^1_{Z^2\oZ^2 U_1}= V^1_{Z^2\oZ^2 U_3}={\rm Re}\, V^2_{Z^3\oZ^2},
\\
0&= V^1_{Z^3\oZ^3} \qquad {\rm and} \qquad (0,0)\neq(V^1_{Z^2\oZ^2U_2U_3}, V^1_{Z^2\oZ^2U_3U_3})
\endaligned
\end{equation}
together with the minor assumption $2V^1_{Z^2\oZ^2U_2U_3}\neq 9 V^1_{Z^2\oZ^2U_3U_3}$. Then, there exists a unique origin-preserving holomorphic transformation which brings $M^5$ into the complete normal form:
\begin{equation}
\label{normal-form-Branch-3-3-2-a}
\aligned
v^1&=z\oz+ \frac{1}{4}\,V^1_{Z^2\oZ^2U_2U_3} z^2\oz^2u_2u_3+\frac{1}{4}\,V^1_{Z^2\oZ^2U_3U_3} z^2\oz^2u_3u_3+\sum_{j+k+\sharp\ell\geq 7} \frac{1}{j!\,k!\,\ell!}\,V^1_{Z^j\oZ^k U^\ell} z^j \oz^k u^\ell,
\\
v^2&=\frac{1}{2}\,(z^2\oz+z\oz^2)+\sum_{j+k+\sharp\ell\geq 7} \frac{1}{j!\,k!\,\ell!}\,V^2_{Z^j\oZ^k U^\ell} z^j \oz^k u^\ell,
\\
v^3&=-\frac{i}{2}\,(z^2\oz-z\oz^2)+\sum_{j+k+\sharp\ell\geq 7} \frac{1}{j!\,k!\,\ell!}\,V^3_{Z^j\oZ^k U^\ell} z^j \oz^k u^\ell,
\endaligned
\end{equation}
with the cross-section normalizations \eqref{normalization-partial}-\eqref{normalizations-branch-3} added by either $V^1_{Z^2\oZ^2U_2U_3}, V^1_{Z^2\oZ^2U_3U_3}=1$ or $0$, depending upon their vanishing$/$non-vanishing. In this branch, the infinitesimal CR automorphism algebra $\frak{aut}_{CR}(M^5)$ is five dimensional and the biholomorphic equivalence problem to $M^5$ is reducible to an absolute parallelism, namely $\{e\}$-structure, on itself with the structure equations of the lifted horizontal coframe, obtained by \eqref{eq: order zero structure equations} after applying the already mentioned cross-section normalizations and inserting the achieved expressions of the normalized Maurer-Cartan forms.
\end{Theorem}

\subsubsection{{\bf Branch 3-3-3}}

By the assumptions of this branch, all order six differential invariants are vanished identically and none of the remained Maurer-Cartan forms $\alpha^2_{U_3}$ and $\alpha^3_{U_3}$ is normalizable in this order. In this case, the Maurer-Cartan forms appeared in \eqref{MC-forms} have the same expressions with the only difference that here $\alpha^3_{U_3}$ is not specified. Inserting these expressions into the structure equations \eqref{eq: order zero structure equations} of the lifted horizontal coframe gives:

\begin{equation}
\label{struc-eq-branch-3-3-1}
\aligned
d\omega^z&=\big(\frac{1}{3}\,\alpha^3_{U_3}-i\alpha^2_{U_3}\big)\wedge\omega^z,
\\
d\omega^{\oz}&=\big(\frac{1}{3}\,\alpha^3_{U_3}+i\alpha^2_{U_3}\big)\wedge\omega^{\oz},
\\
d\omega^1&=2i\,\omega^z\wedge\omega^{\oz}+\frac{2}{3}\,\alpha^3_{U_3}\wedge\omega^{\oz},
\\
d\omega^2&=-\omega^z\wedge\omega^1-\omega^{\oz}\wedge\omega^1+\alpha^3_{U_3}\wedge\omega^2+\alpha^2_{U_3}\wedge\omega^3,
\\
d\omega^3&=i\,\omega^z\wedge\omega^1-i\,\omega^{\oz}\wedge\omega^1-\alpha^2_{U_3}\wedge\omega^2+\alpha^3_{U_3}\wedge\omega^3.
\endaligned
\end{equation}

These structure equations are of {\it constant type} and hence it will be impossible, even in higher orders, to normalize two remained Maurer-Cartan forms. Moreover, they plainly imply that the horizontal lifted coframe $\{\omega^z, \omega^{\oz}, \omega^1, \omega^2, \omega^3\}$ is non-involutive and hence the solution of biholomorphic equivalence problem to $5$-dimensional CR manifolds $M^5$, belonging to this subbranch, completely relies upon the solution of equivalence problem between $7$-dimensional {\it prolonged spaces} $M^{\sf pr}:=M^5\times \mathcal G^{\sf red}$ equipped with the coframe $\{\omega^z, \omega^{\oz}, \omega^1, \omega^2, \omega^3, \alpha^2_{U_3}, \alpha^3_{U_3}\}$. Here, $\mathcal G^{\sf red}$ is the $2$-dimensional subgroup obtained by the normalizations applied to the original pseudo-group $\mathcal G$ . Thus, we need only to find the structure equations of the two added $1$-forms $\alpha^2_{U_3}$ and $\alpha^3_{U_3}$. In this branch, our computations show that:
\[
0=\alpha^k_{ZU_3}=\alpha^k_{\oZ U_3}=\alpha^k_{U_jU_3}=\mu_{U_3}=\alpha^1_{U_3}=\alpha^k_{V^j}, \qquad \alpha^2_{U_2}=\alpha^3_{U_3}, \qquad \alpha^3_{U_2}=-\alpha^2_{U_3}.
\]
Then, by the help of the formula \eqref{struc-eq-MC}, one plainly receives that:
\begin{equation}
\label{struc-eq-MC-3-3-1}
d\alpha^2_{U_3}=0, \qquad \qquad d\alpha^3_{U_3}=0.
\end{equation}

These two structure equations together with \eqref{struc-eq-branch-3-3-1} provide the final desired $\{e\}$-{\it structure} of biholomorphic equivalence problem to 5-dimensional totally nondegenerate CR manifolds, belonging to this branch. Since these seven structure equations are of constant type, all CR manifolds in this branch are {\it biholomorphically equivalent}. Since the cubic model $M^5_{\tt c}$ defined as \eqref{model-cubic} belongs to this branch thus, we may assert that

\begin{Theorem}
\label{th-branch-3-3-2-b}
Let $M^5\subset\mathbb C^4$ be a five dimensional real-analytic totally nondegenerate CR manifold  enjoying the branch assumptions:
\begin{equation}\label{umbilical}
\aligned
0&=V^3_{Z^3\oZ}=V^3_{Z^2\oZ U_1},
\\
0&=V^2_{Z^2\oZ^2U_1}=V^3_{Z^2\oZ^2U_1}=V^1_{Z^3\oZ U_1}=V^3_{Z^2\oZ^2 U_2}=V^1_{Z^2\oZ^2 U_1}= V^1_{Z^2\oZ^2 U_3}={\rm Re}\, V^2_{Z^3\oZ^2}
\\
0&= V^1_{Z^3\oZ^3}=V^1_{Z^2\oZ^2 U_2 U_3}=V^1_{Z^2\oZ^2 U_3 U_3}.
\endaligned
\end{equation}
Then, $M^5$ is biholomorphically equivalent to the model cubic $M^5_{\tt c}$ and its defining equations can be converted into the simple normal form:
\begin{equation}
\label{model-cubic-2}
\aligned
v^1&=z\oz,
\\
v^2&=z^2\oz+z\oz^2,
\\
v^3&=i\,(z^2\oz-z\oz^2).
\endaligned
\end{equation}
In this case, the infinitesimal CR automorphism algebra $\frak{aut}_{CR}(M^5)$ is of the maximum possible dimension $7$ with the structure, displayed in the table on page $3231$ of \cite{5-cubic}.
\end{Theorem}

\begin{Remark}
The structure equations \eqref{struc-eq-branch-3-3-1}-\eqref{struc-eq-MC-3-3-1} obtained in this branch are fully equivalent to the structure equations \cite[eq. (36)]{5-cubic}. The correspondence is given by:
\[
\zeta\leftrightarrow\omega^z, \qquad \rho\leftrightarrow\frac{1}{2}\,\omega^1, \qquad \sigma\leftrightarrow\frac{1}{4}\,(\omega^2+i\,\omega^3), \qquad \alpha\leftrightarrow\frac{1}{3}\,\alpha^3_{U_3}-i\,\alpha^2_{U_3}.
\]
\end{Remark}

As the construction of a complete moving frame was here impossible, the origin-preserving holomorphic transformations which bring CR manifolds of this branch into the simple normal form \eqref{model-cubic-2} are no longer unique. Indeed, applying each CR automorphism of the $2$-dimensional isotropy group ${\sf Aut}_{0}(M^5_{\sf c})$ of the cubic model at the origin, to the normal form \eqref{model-cubic-2} keeps it still in normal form. This isotropy group is found in  \cite[Proposition 3.2]{5-cubic} as the group of local flows of holomorphic vector fields generated by:
\begin{equation*}
\aligned
D&=z\,\frac{\partial}{\partial z}+2w^1\,\frac{\partial}{\partial w^1}+3w^2\,\frac{\partial}{\partial w^2}+3w^3\,\frac{\partial}{\partial w^3},
\\
R&=iz\,\frac{\partial}{\partial z}-w^3\,\frac{\partial}{\partial w^2}+w^2\,\frac{\partial}{\partial w^3}.
\endaligned
\end{equation*}
If $\varphi, \psi: M^5\rightarrow M^5_{\sf c}$ are two origin-preserving holomorphic transformations which bring a CR manifold $M^5$ to the normal form $M^5_{\sf c}$, then clearly there exists a unique transformation $\gamma\in{\sf Aut}_{0}(M^5_{\sf c})$ satisfying $\varphi=\gamma\circ\psi$. Thus, we have

\begin{Corollary}
The origin-preserving holomorphic transformation which brings a $5$-dimensional totally nondegenerate CR manifold $M^5$ of Branch {\rm 3-3-3} into the normal form \eqref{model-cubic-2} is unique, up to the action of the isotropy subgroup ${\sf Aut}_{0}(M^5_{\sf c})$, generated by the above holomorphic vector fields $D$ and $R$.
\end{Corollary}

The above Theorem \ref{th-branch-3-3-2-b} has still another interesting consequence. Indeed, among all branches appeared in this paper, only the present branch 3-3-3 consists of totally nondegenerate CR manifolds $M^5$ with ${\rm dim}\,\frak{aut}_{CR}(M^5)=7$. This characterizes such CR manifolds plainly in terms of their infinitesimal CR automorphisms.

\begin{Corollary}
\label{cor}
A $5$-dimensional real-analytic totally nondegenerate CR manifold $M^5$ is biholomorphically equivalent to the cubic model $M^5_{\sf c}$ if and only if ${\rm dim}\,\frak{aut}_{CR}(M^5)=7$. In other words, up to the biholomorphic equivalence and in the class of totally nondegenerate CR manifolds considered in this paper, $M^5_{\sf c}$ is the unique member with the maximum dimension of the associated algebra of infinitesimal CR automorphisms.
\end{Corollary}

The results achieved in this branch may remind one of the Chern-Moser discussion of {\it umbilical points}.  In \cite{Chern-Moser}, a point $p$ of a nondegenerate real hypersurface of $\mathbb C^2$ is called umbilical if the coefficient $c_{42}$ of the monomial $z^4\oz$  in the associated normal form \cite[eq. (3.18)]{Chern-Moser} vanishes. It is well-known that locally, a hypersurface is umbilical at each point if and only if it is {\it spherical}, that is: biholomorphically equivalent to the {\it Heisenberg sphere} $\mathbb H^3$, represented in local coordinates $z, w=u+iv$ as:
\[
v=z\oz.
\]
 Following this terminology and thanks to Corollary \ref{cor}, we may regard the assumptions \eqref{umbilical} as {\it umbilical conditions} in which satisfying them by the CR manifold $M^5$ guarantees its biholomorphic equivalence to the cubic model $M^5_{\sf c}$.

\subsection*{Acknowledgment} The author gratefully expresses his sincere thanks to Peter Olver and Francis Valiquette for their helpful comments and discussions during the preparation of this paper. He is also very grateful to Maria Stepanova who noted via a counterexample the existence of a certain error concerning the results of Branch 3-3-1 in the first version of this paper. The research of the author was supported in part by a grant from IPM, No. 98510420.

\end{document}